\documentclass[letter,10pt]{amsart}
\usepackage{amssymb,amsmath,amsfonts}
\usepackage{mathrsfs}
\usepackage[left=1 in, right=1 in,top=1 in, bottom=1 in]{geometry}
\usepackage{mathenv}

\newtheorem{theorem}{Theorem}[section]
\newtheorem{lemma}[theorem]{Lemma}

\newtheorem{remark}[theorem]{Remark}

\makeatletter
\renewcommand \theequation {%
\ifnum \c@section>\z@ \@arabic\c@section.%
\fi\@arabic\c@equation} \@addtoreset{equation}{section}
\makeatother

\providecommand{\ud}[1]{\mathrm{d}{#1}}
\providecommand{\abs}[1]{\left\vert#1\right\vert}
\providecommand{\nm}[1]{\left\Vert#1\right\Vert}

\providecommand{\br}[1]{\left\langle #1 \right\rangle}

\providecommand{\tm}[1]{\left\Vert#1\right\Vert_{L^2}}
\providecommand{\im}[1]{\left\Vert#1\right\Vert_{L^{\infty}}}
\providecommand{\ts}[1]{\left\vert#1\right\vert_{L^2}}
\providecommand{\is}[1]{\left\vert#1\right\vert_{L^{\infty}}}
\providecommand{\tss}[2]{\left\vert#1\right\vert_{L^2_{#2}}}
\providecommand{\iss}[2]{\left\vert#1\right\vert_{L^{\infty}_{#2}}}
\providecommand{\um}[1]{\left\Vert#1\right\Vert_{L^2_{\nu}}}

\providecommand{\lnm}[2]{\left\Vert#1\right\Vert_{L^{\infty}_{#2}}}
\providecommand{\tnm}[1]{\left\Vert#1\right\Vert_{L^2}}
\providecommand{\lnnm}[2]{{\left\Vert\left\vert#1\right\vert\right\Vert}_{L^{\infty}L^{\infty}_{#2}}}
\providecommand{\tnnm}[1]{{\left\Vert\left\vert#1\right\vert\right\Vert}_{L^2L^2}}
\providecommand{\ltnm}[2]{{\left\Vert#1\right\Vert}_{L^{\infty}L^2_{#2}}}

\def\p{\partial}
\def\ls{\lesssim}

\def\half{\dfrac{1}{2}}
\def\rt{\rightarrow}
\def\r{\mathbb{R}}
\def\no{\nonumber}
\def\ue{\mathrm{e}}

\def\id{{\bf{1}}}

\def\f{\mathcal{F}^{\e}}
\def\fb{\mathscr{F}^{\e}}
\def\fc{\mathcal{F}}
\def\fbc{\mathscr{F}}
\def\w{\mathscr{W}^{\e}}
\def\wc{\mathscr{W}}

\def\e{\epsilon}
\def\s{\mathcal{S}}

\def\vvu{\vec{\mathfrak{u}}}

\def\vx{\vec x}
\def\v{v}
\def\vv{\vec v}
\def\vvv{\vec{\mathfrak{v}}}
\def\u{u}
\def\uh{\vec u}
\def\th{\theta}
\def\rh{\rho}
\def\vo{\vec\omega}
\def\ll{\mathcal{L}}

\def\vn{\vec n}
\def\nx{\nabla_x}
\def\dx{\Delta_x}
\def\b{b}
\def\bb{B}
\def\m{\mu}
\def\rr{r}
\def\ph{\phi}

\def\vr{\vec\v_{r}}
\def\vp{\v_{\phi}}
\def\ve{\v_{\eta}}
\def\et{\eta}
\def\g{g}
\def\q{\mathcal{Q}}
\def\qb{\mathscr{Q}}
\def\lll{\mathscr{L}}
\def\k{\kappa}
\def\ta{\vec\tau}

\def\l{\lambda}
\def\vt{\vartheta}
\def\ze{\zeta}
\def\d{\delta}
\def\pk{{\mathbb{P}}}
\def\ik{{\mathbb{I}}}
\def\bv{\br{\vv}^{\vt}\ue^{\ze\abs{\vv}^2}}
\def\bvv{\br{\vvv}^{\vt}\ue^{\ze\abs{\vvv}^2}}
\def\vb{\vec b}
\def\vh{w}
\def\tvh{\tilde{w}}

\def\gl{g^L}
\def\wl{w^L}
\def\ql{q^L}
\def\nk{\mathcal{N}}
\def\wi{w}
\def\qi{q}
\def\a{\mathcal{A}}
\def\t{\mathcal{T}}
\def\gg{\mathcal{G}}
\def\sn{\sqrt{\nu}}

\begin{document}


\title{Hydrodynamic Limit with Geometric Correction of
Stationary Boltzmann Equation}
\author{Lei Wu}
\address{
Department of Mathematical Sciences\\
Carnegie Mellon University \\
Wean Hall 6113, 5000 Forbes Avenue, Pittsburgh, PA 15213, USA } \email[L.
Wu]{lwu2@andrew,cmu.edu}

\begin{abstract}
We consider the hydrodynamic limit of a stationary Boltzmann equation in a unit plate with in-flow boundary. We prove the solution can be approximated in $L^{\infty}$ by the sum of interior solution which satisfies steady incompressible Navier-Stokes-Fourier system, and boundary layer with geometric correction. Also, we construct a counterexample to the classical theory which states the behavior of solution near boundary can be described by the Knudsen layer derived from the Milne problem.
\end{abstract}

\maketitle

\tableofcontents

\newpage


\pagestyle{myheadings} \thispagestyle{plain} \markboth{LEI WU}{GEOMETRIC CORRECTION FOR HYDRODYNAMIC LIMIT OF BOLTZMANN
EQUATION}

\section{Introduction}

\subsection{Problem Formulation}

We consider stationary Boltzmann equation for $F^{\e}(\vx,\vv)$ in a two-dimensional unit plate $\Omega=\{\vx=(x_1,x_2):\ \abs{\vx}\leq 1\}$
with velocity $\Sigma=\{\vv=(v_1,v_2)\in\r^2\}$ as
\begin{eqnarray}\label{large system}
\left\{
\begin{array}{rcl}
\e\vv\cdot\nx F^{\e}&=&Q[F^{\e},F^{\e}]\ \ \text{in}\ \
\Omega\times\r^2,\\\rule{0ex}{1.0em} F^{\e}(\vx_0,\vv)&=&\bb^{\e}
(\vx_0,\vv) \ \ \text{for}\ \ \vx_0\in\p\Omega\ \ \text{and}\ \ \vn(\vx_0)\cdot\vv<0,
\end{array}
\right.
\end{eqnarray}
where
$\vn(\vx_0)$ is the outward normal vector at $\vx_0$ and the
Knudsen number $\e$ satisfies $0<\e<<1$. Here we have
\begin{eqnarray}
Q[F,G]&=&\int_{\r^2}\int_{\s^1}q(\vo,\abs{\vvu-\vv})\bigg(F(\vvu_{\ast})G(\vv_{\ast})-F(\vvu)G(\vv)\bigg)\ud{\vo}\ud{\vvu},
\end{eqnarray}
with
\begin{eqnarray}
\vvu_{\ast}=\vvu+\vo\bigg((\vv-\vvu)\cdot\vo\bigg),\qquad \vv_{\ast}=\vv-\vo\bigg((\vv-\vvu)\cdot\vo\bigg),
\end{eqnarray}
and the hard-sphere collision kernel
\begin{eqnarray}
q(\vo,\abs{\vvu-\vv})=q_0\abs{\vvu-\vv}\abs{\cos\phi},
\end{eqnarray}
for positive constant $q_0$ related to the size of ball, $\vo\cdot(\vv-\vvu)=\abs{\vv-\vvu}\cos\phi$ and $0\leq\phi\leq\pi/2$.
We intend to study the behavior of $F^{\e}$ as $\e\rt0$.\\
\ \\
Based on the flow direction, we can divide the boundary
$\gamma=\{(\vx_0,\vv):\ \vx_0\in\p\Omega\}$ into the in-flow boundary
$\gamma_-$, the out-flow boundary $\gamma_+$, and the grazing set
$\gamma_0$ as
\begin{eqnarray}
\gamma_{-}&=&\{(\vx_0,\vv):\ \vx_0\in\p\Omega,\ \vv\cdot\vec n(\vx_0)<0\},\\
\gamma_{+}&=&\{(\vx_0,\vv):\ \vx_0\in\p\Omega,\ \vv\cdot\vec n(\vx_0)>0\},\\
\gamma_{0}&=&\{(\vx_0,\vv):\ \vx_0\in\p\Omega,\ \vv\cdot\vec n(\vx_0)=0\}.
\end{eqnarray}
It is easy to see $\gamma=\gamma_+\cup\gamma_-\cup\gamma_0$. Hence,
the boundary condition is only given on $\gamma_{-}$.\\
\ \\
We assume that the boundary data can be expanded as
\begin{eqnarray}
\bb^{\e}(\vx_0,\vv)&=&\m+\sqrt{\m}\b^{\e}(\vx_0,\vv)=\m+\sqrt{\m}\bigg(\sum_{k=1}^{\infty}\e^k\b_{k}(\vx_0,\vv)\bigg),
\end{eqnarray}
where the standard Maxwellian is defined as
\begin{eqnarray}
\m(\vv)=\frac{1}{2\pi}\exp\left(-\frac{\abs{\vv}^2}{2}\right),
\end{eqnarray}
and $\b_k(\vx_0,\vv)$ is independent of $\e$.
We further require
\begin{eqnarray}\label{smallness assumption}
\is{\bv\frac{B^{\e}-\m}{\sqrt{\m}}}\leq C_0\e,
\end{eqnarray}
for $0\leq\ze\leq1/4$ and $\vt\geq3$, where $C_0>0$ is sufficiently small.
The solution $F^{\e}$ can be expressed as a perturbation of the standard Maxwellian
\begin{eqnarray}
F^{\e}(\vx,\vv)&=&\m+\sqrt{\m}f^{\e}(\vx,\vv).
\end{eqnarray}
Then $f^{\e}$ satisfies the equation
\begin{eqnarray}\label{small system}
\left\{
\begin{array}{rcl}
\e\vv\cdot\nx
f^{\e}+\ll[f^{\e}]&=&\Gamma[f^{\e},f^{\e}],\\\rule{0ex}{1.0em}
f^{\e}(\vx_0,\vv)&=&\b^{\e}(\vx_0,\vv) \ \ \text{for}\ \
\vn\cdot\vv<0\ \ \text{and}\ \ \vx_0\in\p\Omega,
\end{array}
\right.
\end{eqnarray}
where
\begin{eqnarray}
\Gamma[f^{\e},f^{\e}]&=&\frac{1}{\sqrt{\m}}Q[\sqrt{\m}f^{\e},\sqrt{\m}f^{\e}],\\
\ll[f^{\e}]&=&-\frac{2}{\sqrt{\m}}Q[\m,\sqrt{\m}f^{\e}]=\nu(\vv)
f^{\e}-K[f^{\e}],\\\
\nu(\vv)&=&\int_{\r^2}\int_{\s^1}q(\vv-\vvu,\vo)\m(\vvu)\ud{\vo}\ud{\vvu}\\
K[f^{\e}](\vv)&=&K_2[f^{\e}](\vv)-K_1[f^{\e}](\vv)\\
&=&\int_{\r^2}k(\vvu,\vv)f^{\e}(\vvu)\ud{\vvu}=\int_{\r^2}k_2(\vvu,\vv)f^{\e}(\vvu)\ud{\vvu}-\int_{\r^2}k_1(\vvu,\vv)f^{\e}(\vvu)\ud{\vvu},\no\\
K_1[f^{\e}](\vv)&=&\sqrt{\m(\vv)}\int_{\r^2}\int_{\s^1}q(\vv-\vvu,\vo)\sqrt{\m(\vvu)}f^{\e}(\vvu)\ud{\vo}\ud{\vvu},\\
K_2[f^{\e}](\vv)&=&\int_{\r^2}\int_{\s^1}q(\vv-\vvu,\vo)\sqrt{\m(\vvu)}\bigg(\sqrt{\m(\vv_{\ast})}f^{\e}(\vvu_{\ast})
+\sqrt{\m(\vvu_{\ast})}f^{\e}(\vv_{\ast})\bigg)\ud{\vo}\ud{\vvu}.
\end{eqnarray}

\subsection{History}

Hydrodynamic limit is one of the main steps to tackle Hilbert's Six Problems, i.e. to derive macroscopic physical rules from elementary atomic laws, especially, Newton's law in the framework of classical mechanics. In particular, it is well known the fluid-type equations can be formally derived from Boltzmann equation by applying Hilbert expansion to obtain the leading order term. A lot of works of evolutionary Boltzmann equation in the whole space domain have been presented (see \cite{Golse.Saint-Raymond2004}, \cite{Masi.Esposito.Lebowitz1989}, \cite{Bardos.Golse.Levermore1991}, \cite{Bardos.Golse.Levermore1993}, \cite{Bardos.Golse.Levermore1998}, \cite{Bardos.Golse.Levermore2000}) for either smooth solutions or renormalized solutions.

Unfortunately, despite its importance in both theory and practice, much less results are known for steady problem in a bounded domain, as \cite{Golse2005} pointed out. In \cite{Esposito.Guo.Kim.Marra2013} and \cite{Esposito.Guo.Kim.Marra2015}, the hydrodynamic limit of stationary Boltzmann equation with diffusive boundary were studied, both of which do not involve boundary layer approximation. However, for general boundary conditions, the presence of boundary layer effects is inevitable. It is noticeable that a lot of works, even for evolutionary Boltzmann equation in a bounded domain, cite the results from the classical paper \cite{Bensoussan.Lions.Papanicolaou1979} and take boundary layer analysis as being completely solved (see \cite{Sone2002}, \cite{Sone2007}).  Surprisingly, in \cite{AA003}, it was shown the boundary layer approximation based on Milne problem as in \cite{Bensoussan.Lions.Papanicolaou1979} breaks down both in the proof and result, due to intrinsic singularity of normal derivative. Therefore, any results on kinetic equations involving boundary layer effects from \cite{Bensoussan.Lions.Papanicolaou1979} should be reexamined.

In this paper, we will give a complete asymptotic analysis of stationary Boltzmann equations with in-flow boundary and construct a boundary layer with geometric correction, which has been proved to be effective in neutron transport equation by \cite{AA003}.

\subsection{Main Theorem}

\begin{theorem}\label{main 1}
For given $\bb^{\e}>0$ satisfying (\ref{smallness assumption}) and $0<\e<<1$, there exists a unique positive
solution $F^{\e}=\m+\sqrt{\m}f^{\e}$ to the stationary Boltzmann equation (\ref{large system}), where
\begin{eqnarray}
f^{\e}=\e^3 R_N+\bigg(\sum_{k=1}^{N}\e^{k}\f_k\bigg)+\bigg(\sum_{k=1}^{N}\e^{k}\fb_k\bigg),
\end{eqnarray}
for $N\geq3$, $R_N$ satisfies (\ref{remainder}), $\f_k$ and $\fb_k$ satisfy (\ref{at 14}) and (\ref{at 11}). Also, there exists a $C>0$ such that $f^{\e}$ satisfies
\begin{eqnarray}
\im{\bv f^{\e}}\leq C\e,
\end{eqnarray}
for any $\vartheta>2$, $0\leq\zeta\leq1/4$.
\end{theorem}
\begin{theorem}\label{main 2}
For given $\bb^{\e}>0$ satisfying (\ref{smallness assumption}) with
\begin{eqnarray}
\frac{\b_1}{\sqrt{\m}}=\bigg(\vp\ue^{-(\vp^2-1)-M\ve^2}\bigg)=h(\ve,\vp),
\end{eqnarray}
where $\ve$ and $\vp$ are defined as in (\ref{substitution 3}) and $M$ is sufficiently large such that
\begin{eqnarray}
h(0,1)&=&1,\\
\tss{h}{-}&<<&1,
\end{eqnarray}
there exists $C>0$ such that
\begin{eqnarray}
\im{f^{\e}-(\fc_1+\fbc_1)}\geq C\e,
\end{eqnarray}
where the interior solution $\fc_1$ is defined in (\ref{at 14.}) and boundary layer $\fbc_1$ is defined in (\ref{at 11.}).
\end{theorem}

\subsection{Notation and Structure of This Paper}

Throughout this paper, $C>0$ denotes a constant that only depends on
the parameter $\Omega$, but does not depend on the data. It is
referred as universal and can change from one inequality to another.
When we write $C(z)$, it means a certain positive constant depending
on the quantity $z$. We write $a\ls b$ to denote $a\leq Cb$.

Our paper is organized as follows: in Section 2, we establish the $L^{\infty}$ well-posedness of
the linearized Boltzmann equation; in Section 3, we present the asymptotic analysis of the equation (\ref{small system});  in Section 4, we prove the well-posedness and decay of the $\e$-Milne problem with geometric correction; in Section 5, we prove Theorem \ref{main 1}; finally, in
Section 6, we show the classical approach and prove Theorem \ref{main 2}.

\section{Linearized Stationary Boltzmann Equation}

We consider the linearized stationary Boltzmann equation
\begin{eqnarray}\label{linear steady}
\left\{
\begin{array}{rcl}
\e\vv\cdot\nx f+\ll[f]&=&S(\vx,\vv)\ \ \text{in}\ \ \Omega,\\
f(\vx_0,\vv)&=&h(\vx_0,\vv)\ \ \text{for}\ \ \vx_0\in\p\Omega\ \
\text{and}\ \ \vv\cdot\vn<0.
\end{array}
\right.
\end{eqnarray}
Let $\br{\cdot,\cdot}$ be the standard $L^2$ inner product in
$\Omega\times\r^2$. We define the $L^p$ and $L^{\infty}$ norms in
$\Omega\times\r^2$ as usual:
\begin{eqnarray}
\nm{f}_{L^p}&=&\bigg(\int_{\Omega}\int_{\r^2}\abs{f(\vx,\vv)}^p\ud{\vv}\ud{\vx}\bigg)^{1/p},\\
\nm{f}_{L^{\infty}}&=&\sup_{(\vx,\vv)\in\Omega\times\r^2}\abs{f(\vx,\vv)}.
\end{eqnarray}
Define $\ud{\gamma}=\abs{\vv\cdot\vn}\ud{\varpi}\ud{\vv}$ on the
boundary $\p\Omega\times\r^2$ for $\varpi$ as the surface measure. Define the $L^p$ and
$L^{\infty}$ norms on the boundary as follows:
\begin{eqnarray}
\abs{f}_{L^p}&=&\bigg(\iint_{\gamma}\abs{f(\vx,\vv)}^p\ud{\gamma}\bigg)^{1/p},\\
\abs{f}_{L^p_{\pm}}&=&\bigg(\iint_{\gamma_{\pm}}\abs{f(\vx,\vv)}^p\ud{\gamma}\bigg)^{1/p},\\
\abs{f}_{L^{\infty}}&=&\sup_{(\vx,\vv)\in\gamma}\abs{f(\vx,\vv)},\\
\abs{f}_{L^{\infty}_{\pm}}&=&\sup_{(\vx,\vv)\in\gamma_{\pm}}\abs{f(\vx,\vv)}.
\end{eqnarray}
Also, we define
\begin{eqnarray}
\um{f}=\tm{\sqrt{\nu}f}.
\end{eqnarray}
Denote the Japanese bracket as
\begin{eqnarray}
\br{\vv}=\sqrt{1+\abs{\vv}^2}
\end{eqnarray}
Define the kernel operator $\pk$ as
\begin{eqnarray}
\pk[f]=\sqrt{\m}\bigg(a_f(t,\vx)+\vv\cdot
\vb_f(t,\vx)+\frac{\abs{\vv}^2-2}{2}c_f(t,\vx)\bigg),
\end{eqnarray}
and the non-kernel operator $\ik-\pk$ as
\begin{eqnarray}
(\ik-\pk)[f]=f-\pk[f].
\end{eqnarray}
with
\begin{eqnarray}
\int_{\r^2}(\ik-\pk)[f]\left(\begin{array}{c}1\\\vv\\\abs{\vv}^2\end{array}\right)\ud{\vv}=0
\end{eqnarray}
Our analysis is based on the ideas in \cite{Esposito.Guo.Kim.Marra2013, Guo2010}.

\subsection{Preliminaries}

\begin{lemma}(Green's Identity)\label{wellposedness prelim lemma 2}
Assume $f(\vx,\vv),\ g(\vx,\vv)\in L^2(\Omega\times\r^2)$ and
$\vv\cdot\nx f,\ \vv\cdot\nx g\in L^2(\Omega\times\r^2)$ with $f,\
g\in L^2(\gamma)$. Then
\begin{eqnarray}
\iint_{\Omega\times\r^2}\bigg((\vv\cdot\nx f)g+(\vv\cdot\nx
g)f\bigg)\ud{\vx}\ud{\vv}=\int_{\gamma_+}fg\ud{\gamma}-\int_{\gamma_-}fg\ud{\gamma}.
\end{eqnarray}
\end{lemma}
\begin{proof}
See the proof of \cite[Lemma 2.2]{Esposito.Guo.Kim.Marra2013}.
\end{proof}
\begin{lemma}\label{wellposedness prelim lemma 3}
For any $\l>0$, there exists a unique solution $f_{\l}(\vx,\vv)\in
L^{\infty}(\Omega\times\r^2)$ to the penalized transport equation
\begin{eqnarray}\label{wellposedness prelim penalty transport equation}
\left\{
\begin{array}{rcl}
\lambda f_{\l}+\epsilon\vv\cdot\nx f_{\l}&=&S(\vx,\vv)\ \ \text{in}\
\ \Omega,\\\rule{0ex}{1.0em} f_{\l}(\vx_0,\vv)&=&h(\vx_0,\vv)\ \
\text{for}\ \ \vx_0\in\p\Omega\ \ \text{and}\ \ \vv\cdot\vn<0,
\end{array}
\right.
\end{eqnarray}
such that
\begin{eqnarray}\label{wt 01}
\im{\bv f_{\l}}+\iss{\bv f_{\l}}{+}
&\leq& C\left( \frac{1}{\l}\im{\bv S}+\iss{\bv h}{-}\right),
\end{eqnarray}
for all $\vt\geq0$, $\ze\geq0$, and
\begin{eqnarray}\label{wt 06}
\tm{f_{\l}}^2+\frac{\e}{\l}\tss{f_{\l}}{+}^2\leq C\left(
\frac{1}{\l^2}\tm{ S}^2+\frac{\e}{\l}\tss{h}{-}^2\right).
\end{eqnarray}
\end{lemma}
\begin{proof}
The characteristics $(X(s), V(s))$ of the equation
(\ref{wellposedness prelim penalty transport equation}) which goes
through $(\vx,\vv)$ is defined by
\begin{eqnarray}\label{character}
\left\{
\begin{array}{rcl}
(X(0), V(0))&=&(\vx,\vv)\\\rule{0ex}{2.0em}
\dfrac{\ud{X(s)}}{\ud{s}}&=&\e V(s),\\\rule{0ex}{2.0em}
\dfrac{\ud{V(s)}}{\ud{s}}&=&0.
\end{array}
\right.
\end{eqnarray}
which implies
\begin{eqnarray}
\left\{
\begin{array}{rcl}
X(s)&=&\vx+\e s\vv\\
V(s)&=&\vv
\end{array}
\right.
\end{eqnarray}
Define the backward exit time $t_{\b}(\vx,\vv)$ and backward exit
position $\vx_{\b}(\vx,\vv)$ as
\begin{eqnarray}
t_{\b}(\vx,\vv)&=&\inf\{t>0: \vx-\e t\vv\notin\Omega\},\label{exit time}\\
\vx_{\b}(\vx,\vv)&=&\vx-\e t_{\b}(\vx,\vv)\vv\notin\Omega.
\end{eqnarray}
Hence, we can rewrite the equation (\ref{wellposedness prelim
penalty transport equation}) along the characteristics as
\begin{eqnarray}
\bv f_{\l}(\vx,\vv)
&=& \bv h(\vx_{\b},\vv)e^{-\l
t_{\b}}+\int_{0}^{t_{\b}}\bv S(\vx_{\b}+\e s\vv,\vv)e^{-\l
(t_{\b}-s)}\ud{s}.
\end{eqnarray}
Then we can naturally estimate
\begin{eqnarray}
\im{\bv f_{\l}}&\leq&e^{-\l t_{\b}}\iss{\bv h}{-}+\frac{1-e^{-\l t_{\b}}}{\l}\im{\bv S}\\
&\leq&\iss{\bv h}{-}+\frac{1}{\l}\im{\bv S},\no
\end{eqnarray}
which further implies
\begin{eqnarray}
\iss{\bv f_{\l}}{+}&\leq&\iss{\bv h}{-}+\frac{1}{\l}\im{\bv S}.
\end{eqnarray}
Since $f_{\l}$ can be explicitly traced back to the boundary data,
the existence naturally follows from above estimate. The uniqueness
and $L^2$ estimates follow from Green's identity and
$\tm{f_{\l}}\leq C\im{\bv f_{\l}}$.
\end{proof}

\subsection{$L^2$ Estimates of Linearized Stationary Boltzmann Equation}

\begin{lemma}\label{wellposedness prelim lemma 4}
For any $\l>0$, $m>0$, there exists a
unique solution $f_{\l,m}(\vx,\vv)\in L^{2}(\Omega\times\r^2)$ to
the equation
\begin{eqnarray}\label{wellposedness prelim penalty equation}
\left\{
\begin{array}{rcl}
\l f_{\l,m}+\e\vv\cdot\nx
f_{\l,m}+\ll_m[f_{\l,m}]&=&S(\vx,\vv)\ \ \text{in}\ \
\Omega,\\\rule{0ex}{1.5em}
f_{\l,m}(\vx_0,\vv)&=&h(\vx_0,\vv)\ \ \text{for}\ \ \vx_0\in\p\Omega\ \ \text{and}\
\vv\cdot\vn<0,
\end{array}
\right.
\end{eqnarray}
with $\ll_m$ the linearized Boltzmann operator corresponding to the
cut-off cross section $q_m=\min\{q,m\}$. Also, the solution
satisfies
\begin{eqnarray}
\e\um{f_{\l,m}}^2+\tss{f_{\l,m}}{+}^2\leq C(\l,m)\left(
\tm{S}^2+\e\tss{h}{-}^2\right).
\end{eqnarray}
\end{lemma}
\begin{proof}
We divide the proof into several steps:\\
\ \\
Step 1: Definition of iteration.\\
Denote $\ll_m=\nu_m-K_m$. We define the iteration in
$l$: $f_{\l,m}^0=0$ and for $l\geq0$,
\begin{eqnarray}\label{wt 03}
\left\{
\begin{array}{rcl}
\l f_{\l,m}^{l+1}+\e\vv\cdot\nx
f_{\l,m}^{l+1}+(1+M)\nu_mf_{\l,m}^{l+1}&=&S(\vx,\vv)-(K_{m}+M\nu_m)[f_{\l,m}^l],\\\rule{0ex}{1.5em}
f_{\l,m}^{l+1}(\vx_0,\vv)&=&h(\vx_0,\vv)\ \ \text{for}\ \ \vx_0\in\p\Omega\ \ \text{and}\
\vv\cdot\vn<0,
\end{array}
\right.
\end{eqnarray}
where $M>0$ is a fixed real number to be determined later. Since
\begin{eqnarray}
\im{(K_{m}+M\nu_m)[f_{\l,m}^l]}\leq C(m,M)\im{f_{\l,m}^l},
\end{eqnarray}
Lemma \ref{wellposedness prelim lemma 3} implies $f_{\l,m}^l\in
L^{\infty}(\Omega\times\r^2)$ are well-defined for $l\geq0$. However, we
cannot directly obtain the existence of limit
$f_{\l,m}^l$ as $l\rt\infty$.\\
\ \\
Step 2: The limit $l\rt\infty$.\\
Based on Green's identity in Lemma \ref{wellposedness prelim lemma
2}, multiplying $f_{\l,m}^{l+1}$ on both sides of (\ref{wt 03}), we have
\begin{eqnarray}
&&\l\tm{f_{\l,m}^{l+1}}^2+\frac{\e}{2}\tss{f_{\l,m}^{l+1}}{+}^2+\br{(1+M)\nu_mf_{\l,m}^{l+1},f_{\l,m}^{l+1}}\\
&=&\br{(K_{m}+M\nu_m)[f_{\l,m}^{l}],f_{\l,m}^{l+1}}+
\frac{\e}{2}\tss{h}{-}^2+\br{f_{\l,m}^{l+1},S}.\no
\end{eqnarray}
Since $\ll_m=\nu_m-K_{m}$ is a non-negative symmetric operator, we
can always find $M$ sufficiently large such that $K_m+M\nu_m$ is also a
non-negative operator. Then we deduce
\begin{eqnarray}
&&\br{(K_{m}+M\nu_m)[f_{\l,m}^{l}],f_{\l,m}^{l+1}}\\
&\leq&\sqrt{\br{(K_{m}+M\nu_m)[f_{\l,m}^{l}],f_{\l,m}^{l}}}\sqrt{\br{(K_{m}+M\nu_m)[f_{\l,m+1}^{l+1}],f_{\l,m}^{l+1}}}\no\\
&\leq&\half\bigg(\br{(K_{m}+M\nu_m)[f_{\l,m}^{l}],f_{\l,m}^{l}}+\br{(K_{m}+M\nu_m)[f_{\l,m+1}^{l+1}],f_{\l,m}^{l+1}}\bigg)\no\\
&\leq&\half\bigg(\br{(1+M)\nu_m[f_{\l,m}^{l}],f_{\l,m}^{l}}+\br{(1+M)\nu_m[f_{\l,m}^{l+1}],f_{\l,m}^{l+1}}\bigg).\no
\end{eqnarray}
Considering the fact
\begin{eqnarray}
\br{(1+M)\nu_m[f_{\l,m}^{l}],f_{\l,m}^{l}}&\leq&
(1+M)m\tm{f_{\l,m}^{l}},
\end{eqnarray}
we obtain
\begin{eqnarray}
&&\bigg(\frac{\l}{(1+M)m}+1\bigg)\br{(1+M)\nu_m[f_{\l,m}^{l+1}],f_{\l,m}^{l+1}}+\frac{\e}{2}\tss{f_{\l,m}^{l+1}}{+}^2\\
&\leq&\half\bigg(\br{(1+M)\nu_m[f_{\l,m}^{l}],f_{\l,m}^{l}}+\br{(1+M)\nu_m[f_{\l,m}^{l}],f_{\l,m}^{l+1}}\bigg)\no\\
&&+\frac{\e}{2}\tss{h}{-}^2+\frac{\l}{2(1+M)m}\br{(1+M)\nu_m[f_{\l,m}^{l+1}],f_{\l,m}^{l+1}}+\frac{(1+M)m}{2\l}\tm{S}^2.\no
\end{eqnarray}
Since
\begin{eqnarray}
\frac{\l}{(1+M)m}+1-\half-\frac{\l}{2(1+M)m}&>&\half,
\end{eqnarray}
by iteration over $l$, for
\begin{eqnarray}
C_1(\l,m)&=&\frac{1}{1+\frac{\l}{(1+M)m}}<1,\\
C_2(\l,m)&=&\frac{1}{1+\frac{\l}{(1+M)m}}\frac{(1+M)m}{\l}>0,
\end{eqnarray}
we have
\begin{eqnarray}
&&\e\tss{f_{\l,m}^{l+1}}{+}^2+\left(1+\frac{\l}{(1+M)m}\right)\br{(1+M)\nu_m[f_{\l,m}^{l+1}],f_{\l,m}^{l+1}}\\
&\leq&
C_1(\l,m)\left(\e\tss{f_{\l,m}^{l}}{+}^2+\left(1+\frac{\l}{(1+M)m}\right)\br{(1+M)\nu_m[f_{\l,m}^{l}],f_{\l,m}^{l}}\right)\no\\
&&+C_2(\l,m)\bigg(\tm{S}^2+\e\tss{h}{-}^2\bigg)\no.
\end{eqnarray}
Taking the difference of $f_{\l,m}^{l+1}-f_{\l,m}^{l}$, we
conclude that $f_{\l,m}^{l}$ is a Cauchy sequence. We take
$l\rt\infty$ to obtain $f_{\l,m}$ as a solution to the equation
(\ref{wellposedness prelim penalty equation}) satisfying
\begin{eqnarray}
&&\e\tss{f_{\l,m}}{+}^2+\left(1+\frac{\l}{(1+M)m}\right)\br{(1+M)\nu_m[f_{\l,m}],f_{\l,m}}\leq
\frac{C_2(\l,m)}{1-C_1(\l,m)}\bigg(\tm{S}^2+\e\tss{h}{-}^2\bigg).
\end{eqnarray}
Then our results naturally follows.
\end{proof}
Note that the estimate is not uniform in $\l$ as $\l\rt0$. We need to find a stronger estimate of $f_{\l,m}$.
\begin{lemma}\label{wellposedness prelim lemma 5}
The solution $f_{\l,m}$ to the equation (\ref{wellposedness prelim
penalty equation}) satisfies the estimate
\begin{eqnarray}\label{wt 08}
\e\tm{\pk[f_{\l,m}]}&\leq& C\bigg(
\e\tss{f_{\l,m}}{+}+\tm{(\ik-\pk)[f_{\l,m}]}+\tm{S}+\e\tss{h}{-}\bigg),
\end{eqnarray}
for $0\leq\l\leq\e<<1$.
\end{lemma}
\begin{proof}
Applying Green's identity in Lemma \ref{wellposedness prelim lemma
2} to the equation (\ref{wellposedness prelim
penalty equation}), for any $\psi\in L^2(\Omega\times\r^2)$
satisfying $\vv\cdot\nx\psi\in L^2(\Omega\times\r^2)$ and $\psi\in
L^2(\gamma)$, we have
\begin{eqnarray}\label{wt 11}
&&\l\iint_{\Omega\times\r^2}f_{\l,m}\psi+\e\int_{\gamma_+}f_{\l,m}\psi\ud{\gamma}-\e\int_{\gamma_-}f_{\l,m}\psi\ud{\gamma}
-\e\iint_{\Omega\times\r^2}(\vv\cdot\nx\psi)f_{\l,m}\\
&=&-\iint_{\Omega\times\r^2}\psi(\ik-\pk)[f_{\l,m}]+\iint_{\Omega\times\r^2}S\psi.\no
\end{eqnarray}
Since
\begin{eqnarray}
\pk[f]=\sqrt{\m}\bigg(a+\vv\cdot\vb+\frac{\abs{\vv}^2-2}{2}c\bigg),
\end{eqnarray}
our goal is to choose a particular test function $\psi$ to estimate
$a$, $\vb$ and $c$.\\
\ \\
Step 1: Estimates of $c$.\\
We choose the test function
\begin{eqnarray}\label{wt 12}
\psi=\psi_c=\sqrt{\m(\vv)}\left(\abs{\vv}^2-\beta_c\right)\left(\vv\cdot\nx\phi_c(\vx)\right),
\end{eqnarray}
where
\begin{eqnarray}
\left\{
\begin{array}{rcl}
-\dx\phi_c(\vx)&=&c(\vx)\ \ \text{in}\ \
\Omega,\\
\phi_c&=&0\ \ \text{on}\ \ \p\Omega,
\end{array}
\right.
\end{eqnarray}
and $\beta_c$ is a real number to be determined later. Based on the
standard elliptic estimates, we have
\begin{eqnarray}
\nm{\phi_c}_{H^2}\leq C\tm{c}.
\end{eqnarray}
With the choice of (\ref{wt 12}), the right-hand side(RHS) of
(\ref{wt 11}) is bounded by
\begin{eqnarray}
\text{RHS}\leq C\tm{c}\bigg(\tm{(\ik-\pk)[f_{\l,m}]}+\tm{S}\bigg).
\end{eqnarray}
We have
\begin{eqnarray}
\vv\cdot\nx\psi_c&=&\sqrt{\m(\vv)}\sum_{i,j=1}^2\left(\abs{\vv}^2-\beta_c\right)\v_i\v_j\p_{ij}\phi_c(\vx),
\end{eqnarray}
so the left-hand side(LHS) of (\ref{wt 11}) takes the form
\begin{eqnarray}\label{wt 13}
\text{LHS}&=&\l\iint_{\Omega\times\r^2}f_{\l,m}\sqrt{\m(\vv)}\left(\abs{\vv}^2-\beta_c\right)\left(\sum_{i=1}^2\v_i\p_i\phi_c\right)\\
&&+\e\int_{\p\Omega\times\r^2}f_{\l,m}\sqrt{\m(\vv)}\left(\abs{\vv}^2-\beta_c\right)\left(\sum_{i=1}^2\v_i\p_i\phi_c\right)(\vn\cdot\vv)\no\\
&&-\e\iint_{\Omega\times\r^2}f_{\l,m}\sqrt{\m(\vv)}\left(\abs{\vv}^2-\beta_c\right)\left(\sum_{i,j=1}^2\v_i\v_j\p_{ij}\phi_c\right).\no
\end{eqnarray}
We decompose
\begin{eqnarray}
(f_{\l,m})_{\gamma}&=&\id_{\gamma_+}f_{\l,m}+\id_{\gamma_-}h\ \
\text{on}\ \ \gamma, \label{wt 14}\\
f_{\l,m}&=&\sqrt{\m}\bigg(a+\vv\cdot
\vb+\frac{\abs{\vv}^2-2}{2}c\bigg)+(\ik-\pk)[f_{\l,m}]\ \ \text{in}\
\ \Omega\times\r^2.\label{wt 15}
\end{eqnarray}
Note that the operator $\pk$ and $\ik-\pk$ are defined independent
of cut-off parameter $m$. We will choose $\beta_c$ such that
\begin{eqnarray}
\int_{\r^2}\sqrt{\m(\vv)}\left(\abs{\vv}^2-\beta_c\right)\v_i^2\ud{\vv}=0\
\ \text{for}\ \ i=1,2.
\end{eqnarray}
Since $\m(\vv)$ takes the form
\begin{eqnarray}
\m(\vv)=C\exp\left(-\frac{\abs{\vv}^2}{2}\right),
\end{eqnarray}
this $\beta_c$ can always be achieved. Now substitute (\ref{wt 14})
and (\ref{wt 15}) into (\ref{wt 13}). Then based on this choice of
$\beta_c$ and oddness in $\vv$, there is no $\pk[f_{\l,m}]$
contribution in the first term and no $a$ contribution in the third term of
(\ref{wt 13}). Since $\vb$ contribution and the off-diagonal $c$
contribution in the third term of (\ref{wt 13}) also vanish due to
the oddness in $\vv$, we can simplify (\ref{wt 13}) into
\begin{eqnarray}
\text{LHS}
&=&\l\iint_{\Omega\times\r^2}(\ik-\pk)[f_{\l,m}]\sqrt{\m(\vv)}\left(\abs{\vv}^2-\beta_c\right)\left(\sum_{i=1}^2\v_i\p_i\phi_c\right)\\
&&+\e\int_{\p\Omega\times\r^2}\id_{\gamma_+}f_{\l,m}
\sqrt{\m(\vv)}\left(\abs{\vv}^2-\beta_c\right)\left(\sum_{i=1}^2\v_i\p_i\phi_c\right)(\vn\cdot\vv)\no\\
&&+\e\int_{\p\Omega\times\r^2}\id_{\gamma_-}h
\sqrt{\m(\vv)}\left(\abs{\vv}^2-\beta_c\right)\left(\sum_{i=1}^2\v_i\p_i\phi_c\right)(\vn\cdot\vv)\no\\
&&-\e\sum_{i=1}^2\int_{\r^2}\m(\vv)\abs{\v_i}^2\left(\abs{\vv}^2-\beta_c\right)\frac{\abs{\vv}^2-2}{2}\ud{\vv}
\int_{\Omega}c(\vx)\p_{ii}\phi_c(\vx)\ud{\vx}\no\\
&&-\e\iint_{\Omega\times\r^2}(\ik-\pk)[f_{\l,m}]\sqrt{\m(\vv)}\left(\abs{\vv}^2-\beta_c\right)\left(\sum_{i,j=1}^2\v_i\v_j\p_{ij}\phi_c\right).\no
\end{eqnarray}
Since
\begin{eqnarray}
\int_{\r^2}\m(\vv)\abs{\v_i}^2\left(\abs{\vv}^2-\beta_c\right)\frac{\abs{\vv}^2-2}{2}\ud{\vv}=C,
\end{eqnarray}
we have
\begin{eqnarray}
\\
\e\abs{\int_{\Omega}\dx\phi_c(\vx)c(\vx)\ud{\vx}}
&\leq&C\tm{c}\bigg(\e\tss{f_{\l,m}}{+}+(1+\e+\l)\tm{(\ik-\pk)[f_{\l,m}]}+\tm{S}+\e\tss{h}{-}\bigg),\no
\end{eqnarray}
where we have used the elliptic estimates and the trace estimate:
$\ts{\nx\phi_c}\leq C\nm{\phi_c}_{H^2}\leq C\tm{c}$. Since $-\dx\phi_c=c$,
we know
\begin{eqnarray}
\e\tm{c}^2
&\leq&C
\tm{c}\bigg(\e\tss{f_{\l,m}}{+}+(1+\e+\l)\tm{(\ik-\pk)[f_{\l,m}]}+\tm{S}+\e\tss{h}{-}\bigg),
\end{eqnarray}
which further implies
\begin{eqnarray}\label{wt 16}
&&\e\tm{c}\leq C
\e\tss{f_{\l,m}}{+}+(1+\e+\l)\tm{(\ik-\pk)[f_{\l,m}]}+\tm{S}+\e\tss{h}{-}.
\end{eqnarray}
\ \\
Step 2: Estimates of $\vb$.\\
\ \\
Step 2 - Phase 1: Estimates of $(\p_{ij}\dx^{-1}b_j)b_i$ for $i,j=1,2$.\\
We choose the test function
\begin{eqnarray}\label{wt 17}
\psi=\psi_b^{i,j}=\sqrt{\m(\vv)}\left(\v_i^2-\beta_b\right)\p_j\phi_b^j,
\end{eqnarray}
where
\begin{eqnarray}
\left\{
\begin{array}{rcl}
-\dx\phi_b^j(\vx)&=&b_j(\vx)\ \ \text{in}\ \
\Omega,\\
\phi_b^j&=&0\ \ \text{on}\ \ \p\Omega,
\end{array}
\right.
\end{eqnarray}
and $\beta_b$ is a real number to be determined later. Based on the
standard elliptic estimates, we have
\begin{eqnarray}
\nm{\phi_b^j}_{H^2}\leq C \tm{\vb}.
\end{eqnarray}
With the choice of (\ref{wt 17}), the right-hand side(RHS) of
(\ref{wt 11}) is bounded by
\begin{eqnarray}
\text{RHS}\leq C \tm{\vb}\bigg(\tm{(\ik-\pk)[f_{\l,m}]}+\tm{S}\bigg).
\end{eqnarray}
Hence, the left-hand side(LHS) of (\ref{wt 11}) takes the form
\begin{eqnarray}\label{wt 18}
\text{LHS}&=&\l\iint_{\Omega\times\r^2}f_{\l,m}\sqrt{\m(\vv)}\left(\v_i^2-\beta_b\right)\p_j\phi_b^j\\
&&+\e\int_{\p\Omega\times\r^2}f_{\l,m}\sqrt{\m(\vv)}\left(\v_i^2-\beta_b\right)\p_j\phi_b^j(\vn\cdot\vv)\no\\
&&-\e\iint_{\Omega\times\r^2}f_{\l,m}\sqrt{\m(\vv)}\left(\v_i^2-\beta_b\right)\left(\sum_{i=1}^2\v_i\p_{ij}\phi_b^j\right).\no
\end{eqnarray}
Now substitute (\ref{wt 14}) and (\ref{wt 15}) into (\ref{wt 18}).
Then based on the oddness in $\vv$, there is no $\vb$ contribution
in the first term and no $a$ and $c$ contribution in the third term of (\ref{wt
18}). We can simplify (\ref{wt 18}) into
\begin{eqnarray}
\text{LHS}
&=&\l\iint_{\Omega\times\r^2}(\ik-\pk)[f_{\l,m}]\sqrt{\m(\vv)}\left(\v_i^2-\beta_b\right)\p_j\phi_b^j\\
&&+\l\iint_{\Omega\times\r^2}a(\vx)\m(\vv)\left(\v_i^2-\beta_b\right)\p_j\phi_b^j\no\\
&&+\l\iint_{\Omega\times\r^2}c(\vx)\m(\vv)\frac{\abs{\vv}^2-2}{2}\left(\v_i^2-\beta_b\right)\p_j\phi_b^j\no\\
&&+\e\int_{\p\Omega\times\r^2}\id_{\gamma_+}f_{\l,m}
\sqrt{\m(\vv)}\left(\v_i^2-\beta_b\right)\p_j\phi_b^j(\vn\cdot\vv)\no\\
&&+\e\int_{\p\Omega\times\r^2}\id_{\gamma_-}h
\sqrt{\m(\vv)}\left(\v_i^2-\beta_b\right)\p_j\phi_b^j(\vn\cdot\vv)\no\\
&&-\e\sum_{l=1}^2\iint_{\Omega\times\r^2}\m(\vv)\v_l^2\left(\v_i^2-\beta_b\right)\p_{lj}\phi_b^j(\vx)b_l\no\\
&&-\e\sum_{l=1}^2\iint_{\Omega\times\r^2}(\ik-\pk)[f_{\l,m}]\sqrt{\m(\vv)}\left(\v_i^2-\beta_b\right)\v_l\p_{lj}\phi_b^j.\no
\end{eqnarray}
We will choose $\beta_b$ such that
\begin{eqnarray}
\int_{\r^2}\m(\vv)\left(\abs{\v_i}^2-\beta_b\right)\ud{\vv}=0\ \
\text{for}\ \ i=1,2.
\end{eqnarray}
Since $\m(\vv)$ takes the form
\begin{eqnarray}
\m(\vv)=C\exp\left(-\frac{\abs{\vv}^2}{2}\right),
\end{eqnarray}
this $\beta_b$ can always be achieved. Based on this choice of
$\beta_b$,we have
\begin{eqnarray}
\l\iint_{\Omega\times\r^2}a_{f_{\l,m}}\m(\vv)\left(\v_i^2-\beta_b\right)\p_j\phi_b^j=0
\end{eqnarray}
For such $\beta_b$ and any $i\neq l$, we can directly compute
\begin{eqnarray}
\int_{\r^2}\m(\vv)\left(\abs{\v_i}^2-\beta_b\right)\v_l^2\ud{\vv}&=&0,\\
\int_{\r^2}\m(\vv)\left(\abs{\v_i}^2-\beta_b\right)\v_i^2\ud{\vv}&=&C\neq0.
\end{eqnarray}
Then we deduce
\begin{eqnarray}
&&-\e\sum_{l=1}^2\iint_{\Omega\times\r^2}\m(\vv)\v_l^2\left(\v_i^2-\beta_b\right)\p_{lj}\phi_b^j(\vx)b_i\\
&=&-\e\iint_{\Omega\times\r^2}\m(\vv)\v_i^2\left(\v_i^2-\beta_b\right)\p_{ij}\phi_b^j(\vx)b_l-\e\sum_{l\neq
i}\iint_{\Omega\times\r^2}\m(\vv)\v_l^2\left(\v_i^2-\beta_b\right)\p_{lj}\phi_b^j(\vx)b_l\no\\
&=&C\int_{\Omega}(\p_{ij}\dx^{-1}b_j)b_i,\no
\end{eqnarray}
and
\begin{eqnarray}
\l\iint_{\Omega\times\r^2}c(\vx)\m(\vv)\frac{\abs{\vv}^2-2}{2}\left(\v_i^2-\beta_b\right)\p_j\phi_b^j
&=&\l\iint_{\Omega\times\r^2}c(\vx)\m(\vv)\frac{\v_i^2-2}{2}\left(\v_i^2-\beta_b\right)\p_j\phi_b^j.
\end{eqnarray}
Hence, by (\ref{wt 16}), we may estimate
\begin{eqnarray}\label{wt 19}
&&\e\abs{\int_{\Omega}(\p_{ij}\dx^{-1}b_j)b_i}\\
&\leq&
C\tm{\vb}\bigg(\e\tss{f_{\l,m}}{+}+(1+\e+\l)\tm{(\ik-\pk)[f_{\l,m}]}+\tm{S}+\e\tss{h}{-}+\l\tm{c}\bigg)\no\\
&\leq&
C\tm{\vb}\bigg((\e+\l)\tss{f_{\l,m}}{+}+(1+\e+\l)\tm{(\ik-\pk)[f_{\l,m}]}+(1+\l)\tm{S}+(\e+\l)\tss{h}{-}\bigg)\no.
\end{eqnarray}
\ \\
Step 2 - Phase 2: Estimates of $(\p_{jj}\dx^{-1}b_i)b_i$ for $i\neq
j$.\\
We choose the test function
\begin{eqnarray}
\psi=\sqrt{\m(\vv)}\abs{\vv}^2\v_i\v_j\p_j\phi_b^i\ \ i\neq j.
\end{eqnarray}
The right-hand side(RHS) of (\ref{wt 11}) is still bounded by
\begin{eqnarray}
\text{RHS}\leq C\tm{\vb}\bigg(\tm{(\ik-\pk)[f_{\l,m}]}+\tm{S}\bigg).
\end{eqnarray}
Hence, the left-hand side(LHS) of (\ref{wt 11}) takes the form
\begin{eqnarray}\label{wt 20}
\text{LHS}&=&\l\iint_{\Omega\times\r^2}f_{\l,m}\sqrt{\m(\vv)}\abs{\vv}^2\v_i\v_j\p_j\phi_b^i\\
&&+\e\int_{\p\Omega\times\r^2}f_{\l,m}\sqrt{\m(\vv)}\abs{\vv}^2\v_i\v_j\p_j\phi_b^i(\vn\cdot\vv)\no\\
&&-\e\iint_{\Omega\times\r^2}f_{\l,m}\sqrt{\m(\vv)}\abs{\vv}^2\v_i\v_j\left(\sum_{l=1}^2\v_i\p_{lj}\phi_b^i\right).\no
\end{eqnarray}
Now substitute (\ref{wt 14}) and (\ref{wt 15}) into (\ref{wt 20}).
Then based on the oddness in $\vv$, there is no $\pk[f_{\l,m}]$
contribution in the first term and no $a$ and $c$ contribution in the third term
of (\ref{wt 20}). We can simplify (\ref{wt 20}) into
\begin{eqnarray}
\text{LHS}
&=&\l\iint_{\Omega\times\r^2}(\ik-\pk)[f_{\l,m}]\sqrt{\m(\vv)}\abs{\vv}^2\v_i\v_j\p_j\phi_b^i\\
&&+\e\int_{\p\Omega\times\r^2}\id_{\gamma_+}f_{\l,m}
\sqrt{\m(\vv)}\abs{\vv}^2\v_i\v_j\p_j\phi_b^i(\vn\cdot\vv)\no\\
&&+\e\int_{\p\Omega\times\r^2}\id_{\gamma_-}h
\sqrt{\m(\vv)}\abs{\vv}^2\v_i\v_j\p_j\phi_b^i(\vn\cdot\vv)\no\\
&&-\e\iint_{\Omega\times\r^2}\m(\vv)\abs{\vv}^2\v_i^2\v_j^2(\p_{ij}\phi_b^i(\vx)b_j+\p_{jj}\phi_b^i(\vx)b_i)\no\\
&&-\e\sum_{l=1}^2\iint_{\Omega\times\r^2}(\ik-\pk)[f_{\l,m}]\sqrt{\m(\vv)}\abs{\vv}^2\v_i\v_j\v_l\p_{lj}\phi_b^i.\no
\end{eqnarray}
Then we deduce
\begin{eqnarray}
&&-\e\iint_{\Omega\times\r^2}\m(\vv)\abs{\vv}^2\v_i^2\v_j^2(\p_{ij}\phi_b^i(\vx)b_j+\p_{jj}\phi_b^i(\vx)b_i)=
C\bigg(\int_{\Omega}(\p_{ij}\dx^{-1}b_i)b_j+\int_{\Omega}(\p_{jj}\dx^{-1}b_i)b_i\bigg).
\end{eqnarray}
Hence, we may estimate for $i\neq j$,
\begin{eqnarray}\label{wt 21}
\e\abs{\int_{\Omega}(\p_{jj}\dx^{-1}b_i)b_i}
&\leq&C\tm{\vb}\bigg(\e\tss{f_{\l,m}}{+}+(1+\e+\l)\tm{(\ik-\pk)[f_{\l,m}]}\\
&&+\tm{S}+\e\tss{h}{-}\bigg)+C\e\abs{\int_{\Omega}(\p_{ij}\dx^{-1}b_i)b_j}.\no
\end{eqnarray}
Moreover, by (\ref{wt 19}), for $i=j=1,2$,
\begin{eqnarray}\label{wt 22}
\e\abs{\int_{\Omega}(\p_{jj}\dx^{-1}b_j)b_j}
&\leq&C\tm{\vb}\bigg((\e+\l)\tss{f_{\l,m}}{+}+(1+\e+\l)\tm{(\ik-\pk)[f_{\l,m}]}\\
&&+(1+\l)\tm{S}+(\e+\l)\tss{h}{-}\bigg).\no
\end{eqnarray}
\ \\
Step 2 - Phase 3: Synthesis.\\
Summarizing (\ref{wt 21}) and (\ref{wt 22}), we may sum up over
$j=1,2$ to obtain, for any $i=1,2$,
\begin{eqnarray}\label{wt 23}
\e\tm{b_i}^2
&\leq&C\tm{\vb}\bigg((\e+\l)\tss{f_{\l,m}}{+}+(1+\e+\l)\tm{(\ik-\pk)[f_{\l,m}]}\\
&&+(1+\l)\tm{S}+(\e+\l)\tss{h}{-}\bigg),\no
\end{eqnarray}
which further implies
\begin{eqnarray}\label{wt 24}
\\
\e\tm{\vb}&\leq&C\bigg((\e+\l)\tss{f_{\l,m}}{+}+(1+\e+\l)\tm{(\ik-\pk)[f_{\l,m}]}+(1+\l)\tm{S}+(\e+\l)\tss{h}{-}\bigg).\no
\end{eqnarray}
\ \\
Step 3: Estimates of $a$.\\
We choose the test function
\begin{eqnarray}\label{wt 25}
\psi=\psi_a=\sqrt{\m(\vv)}\left(\abs{\vv}^2-\beta_a\right)\left(\vv\cdot\nx\phi_a(\vx)\right),
\end{eqnarray}
where
\begin{eqnarray}
\left\{
\begin{array}{rcl}
-\dx\phi_a(\vx)&=&a(\vx)\ \ \text{in}\ \ \Omega,\\\rule{0ex}{1.5em}
\phi_a&=&0\ \ \text{on}\ \ \p\Omega,
\end{array}
\right.
\end{eqnarray}
and $\beta_a$ is a real number to be determined later. Based on the
standard elliptic estimates,
we have
\begin{eqnarray}
\nm{\phi_a}_{H^2}\leq C \tm{a}.
\end{eqnarray}
With the choice of (\ref{wt 25}), the right-hand side(RHS) of
(\ref{wt 11}) is bounded by
\begin{eqnarray}
\text{RHS}\leq C \tm{a}\bigg(\tm{(\ik-\pk)[f_{\l,m}]}+\tm{S}\bigg).
\end{eqnarray}
We have
\begin{eqnarray}
\vv\cdot\nx\psi_a&=&\sqrt{\m(\vv)}\sum_{i,j=1}^2\left(\abs{\vv}^2-\beta_a\right)\v_i\v_j\p_{ij}\phi_a(\vx),
\end{eqnarray}
so the left-hand side(LHS) of (\ref{wt 11}) takes the form
\begin{eqnarray}\label{wt 26}
\text{LHS}&=&\l\iint_{\Omega\times\r^2}f_{\l,m}\sqrt{\m(\vv)}\left(\abs{\vv}^2-\beta_a\right)\left(\sum_{i=1}^2\v_i\p_i\phi_a\right)\\
&&+\e\int_{\p\Omega\times\r^2}f_{\l,m}\sqrt{\m(\vv)}\left(\abs{\vv}^2-\beta_a\right)\left(\sum_{i=1}^2\v_i\p_i\phi_a\right)(\vn\cdot\vv)\no\\
&&-\e\iint_{\Omega\times\r^2}f_{\l,m}\sqrt{\m(\vv)}\left(\abs{\vv}^2-\beta_a\right)\left(\sum_{i,j=1}^2\v_i\v_j\p_{ij}\phi_a\right).\no
\end{eqnarray}
We will choose $\beta_a$ such that
\begin{eqnarray}
\int_{\r^2}\sqrt{\m(\vv)}\left(\abs{\vv}^2-\beta_a\right)\frac{\abs{\vv}^2-2}{2}\v_i^2\ud{\vv}=0\
\ \text{for}\ \ i=1,2.
\end{eqnarray}
Since
\begin{eqnarray}
\int_{\r^2}\sqrt{\m(\vv)}\frac{\abs{\vv}^2-2}{2}\v_i^2\ud{\vv}\neq0,
\end{eqnarray}
this $\beta_a$ can always be achieved. Now substitute (\ref{wt 14})
and (\ref{wt 15}) into (\ref{wt 26}). Then based on this choice of
$\beta_a$ and oddness in $\vv$, there is no $a$ and $c$ contribution
in the first term, and no $\vb$ and $c$ contribution in the third
term of (\ref{wt 26}). Since $\vb$ contribution and the off-diagonal
$c$ contribution in the third term of (\ref{wt 26}) also vanish due
to the oddness in $\vv$, we can simplify (\ref{wt 26}) into
\begin{eqnarray}
\text{LHS}
&=&\l\iint_{\Omega\times\r^2}(\ik-\pk)[f_{\l,m}]\sqrt{\m(\vv)}\left(\abs{\vv}^2-\beta_a\right)\left(\sum_{i=1}^2\v_i\p_i\phi_a\right)\\
&&+\l\iint_{\Omega\times\r^2}\m(\vv)\left(\abs{\vv}^2-\beta_a\right)\left(\sum_{i=1}^2b_i\v_i^2\p_i\phi_a\right)\no\\
&&+\e\int_{\p\Omega\times\r^2}\id_{\gamma_+}f_{\l,m}
\sqrt{\m(\vv)}\left(\abs{\vv}^2-\beta_a\right)\left(\sum_{i=1}^2\v_i\p_i\phi_a\right)(\vn\cdot\vv)\no\\
&&+\e\int_{\p\Omega\times\r^2}\id_{\gamma_-}h
\sqrt{\m(\vv)}\left(\abs{\vv}^2-\beta_a\right)\left(\sum_{i=1}^2\v_i\p_i\phi_a\right)(\vn\cdot\vv)\no\\
&&-\sum_{i=1}^2\e\int_{\r^2}\m(\vv)\abs{\v_i}^2\left(\abs{\vv}^2-\beta_a\right)\ud{\vv}
\int_{\Omega}a(\vx)\p_{ii}\phi_a(\vx)\ud{\vx}\no\\
&&-\e\iint_{\Omega\times\r^2}(\ik-\pk)[f_{\l,m}]\sqrt{\m(\vv)}\left(\abs{\vv}^2-\beta_a\right)\left(\sum_{i,j=1}^2\v_i\v_j\p_{ij}\phi_a\right).\no
\end{eqnarray}
Since
\begin{eqnarray}
\int_{\r^2}\sqrt{\m(\vv)}\abs{\v_i}^2\left(\abs{\vv}^2-\beta_a\right)\ud{\vv}=C,
\end{eqnarray}
we have
\begin{eqnarray}
&&-\e\int_{\Omega}\dx\phi_a(\vx)a(\vx)\ud{\vx}\\
&\leq&C
\tm{a}\bigg(\e\tss{f_{\l,m}}{+}+(1+\e+\l)\tm{(\ik-\pk)[f_{\l,m}]}+\tm{S}+\e\tss{h}{-}+\l\tm{\vb}\bigg).\no
\end{eqnarray}
Since $-\dx\phi_a=a$, by (\ref{wt 24}), we know
\begin{eqnarray}
\e\tm{a}^2 &\leq&C
\tm{a}\bigg((\e+\l)\tss{f_{\l,m}}{+}+(1+\e+\l)\tm{(\ik-\pk)[f_{\l,m}]}\\
&&+(1+\l)\tm{S}+(\e+\l)\tss{h}{-}\bigg),\no
\end{eqnarray}
which further implies
\begin{eqnarray}\label{wt 28}
\\
\e\tm{a}&\leq&C\bigg(
(\e+\l)\tss{f_{\l,m}}{+}+(1+\e+\l)\tm{(\ik-\pk)[f_{\l,m}]}+(1+\l)\tm{S}+(\e+\l)\tss{h}{-}\bigg).\no
\end{eqnarray}
\ \\
Step 4: Synthesis.\\
Collecting (\ref{wt 16}), (\ref{wt 24}) and (\ref{wt 28}), we deduce
\begin{eqnarray}
&&\e\tm{\pk[f_{\l,m}]}\\
&\leq& C\bigg((\e+\l)\tss{f_{\l,m}}{+}+(1+\e+\l)\tm{(\ik-\pk)[f_{\l,m}]}+(1+\l)\tm{S}+(\e+\l)\tss{h}{-}\bigg).\no
\end{eqnarray}
This completes our proof.
\end{proof}
\begin{theorem}\label{wellposedness LT estimate}
There exists a unique solution $f\in L^2(\Omega\times\r^2)$ to the
equation (\ref{linear steady}) that satisfies the estimate
\begin{eqnarray}
\tm{f}+\frac{1}{\e^{1/2}}\tss{f}{+}&\leq&C\bigg(\frac{1}{\e^2}\tm{(\ik-\pk)[S]}+\frac{1}{\e}\tm{\pk[S]}+\frac{1}{\e^{1/2}}\tss{h}{-}\bigg).
\end{eqnarray}
\end{theorem}
\begin{proof}
We square on both sides of (\ref{wt 08}) to obtain
\begin{eqnarray}\label{wt 29}
\e^2\tm{\pk[f_{\l,m}]}&\leq&C
\bigg(\e^2\tss{f_{\l,m}}{+}+\tm{(\ik-\pk)[f_{\l,m}]}+\tm{S}+\e^2\tss{h}{-}\bigg).
\end{eqnarray}
On the other hand, by Green's identity, multiplying $f_{l,m}$ on both sides of (\ref{wellposedness prelim
penalty equation}) implies
\begin{eqnarray}\label{wt 34}
\l\tm{f_{\l,m}}^2+\br{\ll_mf_{\l,m},f_{\l,m}}+\half\e\tss{f_{\l,m}}{+}^2
&=&\half\e\tss{h}{-}^2+\iint_{\Omega\times\r^2}f_{\l,m}S.
\end{eqnarray}
We deduce from the spectral gap of $\ll_m$,
\begin{eqnarray}\label{wt 31}
\l\tm{f_{\l,m}}^2+\tm{(\ik-\pk)[f_{\l,m}]}+\frac{\e}{2}\tss{f_{\l,m}}{+}^2
&\leq&\e\tss{h}{-}^2+\iint_{\Omega\times\r^2}f_{\l,m}S.
\end{eqnarray}
Multiplying a small constant on both sides of (\ref{wt 29}) and
adding to (\ref{wt 31}), we obtain
\begin{eqnarray}\label{wt 37}
\e^2\tm{\pk[f_{\l,m}]}^2+\tm{(\ik-\pk)[f_{\l,m}]}^2+\e\tss{f_{\l,m}}{+}^2\leq C\bigg(\e\tss{h}{-}^2+\iint_{\Omega\times\r^2}f_{\l,m}S+\tm{S}\bigg).
\end{eqnarray}
Since
\begin{eqnarray}
\iint_{\Omega\times\r^2}f_{\l,m}S&=&\iint_{\Omega\times\r^2}f_{\l,m}\pk[S]+\iint_{\Omega\times\r^2}f_{\l,m}(\ik-\pk)[S]\\
&=&\iint_{\Omega\times\r^2}(\ik-\pk)[f_{\l,m}]\pk[S]+\iint_{\Omega\times\r^2}\pk[f_{\l,m}](\ik-\pk)[S]\no\\
&\leq&
\bigg(C\tm{(\ik-\pk)[f_{\l,m}]}^2+\frac{1}{4C}\tm{\pk[S]}^2\bigg)+\bigg(C\e^2\tm{\pk[f_{\l,m}]}^2+\frac{1}{4C\e^2}\tm{(\ik-\pk)[S]}^2\bigg),\no
\end{eqnarray}
for $C$ sufficiently small, we have
\begin{eqnarray}
\e^2\tm{\pk[f_{\l,m}]}^2+\tm{(\ik-\pk)[f_{\l,m}]}^2+\e\tss{f_{\l,m}}{+}^2\leq C\bigg(\frac{1}{\e^2}\tm{(\ik-\pk)[S]}^2+\tm{\pk[S]}^2+\e\tss{h}{-}^2\bigg).
\end{eqnarray}
Hence, we deduce
\begin{eqnarray}\label{wt 35}
\tm{f_{\l,m}}\leq C\bigg(\frac{1}{\e^2}\tm{(\ik-\pk)[S]}+\frac{1}{\e}\tm{\pk[S]}+\frac{1}{\e^{1/2}}\tss{h}{-}\bigg).
\end{eqnarray}
Then based on (\ref{wt 37}), we have
\begin{eqnarray}\label{wt 36}
\tss{f_{\l,m}}{+}\leq C\bigg(\frac{1}{\e^{3/2}}\tm{(\ik-\pk)[S]}+\frac{1}{\e^{1/2}}\tm{\pk[S]}+\tss{h}{-}\bigg).
\end{eqnarray}
This is a uniform estimate in $\l$, so we can obtain a weak solution
$f_{\l,m}\rt f_{m}$ with the same estimate (\ref{wt 35}) and (\ref{wt 36}). Moreover,
we have
\begin{eqnarray}
\left\{
\begin{array}{rcl}
\l(f_{\l,m}-f_m)+\e\vv\cdot\nx
(f_{\l,m}-f_m)+\ll_m[f_{\l,m}-f_m]&=&\l f_m,\\\rule{0ex}{1.0em}
(f_{\l,m}-f_m)(\vx_0,\vv)&=&0.
\end{array}
\right.
\end{eqnarray}
Then we have the estimate
\begin{eqnarray}
\tm{f_{\l,m}-f_m}\leq C\bigg(\frac{\l}{\e^2}\tm{(\ik-\pk)[f_m]}+\frac{\l}{\e}\tm{\pk[f_m]}\bigg).
\end{eqnarray}
Hence, $f_{\l,m}\rt f_m$ strongly in $L^2(\Omega\times\r^2)$ as
$\l\rt0$. Then we can take the limit $f_m\rt f$ as $m\rt\infty$. By
a diagonal process, there exists a unique weak solution such that
$f_m\rt f$ weakly in $L^2(\Omega\times\r^2)$. Then the weak lower
semi-continuity implies $f$ satisfies the same estimate (\ref{wt
35}).
\end{proof}

\subsection{$L^{\infty}$ Estimates of Linearized Stationary Boltzmann Equation}

The characteristics $(X(s), V(s))$ of the equation
(\ref{linear steady}) which goes
through $(\vx,\vv)$ is defined by
\begin{eqnarray}\label{character}
\left\{
\begin{array}{rcl}
(X(0), V(0))&=&(\vx,\vv)\\\rule{0ex}{2.0em}
\dfrac{\ud{X(s)}}{\ud{s}}&=&\e V(s),\\\rule{0ex}{2.0em}
\dfrac{\ud{V(s)}}{\ud{s}}&=&0.
\end{array}
\right.
\end{eqnarray}
which implies
\begin{eqnarray}
\left\{
\begin{array}{rcl}
X(s)&=&\vx+\e s\vv\\
V(s)&=&\vv
\end{array}
\right.
\end{eqnarray}
Define the backward exit time and exit position as
\begin{eqnarray}
t_b(\vx,\vv)&=&\inf\{t>0:\vx-\e t\vv\notin\Omega\},\\
x_b(\vx,\vv)&=&\vx-\e t_b(\vx,\vv)\vv\notin\Omega.
\end{eqnarray}
We define a weight function
\begin{eqnarray}
\vh(\vv)=w_{\vt,\zeta}(\vv)=\br{\vv}^{\vt}\ue^{\zeta\abs{\vv}^2},
\end{eqnarray}
and
\begin{eqnarray}
\tvh(\vv)=\frac{1}{\sqrt{\m(\vv)}\vh(\vv)}=\sqrt{2\pi}\frac{\ue^{\left(\frac{1}{4}-\zeta\right)\abs{\vv}^2}}
{\br{\vv}^{\vt}}.
\end{eqnarray}
\begin{lemma}\label{wellposedness prelim lemma 8}
We have
\begin{eqnarray}
\abs{k(\vv,\vv')}\leq
C\bigg(\abs{\vv-\vv'}+\abs{\frac{1}{\vv-\vv'}}\bigg)
\exp\bigg(-\frac{1}{8}\abs{\vv-\vv'}^2-\frac{1}{8}\frac{\abs{\abs{\vv}^2-\abs{\vv'}^2}^2}{\abs{\vv-\vv'}^2}\bigg).
\end{eqnarray}
Let $0\leq\zeta\leq 1/4$. Then there exists $0\leq C_1(\zeta)<1$ and
$C_2(\zeta)>0$ such that for $0\leq \d\leq C_1(\zeta)$,
\begin{eqnarray}
\\
\int_{\r^2}\bigg(\abs{\vv-\vv'}+\abs{\frac{1}{\vv-\vv'}}\bigg)
\exp\bigg(-\frac{1-\d}{8}\abs{\vv-\vv'}^2-\frac{1-\d}{8}\frac{\abs{\abs{\vv}^2-\abs{\vv'}^2}^2}{\abs{\vv-\vv'}^2}\bigg)
\frac{\ue^{\zeta\abs{\vv}^2}}{\ue^{\zeta\abs{\vv'}^2}}\ud{\vv'}
\leq\frac{C_2(\zeta)}{1+\abs{\vv}}.\no
\end{eqnarray}
\end{lemma}
\begin{proof}
See \cite[Lemma 3]{Guo2010}.
\end{proof}
\begin{theorem}\label{wellposedness LI estimate}
There exists a unique solution $f\in L^{\infty}(\Omega\times\r^2)$
to the equation (\ref{linear steady}) that satisfies the
estimate for $\vt>2$ and $0\leq\zeta\leq1/4$,
\begin{eqnarray}
&&\im{\bv f}+\iss{\bv f}{+}\\
&\leq&C\bigg(\frac{1}{\e}\tm{f}+\im{\bv S}+\iss{\bv h}{-}\bigg)\no\\
&\leq&C\bigg(\frac{1}{\e^3}\tm{(\ik-\pk)[S]}+\frac{1}{\e^2}\tm{\pk[S]}+\frac{1}{\e^{3/2}}\tss{h}{-}+\im{\bv S}+\iss{\bv h}{-}\bigg).\no
\end{eqnarray}
\end{theorem}
\begin{proof}
The existence and uniqueness follow from Theorem \ref{wellposedness LT estimate}, so we focus on the estimate. We divide the proof into several steps:\\
\ \\
Step 1: Mild formulation.\\
Denote
\begin{eqnarray}
g&=&\vh f,\\
K_{\vh}[g]&=&\vh
K\left[\frac{g}{\vh}\right]=\int_{\r^2}k_{\vh}(\vv,\vv')g(\vv')\ud{\vv'}.
\end{eqnarray}
Since $\ll=\nu-K$, we can rewrite the equation (\ref{linear steady})
along the characteristics by Duhamel's principle as
\begin{eqnarray}
g(\vx,\vv)&=& \vh(\vv)h(\vx-\e t_b\vv,\vv)\ue^{-\nu(\vv)
t_{1}}+\int_{0}^{t_b}\vh S(\vx-\e(t_b-s)\vv,\vv)\ue^{-\nu(\vv)
(t_b-s)}\ud{s}\\
&&+\int_{0}^{t_b}K_{\vh}[g(\vx-\e(t_b-s)\vv,\vv)]\ue^{-\nu(\vv)
(t_b-s)}\ud{s}\no\\
&=& \vh(\vv)h(\vx-\e t_b\vv,\vv)\ue^{-\nu(\vv)
t_b}+\int_{0}^{t_b}\vh S(\vx-\e(t_b-s)\vv,\vv)\ue^{-\nu(\vv)
(t_b-s)}\ud{s}\no\\
&&+\int_{0}^{t_b}\bigg(\int_{\r^2}k_{\vh}(\vv,\vv_t)g(\vx-\e(t_b-s)\vv,\vv_t)\ud{\vv_t}\bigg)\ue^{-\nu(\vv)
(t_b-s)}\ud{s},\no
\end{eqnarray}
where $\vv_t\in\r^2$ is a dummy variable.
We may apply Duhamel's principle again to $g(\vx-\e(t_b-s)\vv,\vv_t)$ to obtain
\begin{eqnarray}\label{duhamel}
\
\end{eqnarray}
\begin{eqnarray}
&&g(\vx,\vv)\no\\
&=& \vh(\vv)h(\vx-\e t_b\vv,\vv)\ue^{-\nu(\vv)
t_b}+\int_{0}^{t_b}\vh S(\vx-\e(t_b-s)\vv,\vv)\ue^{-\nu(\vv)
(t_b-s)}\ud{s}\no\\
&&+\int_{0}^{t_b}\bigg(\int_{\r^2}k_{\vh}(\vv,\vv_t)\vh(\vv_t)h(\vx-\e t_b\vv-\e s_b\vv_t,\vv_t)\ue^{-\nu(\vv_t)
s_b}\ud{\vv_t}\bigg)\ue^{-\nu(\vv)
(t_b-s)}\ud{s}\no\\
&&+\int_{0}^{t_b}\int_{\r^2}k_{\vh}(\vv,\vv_t)\bigg(\int_{0}^{s_b}\vh S(\vx-\e(t_b-s)\vv-\e(s_b-r)\vv_t,\vv_t)\ue^{-\nu(\vv_t)
(s_b-r)}\ud{r}\bigg)\ud{\vv_t}\ue^{-\nu(\vv)
(t_b-s)}\ud{s}\no\\
&&+\int_{0}^{t_b}\int_{\r^2}k_{\vh}(\vv,\vv_t)\bigg(\int_{0}^{s_b}\int_{\r^2}k_{\vh}(\vv_t,\vv_s)g(\vx-\e(t_b-s)\vv-\e(s_b-r)\vv_t,\vv_s)\ud{\vv_s}\ue^{-\nu(\vv_t)
(s_b-r)}\ud{r}\bigg)\ud{\vv_t}\ue^{-\nu(\vv)
(t_b-s)}\ud{s}\no,
\end{eqnarray}
where
\begin{eqnarray}
s_b(\vx,\vv,\vv_t,s)&=&\inf\{r>0:\vx-\e(t_b-s)\vv-\e r\vv_t\notin\Omega\},
\end{eqnarray}
and $\vv_s\in\r^2$ is another dummy variable.
We need to estimate each term in (\ref{duhamel}).\\
\ \\
Step 2: Estimates of source terms and boundary terms.\\
We can directly obtain
\begin{eqnarray}
&&\abs{\vh(\vv)h(\vx-\e t_b\vv,\vv)\ue^{-\nu(\vv)
t_b}}\leq C\iss{\vh h}{-},
\end{eqnarray}
\begin{eqnarray}
\abs{\int_{0}^{t_b}\vh S(\vx-\e(t_b-s)\vv,\vv)\ue^{-\nu(\vv)
(t_b-s)}\ud{s}}\leq C\im{\vh S},
\end{eqnarray}
\begin{eqnarray}
&&\abs{\int_{0}^{t_b}\bigg(\int_{\r^2}k_{\vh}(\vv,\vv_t)\vh(\vv_t)h(\vx-\e t_b\vv-\e s_b\vv_t,\vv_t)\ue^{-\nu(\vv_t)
s_b}\ud{\vv_t}\bigg)\ue^{-\nu(\vv)
(t_b-s)}\ud{s}}\\
&\leq&\iss{\vh h}{-}\abs{\int_{0}^{t_b}\bigg(\int_{\r^2}k_{\vh}(\vv,\vv_t)\ue^{-\nu(\vv_t)
s_b}\ud{\vv_t}\bigg)\ue^{-\nu(\vv)
(t_b-s)}\ud{s}}\leq C\iss{\vh h}{-},\no
\end{eqnarray}
\begin{eqnarray}
\\
&&\abs{\int_{0}^{t_b}\int_{\r^2}k_{\vh}(\vv,\vv_t)\bigg(\int_{0}^{s_b}\vh S(\vx-\e(t_b-s)\vv-\e(s_b-r)\vv_t,\vv_t)\ue^{-\nu(\vv_t)
(s_b-r)}\ud{r}\bigg)\ud{\vv_t}\ue^{-\nu(\vv)
(t_b-s)}\ud{s}}\no\\
&\leq&\im{\vh S}\abs{\int_{0}^{t_b}\int_{\r^2}k_{\vh}(\vv,\vv_t)\bigg(\int_{0}^{s_b}\ue^{-\nu(\vv_t)
(s_b-r)}\ud{r}\bigg)\ud{\vv_t}\ue^{-\nu(\vv)
(t_b-s)}\ud{s}}\leq C\im{\vh S}.\no
\end{eqnarray}
\ \\
Step 3: Estimates of $K_{\vh}$ terms.\\
We consider the last and most complicated term $I$ in (\ref{duhamel}) defined as
\begin{eqnarray}
\\
\int_{0}^{t_b}\int_{\r^2}k_{\vh}(\vv,\vv_t)\bigg(\int_{0}^{s_b}\int_{\r^2}k_{\vh}(\vv_t,\vv_s)g(\vx-\e(t_b-s)\vv-\e(s_b-r)\vv_t,\vv_s)\ud{\vv_s}\ue^{-\nu(\vv_t)
(s_b-r)}\ud{r}\bigg)\ud{\vv_t}\ue^{-\nu(\vv)
(t_b-s)}\ud{s}.\no
\end{eqnarray}
We can divide it into four cases:
\begin{eqnarray}
I=I_1+I_2+I_3+I_4.
\end{eqnarray}
\ \\
Case I: $\abs{\vv}\geq N$.\\
Based on Lemma \ref{wellposedness prelim lemma 8} with $\d=0$, we
have
\begin{eqnarray}
\abs{\int_{\r^2}\int_{\r^2}k_{\vh(\vv)}(\vv,\vv_t)k_{\vh(\vv_t)}(\vv_t,\vv_s)\ud{\vv_s}\ud{\vv_t}}\leq\frac{C}{1+\abs{\vv}}\leq\frac{C}{N}.
\end{eqnarray}
Hence, by Fubini's Theorem, we get
\begin{eqnarray}
I_1\leq\frac{C}{N}\im{g}.
\end{eqnarray}
\ \\
Case II: $\abs{\vv}\leq N$, $\abs{\vv_t}\geq2N$, or $\abs{\vv_t}\leq
2N$, $\abs{\vv_s}\geq3N$.\\
Notice this implies either $\abs{\vv_t-\vv}\geq N$ or
$\abs{\vv_t-\vv_s}\geq N$. Hence, either of the following is valid
correspondingly:
\begin{eqnarray}
\abs{k_{\vh(\vv)}(\vv,\vv_t)}&\leq&\ue^{-\frac{\d}{8}N^2}\abs{k_{\vh(\vv)}(\vv,\vv_t)\ue^{\frac{\d}{8}\abs{\vv-\vv_t}^2}},\\
\abs{k_{\vh(\vv_t)}(\vv_t,\vv_s)}&\leq&\ue^{-\frac{\d}{8}N^2}\abs{k_{\vh(\vv_t)}(\vv_t,\vv_s)\ue^{\frac{\d}{8}\abs{\vv_t-\vv_s}^2}}.
\end{eqnarray}
Then based on Lemma \ref{wellposedness prelim lemma 8},
\begin{eqnarray}
\int_{\r^2}\abs{k_{\vh(\vv)}(\vv,\vv_t)\ue^{\frac{\d}{8}\abs{\vv-\vv_t}^2}}\ud{\vv_t}&<&\infty,\\
\int_{\r^2}\abs{k_{\vh(\vv_t)}(\vv_t,\vv_s)\ue^{\frac{\d}{8}\abs{\vv_t-\vv_s}^2}}\ud{\vv_s}&<&\infty.
\end{eqnarray}
Hence, we have
\begin{eqnarray}
I_2\leq C\ue^{-\frac{\d}{8}N^2}\im{g}.
\end{eqnarray}
\ \\
Case III: $s-r\leq\d$ and $\abs{\vv}\leq N$, $\abs{\vv_t}\leq 2N$, $\abs{\vv_s}\leq 3N$.\\
In this case, when $0<\d<<1$ is sufficiently small, since the integral in $r$ is restricted to this short interval, we have
\begin{eqnarray}
I_3\leq C\d\im{g}.
\end{eqnarray}
\ \\
Case IV: $s-r\geq\d$ and $\abs{\vv}\leq N$, $\abs{\vv_t}\leq 2N$, $\abs{\vv_s}\leq 3N$.\\
Since $k_{\vh(\vv)}(\vv,\vv_t)$ has
possible integrable singularity of $1/\abs{\vv-\vv_t}$, we can
introduce $k_N(\vv,\vv_t)$ smooth with compact support such that
\begin{eqnarray}\label{wt 42}
\sup_{\abs{p}\leq 3N}\int_{\abs{\vv_t}\leq
3N}\abs{k_N(p,\vv_t)-k_{\vh(p)}(p,\vv_t)}\ud{\vv_t}\leq\frac{1}{N}.
\end{eqnarray}
Then we can split
\begin{eqnarray}\label{wt 41}
k_{\vh(\vv)}(\vv,\vv_t)k_{\vh(\vv_t)}(\vv_t,\vv_s)
&=&k_N(\vv,\vv_t)k_N(\vv_t,\vv_s)\\
&&+\bigg(k_{\vh(\vv)}(\vv,\vv_t)-k_N(\vv,\vv_t)\bigg)k_{\vh(\vv_t)}(\vv_t,\vv_s)\no\\
&&+\bigg(k_{\vh(\vv_t)}(\vv_t,\vv_s)-k_N(\vv_t,\vv_s)\bigg)k_N(\vv,\vv_t).\no
\end{eqnarray}
This means we further split $I_4$ into
\begin{eqnarray}
I_4=I_{4,1}+I_{4,2}+I_{4,3}.
\end{eqnarray}
Based on the estimate (\ref{wt 42}), we have
\begin{eqnarray}
I_{4,2}&\leq&\frac{C}{N}\im{g},\\
I_{4,3}&\leq&\frac{C}{N}\im{g}.
\end{eqnarray}
Therefore, the only remaining term is $I_{4,1}$. Note we always have
$\vx-\e(t_b-s)\vv-\e(s-r)\vv_t\in\Omega$. Hence, we define the change of
variable $\vec y=(y_1,y_2)=\vx-\e(t_b-s)\vv-\e(s-r)\vv_t$ such that
\begin{eqnarray}
\abs{\frac{\ud{\vec y}}{\ud{\vv_t}}}=\abs{\left\vert\begin{array}{cc}
\e(s-r)&0\\
0&\e(s-r)
\end{array}\right\vert}=\e^2(s-r)\geq \e^2\d^2.
\end{eqnarray}
Since $k_N$ is bounded and $\abs{\vv_s}\leq 3N$, we estimate
\begin{eqnarray}
\\
I_{4,1}
&\leq&
C\abs{\int_{\abs{\vv_t}\leq2N}\int_{\abs{\vv_s}\leq3N}\int_{0}^{s}
\id_{\{\vx-\e(t_b-s)\vv-\e(s-r)\vv_t\in\Omega\}}g(\vx-\e(t_b-s)\vv-\e(s-r)\vv_t,\vv_s)\ud{r}\ud{\vv_t}\ud{\vv_s}}\no\\
&\leq&\frac{C}{\e\d}\abs{\int_{0}^{s}\int_{\abs{\vv_s}\leq3N}\int_{\Omega}\id_{\{\vec
y\in\Omega\}}g(\vec y,\vv_s)\ud{\vec y}\ud{\vv_s}\ud{r}}\no\\
&\leq&\frac{C\left(1+N^2\right)^{\frac{\vt}{2}}\ue^{\zeta N^2}}{\e\d}\tm{\frac{g(\vec y,\vv_s)}{\vh(\vv_s)}}\no\\
&=&\frac{C\left(1+N^2\right)^{\frac{\vt}{2}}\ue^{\zeta N^2}}{\e\d}\tm{\frac{g}{\vh}}.\no
\end{eqnarray}
Summarize all above in Case IV, we obtain
\begin{eqnarray}
I_4\leq \frac{C}{N}\im{g}+\frac{C\left(1+N^2\right)^{\frac{\vt}{2}}\ue^{\zeta N^2}}{\e\d}\tm{\frac{g}{\vh}}.
\end{eqnarray}
Therefore, we already prove
\begin{eqnarray}
I\leq
\bigg(C\ue^{-\frac{\d}{8}N^2}+\frac{C}{N}+C\d\bigg)\im{g}+\frac{C\left(1+N^2\right)^{\frac{\vt}{2}}\ue^{\zeta N^2}}{\e\d}\tm{\frac{g}{\vh}}
\end{eqnarray}
\ \\
Step 4: Synthesis.\\
Collecting all above in Step 2 and Step 3, we
have shown
\begin{eqnarray}
\im{g}\leq \bigg(C\ue^{-\frac{\d}{8}N^2}+\frac{C}{N}+C\d\bigg)\im{g}+\frac{C\left(1+N^2\right)^{\frac{\vt}{2}}\ue^{\zeta N^2}}{\e\d}\tm{\frac{g}{\vh}}+C\im{\vh
S}+C\iss{\vh h}{-}.
\end{eqnarray}
Choosing $\d$ sufficiently small and then taking $N$ sufficiently
large such that $C\ue^{-\frac{\d}{8}N^2}+\dfrac{C}{N}+C\d\leq \dfrac{1}{2}$, we have
\begin{eqnarray}
\im{g}\leq \frac{C}{\e}\tm{\frac{g}{\vh}}+C\im{\vh
S}+C\iss{\vh h}{-}.
\end{eqnarray}
Based on Theorem \ref{wellposedness LT estimate}, we obtain
\begin{eqnarray}
\im{g}&\leq&C\bigg(\frac{1}{\e}\tm{\frac{g}{\vh}}+\im{\vh S}+\iss{\vh
h}{-}\bigg)\\
&=&C\bigg(\frac{1}{\e}\tm{f}+\im{\vh S}+\iss{\vh
h}{-}\bigg)\no\\
&\leq&C\bigg(\frac{1}{\e^3}\tm{(\ik-\pk)[S]}+\frac{1}{\e^2}\tm{\pk[S]}+\frac{1}{\e^{3/2}}\tss{h}{-}+\im{\vh
S}+\iss{\vh h}{-}\bigg).\no
\end{eqnarray}
This completes the proof.
\end{proof}

\section{Asymptotic Analysis}

In this section, we construct the asymptotic expansion of the equation (\ref{small system}).

\subsection{Interior Expansion}

We define the interior expansion
\begin{eqnarray}\label{interior expansion}
\f\sim\sum_{k=1}^{\infty}\e^{k}\f_k(\vx,\vv).
\end{eqnarray}
Plugging it into the equation (\ref{small system}) and comparing the order of $\e$, we obtain
\begin{eqnarray}
\ll[\f_1]&=&0,\label{interior expansion 1}\\
\ll[\f_2]&=&-\vv\cdot\nx\f_1+\Gamma[\f_1,\f_1],\label{interior expansion 2}\\
\ll[\f_3]&=&-\vv\cdot\nx\f_2+\Gamma[\f_1,\f_2]+\Gamma[\f_2,\f_1],\label{interior expansion 3}\\
\ldots\no\\
\ll[\f_k]&=&-\vv\cdot\nx\f_{k-1}+\sum_{i=1}^{k-1}\Gamma[\f_i,\f_{k-i}].\label{interior
expansion 4}
\end{eqnarray}
The solvability of
\begin{eqnarray}
\ll[\f_k]&=&-\vv\cdot\nx\f_{k-1}+\sum_{i=1}^{k-1}\Gamma[\f_i,\f_{k-i}]
\end{eqnarray}
requires
\begin{eqnarray}
\int_{\r^2}\bigg(-\vv\cdot\nx\f_{k-1}+\sum_{i=1}^{k-1}\Gamma[\f_i,\f_{k-i}]\bigg)\psi(\vv)\ud{\vv}=0
\end{eqnarray}
for any $\psi$ satisfying $\ll[\psi]=0$.
Based on the analysis in \cite{Sone2002, Sone2007}, each $\fc_k$ consists of three parts:
\begin{eqnarray}
\f_k(\vx,\vv)=A_k^{\e}(\vx,\vv)+B_k^{\e}(\vx,\vv)+C_k^{\e}(\vx,\vv),
\end{eqnarray}
where
\begin{eqnarray}
A_k^{\e}(\vx,\vv)=\sqrt{\m}\left(A_{k,0}^{\e}(\vx)+A_{k,1}^{\e}(\vx)\v_1+A_{k,2}^{\e}(\vx)\v_2+A_{k,3}^{\e}(\vx)\left(\frac{\abs{\vv}^2-1}{2}\right)\right),
\end{eqnarray}
is the fluid part,
\begin{eqnarray}
B_k^{\e}(\vx,\vv)=\sqrt{\m}\left(B_{k,0}^{\e}(\vx)+B_{k,1}^{\e}(\vx)\v_1+B_{k,2}^{\e}(\vx)\v_2+B_{k,3}^{\e}(\vx)\left(\frac{\abs{\vv}^2-1}{2}\right)\right),
\end{eqnarray}
with $B_{k}^{\e}$ depending on $A_{s,i}^{\e}$ for $1\leq s\leq k-1$ and $i=0,1,2,3$ as
\begin{eqnarray}\label{at 12}
B_{k,0}^{\e}&=&0,\\
B_{k,1}^{\e}&=&\sum_{i=1}^{k-1}A_{i,0}^{\e}A_{k-i,1}^{\e},\\
B_{k,2}^{\e}&=&\sum_{i=1}^{k-1}A_{i,0}^{\e}A_{k-i,2}^{\e},\\
B_{k,3}^{\e}&=&\sum_{i=1}^{k-1}\bigg(A_{i,0}^{\e}A_{k-i,3}^{\e}+A_{i,1}^{\e}A_{k-i,1}^{\e}+A_{i,2}^{\e}A_{k-i,2}^{\e}
+\sum_{j=1}^{k-1-i}A_{i,0}^{\e}(A_{j,1}^{\e}A_{k-i-j,1}^{\e}+A_{j,2}^{\e}A_{k-i-j,2}^{\e})\bigg),
\end{eqnarray}
and $C_k(\vx,\vv)$ satisfies
\begin{eqnarray}
\int_{\r^2}\sqrt{\m}C_k^{\e}(\vx,\vv)\left(\begin{array}{c}1\\\vv\\\abs{\vv}^2
\end{array}\right)\ud{\vv}=0,
\end{eqnarray}
with
\begin{eqnarray}\label{at 13}
\ll[C_k^{\e}]&=&-\vv\cdot\nx\f_{k-1}+\sum_{i=1}^{k-1}\Gamma[\f_i,\f_{k-i}],
\end{eqnarray}
which can be solved explicitly at any fixed $\vx$. Hence, we only need to determine the
relations satisfied by $A_k^{\e}$. For convenience, we define
\begin{eqnarray}
A_k^{\e}=\sqrt{\m}\left(\rh_k^{\e}+\u_{k,1}^{\e}\v_1+\u_{k,2}^{\e}\v_2+\th_k^{\e}\left(\frac{\abs{\vv}^2-1}{2}\right)\right),
\end{eqnarray}
Then the analysis in \cite{Sone2002, Sone2007} shows that $A_k^{\e}$ satisfies the equations as follows:\\
\ \\
$0^{th}$ order equations:
\begin{eqnarray}
P_1^{\e}-(\rh_1^{\e}+\th_1^{\e})&=&0,\\
\nx P_1^{\e}&=&0,
\end{eqnarray}
$1^{st}$ order equations:
\begin{eqnarray}
P_2^{\e}-(\rh_2^{\e}+\th_2^{\e}+\rh_1^{\e}\th_1^{\e})&=&0,\\
\uh\cdot\nx\uh_1^{\e}-\gamma_1\dx\uh_1^{\e}+\nx P_2^{\e}&=&0,\\
\nx\cdot\uh_1^{\e}&=&0,\\
\uh_1^{\e}\cdot\nx\th_1^{\e}-\gamma_2\dx\th_1^{\e}&=&0,
\end{eqnarray}
$k^{th}$ order equations:
\begin{eqnarray}
P_{k+1}^{\e}-\left(\rh_{k+1}^{\e}+\th_{k+1}^{\e}+\sum_{i=1}^{k+1-i}\rh_i^{\e}\th_{k+1-i}^{\e}\right)&=&0,\\
\sum_{i=1}^{k}\uh_i^{\e}\cdot\nx\uh_{k+1-i}^{\e}-\gamma_1\dx\uh_k^{\e}+\nx P_{k+1}^{\e}&=&H_{k,1}^{\e},\\
\nx\cdot\uh_k^{\e}&=&H_{k,2}^{\e},\\
\sum_{i=1}^{k}\uh_i^{\e}\cdot\nx\th_{k+1-i}^{\e}-\gamma_2\dx\th_k^{\e}&=&H_{k,3}^{\e},
\end{eqnarray}
where
\begin{eqnarray}
\\
H_{k,j}^{\e}=G_{k,j}^{\e}[\vx,\vv; \rh_1^{\e},\ldots,\rh_{k-1}^{\e};
\th_1^{\e},\ldots,\th_{k-1}^{\e}; \uh_1,\ldots,\uh_{k-1}^{\e}]\no,
\end{eqnarray}
is explicit functions depending on lower order terms, and $\gamma_1$ and $\gamma_2$ are two positive constants. Since in most cases, we are only interested in the leading order terms, so we omit the detailed description of $G_{k,j}^{\e}$. In order to determine the boundary condition for $\uh_k^{\e}$, $\th_k^{\e}$ and $\rh_k^{\e}$, we have to define the boundary layer expansion.

\subsection{Boundary Layer Expansion with Geometric Correction}

In order to define the boundary layer, we need several substitutions:\\
\ \\
Substitution 1:\\
Define the substitution into polar coordinate $f^{\e}(x_1,x_2,\vv)\rt f^{\e}(\rr,\ph,\vv)$ with
$(\rr,\ph,\vv)\in [0,1)\times[-\pi,\pi)\times\r^2$ as
\begin{eqnarray}\label{substitution 1}
\left\{
\begin{array}{rcl}
x_1&=&\rr\cos\ph,\\
x_2&=&\rr\sin\ph,\\
\vv&=&\vv.
\end{array}
\right.
\end{eqnarray}
The equation (\ref{small system}) can be rewritten as
\begin{eqnarray}
\left\{
\begin{array}{rcl}
\e(\vv\cdot\vn)\dfrac{\p f^{\e}}{\p\rr}+\dfrac{\e}{\rr}(\vv\cdot\ta)\dfrac{\p
f^{\e}}{\p\ph}+\ll[f^{\e}]=\Gamma[f^{\e},f^{\e}],\\\rule{0ex}{1.0em}
f^{\e}(1,\ph,\vv)=\b^{\e}(1,\ph,\vv) \ \ \text{for}\ \
\vv\cdot\vn<0,
\end{array}
\right.
\end{eqnarray}
for $\vn$ the outer normal vector and $\ta$ the counterclockwise tangential vector on $\p\Omega$.\\
\ \\
Substitution 2:\\
We further perform the scaling substitution $f^{\e}(\rr,\ph,\vv)\rt f^{\e}(\et,\ph,\vv)$ with
$(\et,\ph,\vv)\in [0,1/\e)\times[-\pi,\pi)\times\r^2$ as
\begin{eqnarray}\label{substitution 2}
\left\{
\begin{array}{rcl}
\et&=&(1-\rr)/\e,\\
\ph&=&\ph,\\
\vv&=&\vv,
\end{array}
\right.
\end{eqnarray}
which implies
\begin{eqnarray}
\frac{\p f^{\e}}{\p\rr}&=&-\frac{1}{\e}\frac{\p f^{\e}}{\p\et},
\end{eqnarray}
Then the equation (\ref{small system}) in $(\et,\ph,\vv)$
becomes
\begin{eqnarray}
\left\{
\begin{array}{rcl}
-(\vv\cdot\vn)\dfrac{\p f^{\e}}{\p\et}+\dfrac{\e}{1-\e\et}(\vv\cdot\ta)\dfrac{\p
f^{\e}}{\p\ph}+\ll[f^{\e}]=\Gamma[f^{\e},f^{\e}],\\\rule{0ex}{1.0em}
f^{\e}(0,\ph,\vv)=\b^{\e}(0,\ph,\vv) \ \ \text{for}\ \
\vv\cdot\vn<0.
\end{array}
\right.
\end{eqnarray}
\ \\
Substitution 3:\\
We define the velocity substitution $f^{\e}(\rr,\ph,\vv)\rt f^{\e}(\et,\ph,\vr)$ with
$(\et,\ph,\vr)\in [0,1/\e)\times[-\pi,\pi)\times\r^2$ as
\begin{eqnarray}\label{substitution 3}
\left\{
\begin{array}{rcl}
\et&=&\et,\\
\ph&=&\ph,\\
\vr&=&-\vv.
\end{array}
\right.
\end{eqnarray}
This substitution is for the convenience of Milne problem and specifying the in-flow boundary. Then the equation (\ref{small system}) in $(\et,\ph,\vv)$
becomes
\begin{eqnarray}
\left\{
\begin{array}{rcl}
(\vr\cdot\vn)\dfrac{\p f^{\e}}{\p\et}-\dfrac{\e}{1-\e\et}(\vr\cdot\ta)\dfrac{\p
f^{\e}}{\p\ph}+\ll[f^{\e}]=\Gamma[f^{\e},f^{\e}],\\\rule{0ex}{1.0em}
f^{\e}(0,\ph,\vr)=\b^{\e}(0,\ph,\vr) \ \ \text{for}\ \
\vr\cdot\vn>0.
\end{array}
\right.
\end{eqnarray}
\ \\
Substitution 4:\\
We further define the velocity decomposition with respect to the normal and tangential directions at boundary as $f^{\e}(\et,\ph,\v_{r,1},\v_{r,2})\rt f^{\e}(\et,\ph,\ve,\vp)$ with
$(\et,\ph,\ve,\vp)\in [0,1/\e)\times[-\pi,\pi)\times\r^2$
\begin{eqnarray}\label{substitution 4}
\left\{
\begin{array}{rcl}
\et&=&\et,\\
\ph&=&\ph,\\
\ve&=&\v_{r,1}\cos\phi+\v_{r,2}\sin\phi,\\
\vp&=&-\v_{r,1}\sin\phi+\v_{r,2}\cos\phi.
\end{array}
\right.
\end{eqnarray}
Denote $\vvv=(\ve,\vp)$. Then the equation (\ref{small system}) can be rewritten as
\begin{eqnarray}\label{boundary layer system}
\left\{
\begin{array}{rcl}
\ve\dfrac{\p f^{\e}}{\p\et}-\dfrac{\e}{1-\e\et}\bigg(\vp\dfrac{\p
f^{\e}}{\p\ph}+\vp^2\dfrac{\p f^{\e}}{\p\ve}-\ve\vp\dfrac{\p
f^{\e}}{\p\vp}\bigg)+\ll[f^{\e}]=\Gamma[f^{\e},f^{\e}],\\\rule{0ex}{1.5em}
f^{\e}(0,\ph,\vvv)=\b^{\e}(\ph,\vvv) \ \ \text{for}\ \
\ve>0.
\end{array}
\right.
\end{eqnarray}
We define the boundary layer expansion
\begin{eqnarray}\label{boundary layer expansion}
\fb\sim\sum_{k=1}^{\infty}\e^k\fb_k(\et,\ph,\vvv),
\end{eqnarray}
where $\fb_k$ can be determined by plugging it into the equation (\ref{boundary layer system}) and comparing the order of $\e$.
In a neighborhood of the boundary, we have
\begin{eqnarray}
\label{boundary expansion 1} \ve\frac{\p
\fb_1}{\p\et}-\frac{\e}{1-\e\eta}\bigg(\vp^2\dfrac{\p
\fb_1}{\p\ve}-\ve\vp\dfrac{\p
\fb_1}{\p\vp}\bigg)+\ll[\fb_1]
&=&0,\\
\no\\
\label{boundary expansion 2}\ve\frac{\p
\fb_2}{\p\et}-\frac{\e}{1-\e\eta}\bigg(\vp^2\dfrac{\p
\fb_2}{\p\ve}-\ve\vp\dfrac{\p
\fb_2}{\p\vp}\bigg)+\ll[\fb_2]
&=&\frac{1}{1-\e\eta}\vp\dfrac{\p \fb_1}{\p\ph}\\
&&+\Gamma[\fb_1,\fb_1]+2\Gamma[\f_1,\fb_1],
\no\\
\ldots\no\\
\no\\
\label{boundary expansion 3}\ve\frac{\p
\fb_k}{\p\et}-\frac{\e}{1-\e\eta}\bigg(\vp^2\dfrac{\p
\fb_k}{\p\ve}-\ve\vp\dfrac{\p
\fb_k}{\p\vp}\bigg)+\ll[\fb_k]
&=&\frac{1}{1-\e\eta}\vp\dfrac{\p \fb_{k-1}}{\p\ph}\\
&&+\sum_{i=1}^{k-1}\Gamma[\fb_i,\fb_{k-i}]+2\sum_{i=1}^{k-1}\Gamma[\f_i,\fb_{k-i}].\no
\end{eqnarray}

\subsection{Construction of Asymptotic Expansion}

The bridge between interior solution and boundary layer is the boundary condition
\begin{eqnarray}
f^{\e}(\vx_0,\vv)&=&\b^{\e}(\vx_0,\vv).
\end{eqnarray}
Plugging the combined expansion
\begin{eqnarray}
f^{\e}\sim\sum_{k=1}^{\infty}\e^k(\f_k+\fb_k),
\end{eqnarray}
into the boundary condition and comparing the order of $\e$, we obtain
\begin{eqnarray}
\f_1+\fb_1&=&\b_1,\\
\f_2+\fb_2&=&\b_2,\\
\ldots\no\\
\f_k+\fb_k&=&\b_k.
\end{eqnarray}
This is the boundary conditions $\f_k$ and $\fb_k$ need to satisfy.
We divide the construction of asymptotic expansion into several steps for each $k\geq1$:\\
\ \\
Step 1: $\e$-Milne Problem.\\
We solve the $\e$-Milne problem
\begin{eqnarray}
\left\{
\begin{array}{rcl}\displaystyle
\ve\frac{\p \g_k^{\e}}{\p\et}+G(\e;\et)\bigg(\vp^2\dfrac{\p
\g_k^{\e}}{\p\ve}-\ve\vp\dfrac{\p \g_k^{\e}}{\p\vp}\bigg)+\ll[\g_k^{\e}]
&=&S_k^{\e}(\et,\ph,\vvv),\\
\g_k^{\e}(0,\ph,\vvv)&=&h_k^{\e}(\ph,\vvv)\ \ \text{for}\ \
\ve>0,\\\rule{0ex}{1.0em} \displaystyle\int_{\r^2}
\ve\sqrt{\m}\g_k^{\e}(0,\ph,\vvv)\ud{\vvv}
&=&m_f[\g_k^{\e}](\ph),\\\rule{0ex}{1.0em}
\displaystyle\lim_{\et\rt\infty}\g_k^{\e}(\et,\ph,\vvv)&=&\g^{\e}_{k}(\infty,\ph,\vvv),
\end{array}
\right.
\end{eqnarray}
for $\g_k^{\e}(\et,\ph,\vvv)$ with the in-flow boundary data
\begin{eqnarray}
h^{\e}_k&=&\b_k-(B_k^{\e}+C_k^{\e})
\end{eqnarray}
and source term
\begin{eqnarray}
S^{\e}_k&=&\frac{\Upsilon(\e^{1/2}\et)}{1-\e\et}\vp\dfrac{\p
\fb_{k-1}}{\p\ph}
+\sum_{i=1}^{k-1}\Gamma[\fb_i,\fb_{k-i}]+2\sum_{i=1}^{k-1}\Gamma[\f_i,\fb_{k-i}],
\end{eqnarray}
where
\begin{eqnarray}
G(\e;\et)&=&-\frac{\e\Upsilon(\e^{1/2}\et)}{1-\e\et},
\end{eqnarray}
\begin{eqnarray}\label{cutoff 1}
\Upsilon(z)=\left\{
\begin{array}{ll}
1&0\leq z\leq1/2,\\
0&3/4\leq z\leq\infty.
\end{array}
\right.
\end{eqnarray}
Here the mass-flux $m_f[\g_k^{\e}](\ph)$ will be determined later. Based on Theorem \ref{Milne theorem 3}, there exist
\begin{eqnarray}
\tilde h^{\e}_k=\sqrt{\m}\bigg(\tilde
D_{k,0}^{\e}+\tilde D_{k,1}^{\e}\ve+\tilde D_{k,2}^{\e}\vp+\tilde
D_{k,3}^{\e}\bigg(\frac{\abs{\vvv}^2-2}{2}\bigg)\bigg),
\end{eqnarray}
such that the problem
\begin{eqnarray}\label{at 11}
\left\{
\begin{array}{rcl}\displaystyle
\ve\frac{\p \gg_k^{\e}}{\p\et}+G(\e;\et)\bigg(\vp^2\dfrac{\p
\gg_k^{\e}}{\p\ve}-\ve\vp\dfrac{\p
\gg_k^{\e}}{\p\vp}\bigg)+\ll[\gg_k^{\e}]
&=&S_k^{\e}(\et,\ph,\vvv),\\
\gg_k^{\e}(0,\ph,\vvv)&=&h_k^{\e}(\ph,\vvv)-\tilde h_k^{\e}(\ph,\vvv)\ \
\text{for}\ \ \ve>0,\\\rule{0ex}{1.0em}
\displaystyle\int_{\r^2}\ve\sqrt{\m}
\gg_k^{\e}(0,\ph,\vvv)\ud{\vvv}
&=&m_f[\gg_k^{\e}](\phi),\\\rule{0ex}{1.0em}
\displaystyle\lim_{\et\rt\infty}\gg_k^{\e}(\et,\ph,\vvv)&=&0,
\end{array}
\right.
\end{eqnarray}
is well-posed, where we need to specify
\begin{eqnarray}
m_f[\gg_k^{\e}](\phi)=m_f[\g_k^{\e}](\phi)-\displaystyle\int_{\r^2}\ve\sqrt{\m}\tilde h_k^{\e}(\ph,\vvv)\ud{\vvv}.
\end{eqnarray}
\ \\
Step 2: Definition of Interior Solution and Boundary Layer with Geometric Correction.\\
Define
\begin{eqnarray}
\fb_k=\gg_k^{\e}\cdot\Upsilon_0(\e^{1/2}\et)
\end{eqnarray}
where $\gg_k^{\e}$ the solution of $\e$-Milne problem (\ref{at 11}) and
\begin{eqnarray}\label{cutoff 2}
\Upsilon_0(z)=\left\{
\begin{array}{ll}
1&0\leq z\leq1/4,\\
0&1/2\leq z\leq\infty.
\end{array}
\right.
\end{eqnarray}
Naturally, we have
\begin{eqnarray}
\lim_{\et\rt0}\fb_k(\et,\ph,\vvv)=0.
\end{eqnarray}
The interior solution
\begin{eqnarray}
\f_k(\vx,\vv)=A_k^{\e}(\vx,\vv)+B_k^{\e}(\vx,\vv)+C_k^{\e}(\vx,\vv),
\end{eqnarray}
where $B_k^{\e}$ and $C_k^{\e}$ are defined in (\ref{at 12}) and (\ref{at 13}), and $A_k^{\e}$ satisfies
\begin{eqnarray}
A_k^{\e}=\sqrt{\m}\left(\rh_k^{\e}+\u_{k,1}^{\e}\v_1+\u_{k,2}^{\e}\v_2+\th_k^{\e}\left(\frac{\abs{\vv}^2-1}{2}\right)\right),
\end{eqnarray}
and
\begin{eqnarray}\label{at 14}
P_{k+1}^{\e}-\left(\rh_{k+1}^{\e}+\th_{k+1}^{\e}+\sum_{i=1}^{k+1-i}\rh_i^{\e}\th_{k+1-i}^{\e}\right)&=&0,\\
\sum_{i=1}^{k}\uh_i^{\e}\cdot\nx\uh_{k+1-i}^{\e}-\gamma_1\dx\uh_k^{\e}+\nx P_{k+1}^{\e}&=&H_{k,1}^{\e},\\
\nx\cdot\uh_k^{\e}&=&H_{k,2}^{\e},\\
\sum_{i=1}^{k}\uh_i^{\e}\cdot\nx\th_{k+1-i}^{\e}-\gamma_2\dx\th_k^{\e}&=&H_{k,3}^{\e},
\end{eqnarray}
with boundary condition
\begin{eqnarray}
A_{k,0}^{\e}&=&\tilde D_{k,0}^{\e},\\
A_{k,1}^{\e}&=&-\tilde D_{k,1}^{\e}\cos\ph+\tilde D_{k,2}^{\e}\sin\ph,\\
A_{k,2}^{\e}&=&-\tilde D_{k,1}^{\e}\sin\ph-\tilde D_{k,2}^{\e}\cos\ph,\\
A_{k,3}^{\e}&=&\tilde D_{k,3}^{\e}.
\end{eqnarray}
where $\tilde D_{k,i}^{\e}$ comes from the boundary data of $\e$-Milne problem $\tilde h_k^{\e}$.
This determines $A_{k,0}^{\e}$, $A_{k,1}^{\e}$, $A_{k,2}^{\e}$ and $A_{k,3}^{\e}$. Now it is easy to verify the boundary data are satisfied as
\begin{eqnarray}
\f_k+\fb_k&=&\b_k.
\end{eqnarray}
\ \\
Step 3: Boussinesq relation and Vanishing Mass-Flux.\\
Note that the fluid-type equations satisfied by $A_k^{\e}$ imply Boussinesq relation. In detail,
\begin{eqnarray}
P_{k+1}^{\e}-\left(\rh_{k+1}^{\e}+\th_{k+1}^{\e}+\sum_{i=1}^{k+1-i}\rh_i^{\e}\th_{k+1-i}^{\e}\right)&=&0
\end{eqnarray}
yields
\begin{eqnarray}
\rh_{k}^{\e}+\th_{k}^{\e}+\sum_{i=1}^{k-i}\rh_i^{\e}\th_{k-i}^{\e}=E_k
\end{eqnarray}
which is actually
\begin{eqnarray}\label{Boussinesq}
\rh_{k}^{\e}+\th_{k}^{\e}=E_k-\sum_{i=1}^{k-i}\rh_i^{\e}\th_{k-i}^{\e}
\end{eqnarray}
for some constant $E_k$ which is free to choose. To enforce this relation, we need to adjust the mass-flux in the $\e$-Milne problem (\ref{at 11}). Note that the Boussinesq relation (\ref{Boussinesq}) leads to a given $\tilde D_{k,0}^{\e}(\phi)+\tilde D_{k,3}^{\e}(\phi)$ for any $\ph$ up to a constant in the $\e$-Milne problem. Theorem \ref{Milne adjustment} implies we can always adjust the mass-flux $m_f[\gg_k^{\e}](\phi)$ to guarantee the Boussinesq relation. Based on the proof of Theorem \ref{Milne adjustment}, we know this can determine the mass-flux $m_f[\gg_{k}^{\e}](\ph)$ up to a constant.

On the other hand, $f^{\e}(\vx,\vv)$ satisfies the vanishing mass-flux
\begin{eqnarray}
\int_{\p\Omega}\int_{\r^2}f^{\e}(\vx,\vv)\ud{\vv}\ud{\gamma}=0.
\end{eqnarray}
Then we need
\begin{eqnarray}
\int_{\p\Omega}\int_{\r^2}(\f_k+\fb_k)(\vx,\vv)\ud{\vv}\ud{\gamma}=0.
\end{eqnarray}
for any $k\geq1$. The definition of $\fb_k$ implies
\begin{eqnarray}
\int_{\p\Omega}\int_{\r^2}\fb_k(\vx,\vv)\ud{\vv}\ud{\gamma}=0.
\end{eqnarray}
Then the remaining relation
\begin{eqnarray}
\int_{\p\Omega}\int_{\r^2}\f_k(\vx,\vv)\ud{\vv}\ud{\gamma}=0.
\end{eqnarray}
finally determines the free constant in mass-flux $m_f[\gg_{k}^{\e}](\ph)$.

In summary, the free mass-flux $m_f[\gg_{k}^{\e}](\ph)$ can help to enforce two relations: the Boussinesq relation
\begin{eqnarray}
\rh_{k}^{\e}+\th_{k}^{\e}=E_k-\sum_{i=1}^{k-i}\rh_i^{\e}\th_{k-i}^{\e},
\end{eqnarray}
and vanishing mass-flux relation
\begin{eqnarray}
\int_{\p\Omega}\int_{\r^2}\f_k(\vx,\vv)\ud{\vv}\ud{\gamma}=0.
\end{eqnarray}
Therefore, $m_f[\gg_{k}^{\e}](\ph)$ is completed determined and so are $\f_k$ and $\fb_k$.

In particular, when $k=1$, $\fb_1$ satisfies
\begin{eqnarray}\label{at 03}
\left\{
\begin{array}{rcl}
\fb_1(\et,\ph,\vvv)&=&\gg_1^{\e}(\et,\ph,\vvv)\cdot\Upsilon_0(\sqrt{\e\et})\\
\ve\dfrac{\p \gg^{\e}_1}{\p\et}+G(\e;\et)\bigg(\vp^2\dfrac{\p
\gg^{\e}_1}{\p\ve}-\ve\vp\dfrac{\p
\gg^{\e}_1}{\p\vp}\bigg)+\ll[\gg^{\e}_1]
&=&0,\\
\gg^{\e}_1(0,\ph,\vvv)&=&\b_1(\ph,\vvv)-\tilde h^{\e}(\ph,\vvv)\ \
\text{for}\ \ \ve>0,\\\rule{0ex}{1em}
\displaystyle\int_{\r^2}\ve\sqrt{\m}
\gg^{\e}_1(0,\ph,\vvv)\ud{\vvv}
&=&m_f[\gg^{\e}_1],\\\rule{0ex}{1.0em}
\displaystyle\lim_{\et\rt\infty}\gg^{\e}_1(\et,\ph,\vvv)&=&0,
\end{array}
\right.
\end{eqnarray}
where
\begin{eqnarray}
\tilde h_1^{\e}&=&\sqrt{\m(\vv)}\bigg(\tilde
D_{1,0}^{\e}+\tilde D_{1,1}^{\e}\ve+\tilde D_{1,2}^{\e}\vp+\tilde
D_{1,3}^{\e}\bigg(\frac{\abs{\vv}^2-2}{2}\bigg)\bigg),
\end{eqnarray}
and $\f_1$ satisfies
\begin{eqnarray}\label{at 04}
\f_1=\sqrt{\m}\bigg(\rh_1^{\e}+\u_{1,1}^{\e}\v_1+\u_{1,2}^{\e}\v_2+\th_1^{\e}\bigg(\frac{\abs{\vv}^2-2}{2}\bigg)\bigg),
\end{eqnarray}
with
\begin{eqnarray}
\left\{
\begin{array}{rcl}
\nx(\rh_1^{\e}+\th_1^{\e})&=&0,\\\rule{0ex}{1.0em}
\uh^{\e}_1\cdot\nx\uh_1^{\e}-\gamma_1\dx\uh_1^{\e}+\nx P_2^{\e}&=&0,\\\rule{0ex}{1.0em}
\nx\cdot\uh_1^{\e}&=&0,\\\rule{0ex}{1.0em}
\uh_1^{\e}\cdot\nx\th_1^{\e}-\gamma_2\dx\th_1^{\e}&=&0,\\\rule{0ex}{1.0em}
\rh_{1}^{\e}(\vx_0)&=&\tilde D_{1,0}^{\e}(\vx_0),\\\rule{0ex}{1.0em}
\u_{1,1}^{\e}(\vx_0)&=&-\tilde D_{1,1}^{\e}(\vx_0)\cos\ph+\tilde D_{1,2}^{\e}(\vx_0)\sin\ph,\\\rule{0ex}{1.0em}
\u_{1,2}^{\e}(\vx_0)&=&-\tilde D_{1,1}^{\e}(\vx_0)\sin\ph-\tilde D_{1,2}^{\e}(\vx_0)\cos\ph,\\\rule{0ex}{1.0em}
\th_{1}^{\e}(\vx_0)&=&\tilde D_{1,3}^{\e}(\vx_0),
\end{array}
\right.
\end{eqnarray}
where the free mass-flux $m_f[\gg_{1}^{\e}](\ph)$ is chosen to enforce the Boussinesq relation
\begin{eqnarray}
\rh_{1}^{\e}+\th_{1}^{\e}=E_1,
\end{eqnarray}
and vanishing mass-flux relation
\begin{eqnarray}
\int_{\p\Omega}\int_{\r^2}\f_1(\vx,\vv)\ud{\vv}\ud{\gamma}=0.
\end{eqnarray}
Similarly, we can define any $\f_k$ and $\fb_k$ for $k\geq1$.

\section{$\e$-Milne Problem with Geometric Correction}

We consider the $\e$-Milne problem for $\g^{\e}(\et,\ph,\vvv)$ in
the domain
$(\et,\ph,\vvv)\in[0,\infty)\times[-\pi,\pi)\times\r^2$ as
\begin{eqnarray}\label{Milne}
\left\{
\begin{array}{rcl}\displaystyle
\ve\frac{\p \g^{\e}}{\p\et}+G(\e;\et)\bigg(\vp^2\dfrac{\p
\g^{\e}}{\p\ve}-\ve\vp\dfrac{\p \g^{\e}}{\p\vp}\bigg)+\ll[\g^{\e}]
&=&S^{\e}(\et,\ph,\vvv),\\
\g^{\e}(0,\ph,\vvv)&=&h^{\e}(\ph,\vvv)\ \ \text{for}\ \
\ve>0,\\\rule{0ex}{1.0em} \displaystyle\int_{\r^2}
\ve\sqrt{\m}\g^{\e}(0,\ph,\vvv)\ud{\vvv}
&=&m_f[\g^{\e}](\ph),\\\rule{0ex}{1.0em}
\displaystyle\lim_{\et\rt\infty}\g^{\e}(\et,\ph,\vvv)&=&\g^{\e}_{\infty}(\ph,\vvv),
\end{array}
\right.
\end{eqnarray}
where the velocity variables
\begin{eqnarray}
\vvv=(\ve,\vp),
\end{eqnarray}
the standard Maxwellian
\begin{eqnarray}
\m(\vvv)=\frac{1}{2\pi}\exp\bigg(-\frac{\abs{\vvv}^2}{2}\bigg),
\end{eqnarray}
$m_f[\g^{\e}]$ is the mass-flux given a priori, the limit function
\begin{eqnarray}
\g^{\e}_{\infty}(\ph,\vvv)=\sqrt{\m}\bigg(D_0^{\e}(\ph)+D_1^{\e}(\ph)\ve+D_2^{\e}(\ph)\vp+D_3^{\e}(\ph)\frac{\abs{\vvv}^2-2}{2}\bigg),
\end{eqnarray}
the forcing term
\begin{eqnarray}
G(\e;\et)&=&-\frac{\e\Upsilon(\e^{1/2}\et)}{1-\e\et},
\end{eqnarray}
and the cut-off function $\Upsilon(z)\in C^{\infty}$ is defined as
\begin{eqnarray}
\Upsilon(z)=\left\{
\begin{array}{ll}
1&0\leq z\leq1/2,\\
0&3/4\leq z\leq\infty.
\end{array}
\right.
\end{eqnarray}
We assume the boundary data and source term satisfy for $0\leq\zeta\leq 1/4$ and $\vartheta\geq0$,
\begin{eqnarray}\label{Milne bounded}
\abs{\bv h^{\e}(\ph,\vvv)}\leq M,
\end{eqnarray}
and
\begin{eqnarray}\label{Milne decay}
\abs{\bv S^{\e}(\eta,\ph,\vvv)}\leq M\ue^{-K\et},
\end{eqnarray}
for $M(\zeta,\vartheta)\geq0$ and $K(\zeta,\vartheta)>0$ uniform in $\e$ and $\ph$, where the Japanese bracket is defined as
\begin{eqnarray}
\br{\vvv}=(1+\abs{\vvv}^2)^{1/2}.
\end{eqnarray}
We define a potential function $W(\e;\et)$ as
$G(\e;\et)=-\p_{\et}W(\e;\et)$ with $W(\e;0)=0$.
Our main goal is to find
\begin{eqnarray}\label{Milne transform compatibility}
\tilde h^{\e}(\ph,\vvv)=\sqrt{\m}\bigg(\tilde D_0^{\e}(\ph)+\tilde D_1^{\e}(\ph)\ve+\tilde
D_2^{\e}(\ph)\vp+\tilde D_3^{\e}(\ph)\frac{\abs{\vvv}^2-2}{2}\bigg),
\end{eqnarray}
such that the $\e$-Milne problem for
$\gg^{\e}(\et,\ph,\vvv)$ in the domain
$(\et,\ph,\vvv)\in[0,\infty)\times[-\pi,\pi)\times\r^2$ as
\begin{eqnarray}\label{Milne transform}
\left\{
\begin{array}{rcl}\displaystyle
\ve\frac{\p \gg^{\e}}{\p\et}+G(\e;\et)\bigg(\vp^2\dfrac{\p
\gg^{\e}}{\p\ve}-\ve\vp\dfrac{\p
\gg^{\e}}{\p\vp}\bigg)+\ll[\gg^{\e}]
&=&S^{\e}(\et,\ph,\vvv),\\
\gg^{\e}(0,\ph,\vvv)&=&h^{\e}(\ph,\vvv)-\tilde h^{\e}(\ph,\vvv)\ \
\text{for}\ \ \ve>0,\\\rule{0ex}{1.0em}
\displaystyle\int_{\r^2}\ve\sqrt{\m}
\gg^{\e}(0,\ph,\vvv)\ud{\vvv}
&=&m_f[\gg^{\e}](\ph),\\\rule{0ex}{1.0em}
\displaystyle\lim_{\et\rt\infty}\gg^{\e}(\et,\ph,\vvv)&=&0,
\end{array}
\right.
\end{eqnarray}
is well-posed, where
\begin{eqnarray}
m_f[\gg^{\e}](\ph)=m_f[\g^{\e}](\ph)-\displaystyle\int_{\r^2}\ve\sqrt{\m}\tilde h^{\e}(\ph,\vvv)\ud{\vvv}.
\end{eqnarray}
For notational simplicity, we omit superscript $\e$ and $\ph$ dependence in
$\g^{\e}$ and $\gg^{\e}$ in this section. The same convention also applies to
$G(\e;\eta)$, $W(\e;\eta)$, $S^{\e}(\et,\ph,\vvv)$ and
$h^{\e}(\ph,\vvv)$. It is easy to see the estimates are uniform in
$\e$ and $\ph$. Our analysis is based on the ideas in \cite{Arkeryd.Esposito.Marra.Nouri2011, Cercignani.Marra.Esposito1998, Yang2012, AA003}.\\
\ \\
In this section, we introduce some special notations to describe the
norms in the space $(\et,\vvv)\in[0,\infty)\times\r^2$. Define
the $L^2$ norm as follows:
\begin{eqnarray}
\tnm{f(\et)}&=&\bigg(\int_{\r^2}\abs{f(\et,\vvv)}^2\ud{\vvv}\bigg)^{1/2},\\
\tnnm{f}&=&\bigg(\int_0^{\infty}\int_{\r^2}\abs{f(\et,\vvv)}^2\ud{\vvv}\ud{\et}\bigg)^{1/2}.
\end{eqnarray}
Define the inner product in $\vvv$ space
\begin{eqnarray}
\br{f,g}(\et)=\int_{\r^2}
f(\et,\vvv)g(\et,\vvv)\ud{\vvv}.
\end{eqnarray}
Define the weighted $L^{\infty}$ norm as follows:
\begin{eqnarray}
\lnm{f(\et)}{\vt,\ze}&=&\sup_{\vvv\in\r^2}\bigg(\bvv\abs{f(\et,\vvv)}\bigg),\\
\lnnm{f}{\vt,\ze}&=&\sup_{(\et,\vvv)\in[0,\infty)\times\r^2}\bigg(\bvv\abs{f(\et,\vvv)}\bigg),\\
\ltnm{f}{\ze}&=&\sup_{\et\in[0,\infty)}\bigg(\int_{\r^2}\abs{e^{2\ze\abs{\vvv}^2}f(\et,\vvv)}^2\ud{\vvv}\bigg)^{1/2}.
\end{eqnarray}
Since the boundary data $h(\vvv)$ is only defined on $\ve>0$, we
naturally extend above definitions on this half-domain as follows:
\begin{eqnarray}
\tnm{h}&=&\bigg(\int_{\ve>0}\abs{h(\vvv)}^2\ud{\vvv}\bigg)^{1/2},\\
\lnm{h}{\vt,\ze}&=&\sup_{\ve>0}\bigg(\bvv\abs{h(\vvv)}\bigg).
\end{eqnarray}
Since the null of operator $\ll$ is
$\nk=\sqrt{\m}\bigg\{1,\ve,\vp,\dfrac{\abs{\vvv}^2-2}{2}\bigg\}=\{\psi_0,\psi_1,\psi_2,\psi_3\}$,
we can decompose the solution as
\begin{eqnarray}
\g=w_g+q_g,
\end{eqnarray}
where
\begin{eqnarray}
q_g=\sqrt{\m}\bigg(q_{g0}+q_{g1}\ve+q_{g2}\vp+q_{g3}\dfrac{\abs{\vvv}^2-2}{2}\bigg)=q_{g0}\psi_0+q_{g1}\psi_1+q_{g2}\psi_2+q_{g3}\psi_3\in\nk,
\end{eqnarray}
and
\begin{eqnarray}
w_g\in\nk^{\bot}.
\end{eqnarray}
When there is no confusion, we will simply write $\g=w+q$.
\begin{lemma}\label{Milne force}
When $\e\leq 1/2$, the force $G(\et)$ and potential $W(\et)$ satisfy the following properties:\\
\ \\
1. $W(\et)$ is an increasing function satisfying for any $\et>0$,
\begin{eqnarray}\label{force 1}
0\leq W(\et)\leq 1,
\end{eqnarray}
2.
\begin{eqnarray}\label{force 2}
\lim_{\e\rt0}W(\infty)=0.
\end{eqnarray}
3.
\begin{eqnarray}\label{force 3}
\int_0^{\infty}\bigg(\ue^{-W(\et)}-\ue^{-W(\infty)}\bigg)^2\ud{\eta}&\leq&C.
\end{eqnarray}
4.
\begin{eqnarray}\label{force 4}
\int_0^{\infty}G^2(\et)\ud{\et}&\leq&C.
\end{eqnarray}
5.
\begin{eqnarray}\label{force 5}
\int_0^{\infty}\int_{\et}^{\infty}G^2(y)\ud{y}\ud{\et}\leq C.
\end{eqnarray}
\end{lemma}
\begin{proof}
Since $G(\et)$ is always zero or negative, we know $W(\et)$ is increasing. Then we directly compute
\begin{eqnarray}
\int_0^{\infty}\abs{G(\e;\eta)}\ud{\eta}&=&\int_0^{\infty}\frac{\e\Upsilon(\e^{1/2}\et)}{1-\e\et}\ud{\eta}\leq\int_0^{\frac{3}{4\sqrt{\e}}}\frac{\e}{1-\e\et}\ud{\eta}\\
&=&-\ln(1-\e\eta)\bigg|_{0}^{\frac{3}{4\sqrt{\e}}}=-\ln\bigg(1-\frac{3}{4}\e^{1/2}\bigg).\no
\end{eqnarray}
This naturally implies (\ref{force 1}) and (\ref{force 2}). Also,
we have
\begin{eqnarray}
\int_0^{\infty}\bigg(\ue^{-W(\et)}-\ue^{-W(\infty)}\bigg)^2\ud{\eta}&=&\ue^{-2W(\infty)}\int_0^{\frac{3}{4\sqrt{\e}}}\bigg(\ue^{W(\infty)-W(\et)}-1\bigg)^2\ud{\eta}.
\end{eqnarray}
Then we obtain
\begin{eqnarray}
\int_0^{\frac{3}{4\sqrt{\e}}}\ue^{W(\infty)-W(\et)}\ud{\et}&=&\int_0^{\frac{3}{4\sqrt{\e}}}\exp\bigg(-\int_{\et}^{\frac{3}{4\sqrt{\e}}}G(y)\ud{y}\bigg)\ud{\et}=
\int_0^{\frac{3}{4\sqrt{\e}}}\exp\bigg(\int_{\et}^{\frac{3}{4\sqrt{\e}}}\frac{\e\Upsilon(\e^{1/2}y)}{1-\e y}\ud{y}\bigg)\ud{\et}.\no
\end{eqnarray}
Note that
\begin{eqnarray}
1\leq \exp\bigg(\int_{\et}^{\frac{3}{4\sqrt{\e}}}\frac{\e\Upsilon(\e^{1/2}y)}{1-\e y}\ud{y}\bigg)\leq \exp\bigg(\int_{\et}^{\frac{3}{4\sqrt{\e}}}\frac{\e}{1-\e y}\ud{y}\bigg)
=\exp\bigg(-\ln(1-\e y)\bigg|_{\et}^{\frac{3}{4\sqrt{\e}}}\bigg)=\frac{1-\e\et}{1-\frac{3\sqrt{\e}}{4}}
\end{eqnarray}
Hence, we have
\begin{eqnarray}
\bigg(\ue^{W(\infty)-W(\et)}-1\bigg)^2&\leq&\max\bigg\{\bigg(\frac{1-\e\et}{1-\frac{3\sqrt{\e}}{4}}\bigg)^2-1, 2\bigg(\frac{1-\e\et}{1-\frac{3\sqrt{\e}}{4}}-1\bigg)\bigg\}\leq C\e\eta+C\sqrt{\e},
\end{eqnarray}
for some constant $C$ independent of $\e$. Then we have
\begin{eqnarray}
\int_0^{\frac{3}{4\sqrt{\e}}}\bigg(\ue^{W(\infty)-W(\et)}-1\bigg)^2\ud{\eta}\leq C\int_0^{\frac{3}{4\sqrt{\e}}}\bigg(\e\et+\sqrt{\e}\bigg)\ud{\eta}\leq C.
\end{eqnarray}
This proves (\ref{force 3}). Furthermore,
\begin{eqnarray}
\\
\int_0^{\infty}G^2(\et)\ud{\et}&=&\int_0^{\infty}\bigg(\frac{\e\Upsilon(\e^{1/2}\et)}{1-\e\et}\bigg)^2\ud{\et}\leq \int_0^{\frac{3}{4\sqrt{\e}}}\bigg(\frac{\e}{1-\e\et}\bigg)^2\ud{\et}=-\e\ln(1-\e\eta)\bigg|_0^{\frac{3}{4\sqrt{\e}}}=-\e\ln\bigg(1-\frac{3\sqrt{\e}}{4}\bigg).\no
\end{eqnarray}
Therefore, (\ref{force 4}) is obvious. Moreover, we compute
\begin{eqnarray}
\int_0^{\infty}\int_{\et}^{\infty}G^2(y)\ud{y}\ud{\et}&=&\int_0^{\infty}\int_{\et}^{\infty}\bigg(\frac{\e\Upsilon(\e^{1/2}y)}{1-\e y}\bigg)^2\ud{y}\ud{\et}= \int_0^{\frac{3}{4\sqrt{\e}}}\int_{\et}^{\frac{3}{4\sqrt{\e}}}\bigg(\frac{\e\Upsilon(\e^{1/2}y)}{1-\e y}\bigg)^2\ud{y}\ud{\et}\\
&\leq&\int_0^{\frac{3}{4\sqrt{\e}}}\int_{\et}^{\frac{3}{4\sqrt{\e}}}\bigg(\frac{\e}{1-\e y}\bigg)^2\ud{y}\ud{\et}
=-\int_0^{\frac{3}{4\sqrt{\e}}}\bigg(\frac{\e}{1-\e y}\bigg|_{\et}^{\frac{3}{4\sqrt{\e}}}\bigg)\ud{\et}\no\\
&=&\int_0^{\frac{3}{4\sqrt{\e}}}\bigg(\frac{\e}{1-\e \et}-\frac{\e}{1-\frac{3\sqrt{\e}}{4}}\bigg)\ud{\et}\no\\
&=&-\ln(1-\e\et)\bigg|_0^{\frac{3}{4\sqrt{\e}}}+\frac{\frac{3\sqrt{\e}}{4}}{1-\frac{3\sqrt{\e}}{4}}\leq C.\no
\end{eqnarray}
\end{proof}
\begin{lemma}\label{Milne property}
For the operator $\ll=\nu-K$, we have the estimates
\begin{eqnarray}
&&\nu_0(1+\abs{\vvv})\leq\nu(\vvv)\leq\nu_1(1+\abs{\vvv}),\\
&&\br{\g,\ll[\g]}(\et)=\br{w,\ll[w]}(\et)\geq C\tnm{\sqrt{\nu}w(\et)}^2,
\end{eqnarray}
for $\nu_0$, $\nu_1$ and $C$ positive constants.
\end{lemma}
\begin{proof}
See \cite[Chapter 3]{Glassey1996}.
\end{proof}

\subsection{$L^2$ Estimates}

\subsubsection{$L^2$ Estimates in a finite slab}

We first consider the case with zero source term for $\g^L(\et,\vvv)$
in a finite slab $[0,L]\times\r^2$ as
\begin{eqnarray}\label{finite slab LT}
\left\{
\begin{array}{rcl}\displaystyle
\ve\frac{\p \gl}{\p\et}+G(\et)\bigg(\vp^2\dfrac{\p
\gl}{\p\ve}-\ve\vp\dfrac{\p \gl}{\p\vp}\bigg)+\ll[\gl]
&=&0,\\\rule{0ex}{1.0em} \gl(0,\vvv)&=&h(\vvv)\ \ \text{for}\
\ \ve>0,\\\rule{0ex}{1.0em} \gl(L,R[\vvv])&=&\gl(L,\vvv),
\end{array}
\right.
\end{eqnarray}
where
\begin{eqnarray}
R[\ve,\vp]=(-\ve,\vp).
\end{eqnarray}
Similarly, we can decompose $\gl$ as
\begin{eqnarray}
\gl=\wl+\ql.
\end{eqnarray}
\begin{lemma}\label{Milne prelim lemma 1}
There exists a solution of the equation (\ref{finite slab LT})
satisfying the estimates
\begin{eqnarray}
\int_0^L\tnm{\sn\wl(\et)}^2\ud{\et}&\leq&C,\\
\tnm{\ql(\et)}^2&\leq&C\bigg(1+\et+\tnm{\sn\wl(\et)}^2\bigg),
\end{eqnarray}
where $C$ is a constant independent of $L$. Also, the solution satisfies the orthogonal relation
\begin{eqnarray}
\br{\ve\psi_i,\wl}(\et)=0,\ \ \text{for}\ \ i=0,2,3.
\end{eqnarray}
\end{lemma}
\begin{proof}
The existence follows from a standard argument by adding penalty term $\lambda\gl$ on the left-hand side of the equation for $0<\lambda<<1$ and estimate along the characteristics. Hence, we concentrate on the estimates. We divide the proof into several steps:\\
\ \\
Step 1: Estimate of $\wl$.\\
Multiplying $\gl$ on both sides of (\ref{finite slab LT}) and
integrating over $\vvv\in\r^2$, we have
\begin{eqnarray}
\half\frac{\ud{}}{\ud{\et}}\br{\ve\gl,\gl}+
G(\et)\br{\vp^2\dfrac{\p
\gl}{\p\ve}-\ve\vp\dfrac{\p \gl}{\p\vp},\gl}=-(\g^L,\ll[\g^L]).
\end{eqnarray}
An integration by parts implies
\begin{eqnarray}
\br{\vp^2\dfrac{\p
\gl}{\p\ve}-\ve\vp\dfrac{\p \gl}{\p\vp},\gl}=\half\br{\vp^2,\dfrac{\p
(\gl)^2}{\p\ve}}-\half\br{\ve\vp,\dfrac{\p (\gl)^2}{\p\vp}}=\half\br{\ve\gl,\gl}.
\end{eqnarray}
Therefore, using Lemma \ref{Milne property}, we obtain
\begin{eqnarray}
\half\frac{\ud{}}{\ud{\et}}\br{\ve\gl,\gl}+\half
G(\et)\br{\ve\gl,\gl}=-(\wl,\ll[\wl]).
\end{eqnarray}
Define
\begin{eqnarray}
\alpha(\et)=\half\br{\ve\gl,\gl}(\et),
\end{eqnarray}
which implies
\begin{eqnarray}
\frac{\ud{\alpha}}{\ud{\et}}+G(\et)\alpha=-(\wl,\ll[\wl]).
\end{eqnarray}
Then we have
\begin{eqnarray}
\alpha(\et)&=&\alpha(L)\exp\bigg(\int_{\et}^LG(y)\ud{y}\bigg)+
\int_{\et}^L\exp\bigg(-\int_{\et}^yG(z)\ud{z}\bigg)\bigg(\br{\wl,\ll[\wl]}(y)\bigg)\ud{y},\label{mt 01}\\
\alpha(\et)&=&\alpha(0)\exp\bigg(-\int_{0}^{\et}G(y)\ud{y}\bigg)+
\int_{0}^{\et}\exp\bigg(\int_{y}^{\et}G(z)\ud{z}\bigg)\bigg(-\br{\wl,\ll[\wl]}(y)\bigg)\ud{y}.\label{mt 02}
\end{eqnarray}
Since $\alpha(L)=0$ due to reflexive boundary, (\ref{mt 01}) implies
\begin{eqnarray}
\alpha(\et)\geq0.
\end{eqnarray}
By
\begin{eqnarray}
\alpha(0)=\half\int_{\ve>0}\ve(\gl)^2(0)\ud{\vvv}+\half\int_{\ve<0}\ve(\gl)^2(0)\ud{\vvv}\leq\half\int_{\ve>0}\ve h^2\ud{\vvv}\leq C,
\end{eqnarray}
and (\ref{mt 02}), we obtain
\begin{eqnarray}\label{mt 03}
\alpha(\et)\leq C.
\end{eqnarray}
Hence, (\ref{mt 01}) and (\ref{mt 03}) lead to
\begin{eqnarray}
\int_{0}^{L}\exp\bigg(-\int_{0}^{y}G(z)\ud{z}\bigg)\bigg(\br{\wl,\ll[\wl]}(y)\bigg)\ud{y}\leq
C,
\end{eqnarray}
which, by Lemma \ref{Milne force} and Lemma \ref{Milne property}, further yields
\begin{eqnarray}\label{mt 04}
\int_0^L\tnm{\sn\wl(\et)}^2\ud{\et}\leq C
\end{eqnarray}
\ \\
Step 2: Estimate of $\ql$.\\
Multiplying $\ve\psi_j$ with $j\neq1$ on both sides of (\ref{finite slab LT}) and integrating over
$\vvv\in\r^2$, we obtain
\begin{eqnarray}
\frac{\ud{}}{\ud{\et}}\br{\ve^2\psi_j,\gl}+G(\et)\br{\ve\psi_j,
\vp^2\dfrac{\p \gl}{\p\ve}-\ve\vp\dfrac{\p
\gl}{\p\vp}}=-\br{\ve\psi_j,\ll[\wl]}.
\end{eqnarray}
Define $\tilde q^L=\ql-q_1^L\psi_1$ and
\begin{eqnarray}
\beta_j(\et)&=&\br{\ve^2\psi_j,\tilde q^L}(\et),\\
\beta(\et)&=&\bigg(\beta_0(\et),\beta_1(\et),\beta_2(\et),\beta_3(\et)\bigg)^T\\
\tilde\beta(\et)&=&\bigg(\beta_0(\et),\beta_2(\et),\beta_3(\et)\bigg)^T.
\end{eqnarray}
By definition, $\beta_1=0$. For $j\neq1$, using integration by parts, we have
\begin{eqnarray}
\frac{\ud{}}{\ud{\et}}\br{\ve^2\psi_j,\gl}&=&G(\et)\br{\frac{\p}{\p\ve}(\ve\vp^2\psi_j)-\frac{\p}{\p\vp}(\ve^2\vp\psi_j),
\gl}-\br{\ve\psi_j,\ll[\wl]},
\end{eqnarray}
which further implies
\begin{eqnarray}\label{mt 05}
\\
\frac{\ud{\beta_j}}{\ud{\et}}&=&G(\et)\br{\frac{\p}{\p\ve}(\ve\vp^2\psi_j)-\frac{\p}{\p\vp}(\ve^2\vp\psi_j),
\tilde q^L+q_1^L\psi_1+\wl}-\br{\ve\psi_j,\ll[\wl]}-\frac{\ud{}}{\ud{\et}}\br{\ve^2\psi_j,\wl}.\no
\end{eqnarray}
Put $\tilde q_i^L=q_i^L-\d_{i1}q_1^L$. Then we can write
\begin{eqnarray}
\br{\frac{\p}{\p\ve}(\ve\vp^2\psi_j)-\frac{\p}{\p\vp}(\ve^2\vp\psi_j),
\tilde q^L}(\et)=\sum_{i}B_{ji}\tilde q_i^L(\et),
\end{eqnarray}
for $i,j=0,2,3$, where
\begin{eqnarray}
B_{ji}=\br{\frac{\p}{\p\ve}(\ve\vp^2\psi_j)-\frac{\p}{\p\vp}(\ve^2\vp\psi_j),
\psi_i}.
\end{eqnarray}
Moreover,
\begin{eqnarray}
\beta_j(\et)=\sum_{k}A_{jk}\tilde q_k^L(\et),
\end{eqnarray}
where
\begin{eqnarray}
A_{jk}=\br{\ve^2\psi_j,\psi_k},
\end{eqnarray}
is a non-singular matrix such that we can express back
\begin{eqnarray}
\tilde q_j^L(\et)=\sum_{k}A_{jk}^{-1}\beta_k(\et).
\end{eqnarray}
Hence, (\ref{mt 05}) can be rewritten as
\begin{eqnarray}
\frac{\ud{\tilde\beta}}{\ud{\et}}=G(BA^{-1})\tilde\beta+D,
\end{eqnarray}
where
\begin{eqnarray}
D_j&=&G(\et)\br{\frac{\p}{\p\ve}(\ve\vp^2\psi_j)-\frac{\p}{\p\vp}(\ve^2\vp\psi_j),
q_1^L\psi_1+\wl}-\br{\ve\psi_j,\ll[\wl]}-\frac{\ud{}}{\ud{\et}}\br{\ve^2\psi_j,\wl}.
\end{eqnarray}
We can solve for $\beta$ as
\begin{eqnarray}\label{mt 23}
\tilde\beta(\et)&=&\exp\bigg(-W(\et)BA^{-1}\bigg)\theta-\zeta(\et)+\int_0^{\et}\exp\bigg((W(\et)-W(y))BA^{-1}\bigg)Z(y)\ud{y},
\end{eqnarray}
where
\begin{eqnarray}
\theta_j=\br{\ve^2\psi_j, \gl}(0),\ \ j\neq1,
\end{eqnarray}
\begin{eqnarray}
\zeta_j(\et)=\br{\ve^2\psi_j, \wl}(\et),
\end{eqnarray}
and
\begin{eqnarray}
Z=D+\frac{\ud{\zeta}}{\ud{\et}}+G(BA^{-1})\zeta,
\end{eqnarray}
where we use the fact
\begin{eqnarray}
&&\int_0^{\et}\exp\bigg((W(\et)-W(y))BA^{-1}\bigg)\frac{\ud{\zeta}}{\ud{y}}\ud{y}
=\zeta(\et)-\exp\bigg(-W(\et)BA^{-1}\bigg)\zeta(0)-\int_0^{\et}G(BA^{-1})\zeta(y)\ud{y}.
\end{eqnarray}
Hence, using the boundedness of $W(\et)$ and $BA^{-1}$ due to Lemma \ref{Milne force}, we can directly estimate (\ref{mt 23}) to get
\begin{eqnarray}\label{mt 28}
\abs{\beta_j(\et)}\leq
C\abs{\th_j}+\abs{\zeta_j(\eta)}+C\int_0^{\et}\abs{Z_j(y)}\ud{y}\ \ \text{for}\ \ i=0,2,3.
\end{eqnarray}
By Cauchy's inequality, we obtain
\begin{eqnarray}
\abs{\zeta_j(\eta)}&\leq& \tnm{\sn\wl(\et)}\label{mt 24}\\
\abs{Z_j(\et)}&\leq&
C\bigg(\tnm{\sn\wl(\et)}+\tnm{q_1^L(\et)\psi_1}\bigg)\label{mt 25}.
\end{eqnarray}
Multiplying $\sqrt{\m}$ on both sides of (\ref{finite slab LT}) and
integrating over $\vvv\in\r^2$, we have
\begin{eqnarray}
\frac{\ud{}}{\ud{\et}}\br{\sqrt{\m}\ve,\gl}&=&G(\et)\br{\frac{\p}{\p\ve}(\vp^2)-\frac{\p}{\p\vp}(\ve\vp),
\sqrt{\m}\gl}=-G(\et)\br{\sqrt{\m}\ve,\gl},
\end{eqnarray}
which is actually
\begin{eqnarray}
\frac{\ud{q_1^L}}{\ud{\et}}=-G(\et)q_1^L.
\end{eqnarray}
Since $q_1^L(L)=0$, we have for any $\et\in[0,L]$,
\begin{eqnarray}\label{mt 26}
q_1^L(\et)=0.
\end{eqnarray}
Also,
\begin{eqnarray}
\br{\ve^2\psi_j, \gl}(0)\leq C\br{\abs{\ve}\gl(0),
\gl(0)}^{1/2}\br{\abs{\ve}^3,\psi_j^2}^{1/2}\leq C\br{\abs{\ve}\gl(0),
\gl(0)}^{1/2},
\end{eqnarray}
\begin{eqnarray}
\br{\abs{\ve}\gl(0), \gl(0)}=\int_{\ve>0}\m\ve
h^2-\int_{\ve<0}\m\ve(\gl(0))^2.
\end{eqnarray}
Since
\begin{eqnarray}
\int_{\ve>0}\m\ve h^2+\int_{\ve<0}\m\ve(\gl(0))^2=2\alpha(0)\geq0,
\end{eqnarray}
we have
\begin{eqnarray}\label{mt 27}
\theta_j=\br{\ve^2\psi_j, \gl}(0)\leq 2C\int_{\ve>0}\m\ve h^2\leq C.
\end{eqnarray}
In conclusion, collecting (\ref{mt 28}), (\ref{mt 24}), (\ref{mt 25}), (\ref{mt 26}), and (\ref{mt 27}), we have
\begin{eqnarray}
\abs{\beta_j(\et)}\leq C\bigg(
1+\tnm{\sn\wl(\et)}+\int_0^{\et}\tnm{\sn\wl(y)}\ud{y}\bigg)\ \ \text{for}\ \ j=0,2,3,
\end{eqnarray}
which further implies
\begin{eqnarray}
\abs{\ql_j(\et)}\leq C\bigg(
1+\tnm{\sn\wl(\et)}+\int_0^{\et}\tnm{\sn\wl(y)}\ud{y}\bigg)\ \ \text{for}\ \ j=0,2,3,
\end{eqnarray}
and $\ql_1(\et)=0$. An application of Cauchy's inequality leads to our desired result.\\
\ \\
Step 3: Orthogonal Properties.\\
In the equation (\ref{finite slab LT}), multiplying $\sqrt{\m}$ on both
sides and integrating over $\vvv\in\r^2$, we have
\begin{eqnarray}
\frac{\ud{}}{\ud{\et}}\br{\sqrt{\m}\ve,\gl}&=&G(\et)\br{\frac{\p}{\p\ve}(\vp^2)-\frac{\p}{\p\vp}(\ve\vp),
\sqrt{\m}\gl}=-G\br{\sqrt{\m}\ve,\gl}.
\end{eqnarray}
Since $\br{\sqrt{\m}\ve,\gl}(L)=0$, we have
\begin{eqnarray}
\br{\sqrt{\m}\ve,\gl}(\et)=0.
\end{eqnarray}
It is easy to check that
\begin{eqnarray}
\br{\ve\psi_i,\ql}&=&0\ \ i\neq1,\\
\br{\ve\ql,\ql}&=&0.
\end{eqnarray}
Multiplying $\psi_i$ for $i=2,3$ on both sides of (\ref{finite
slab LT}) and integrating over $\vvv\in\r^2$, we have
\begin{eqnarray}
\frac{\ud{}}{\ud{\et}}\br{\ve\psi_i,\wl}&=&=-CG\br{\ve\psi_i,\wl}.
\end{eqnarray}
Since $\br{\ve\psi_i,\wl(L)}=0$, then we have
\begin{eqnarray}
\br{\ve\psi_i,\wl(\et)}=0.
\end{eqnarray}
In summary, we have
\begin{eqnarray}
\br{\ve\psi_i,\wl(\et)}=0\ \ \text{for}\ \ i=0,2,3.
\end{eqnarray}
\end{proof}

\subsubsection{$L^2$ Estimates in an infinite slab}

We consider the case with zero source term and zero mass flux in an
infinite slab
\begin{eqnarray}\label{infinite slab LT}
\left\{
\begin{array}{rcl}\displaystyle
\ve\frac{\p \g}{\p\et}+G(\et)\bigg(\vp^2\dfrac{\p
\g}{\p\ve}-\ve\vp\dfrac{\p \g}{\p\vp}\bigg)+\ll[\g]
&=&0,\\\rule{0ex}{1.0em} \g(0,\vvv)&=&h(\vvv)\ \ \text{for}\ \
\ve>0,\\\rule{0ex}{1.0em}\displaystyle\int_{\r^2}\ve\sqrt{\m}
\g(0,\vvv)\ud{\vvv}
&=&0\\
\displaystyle\lim_{\et\rt\infty}\g(\et,\vvv)&=&\g_{\infty}(\vvv).
\end{array}
\right.
\end{eqnarray}
\begin{lemma}\label{Milne prelim lemma 2}
There exists a unique solution of the equation (\ref{infinite slab
LT}) satisfying the estimate
\begin{eqnarray}
\tnnm{\sn\wi}&\leq&C,\label{mt 17}\\
\abs{\qi_{i,\infty}}&\leq&C,\label{mt 18}\\
\tnnm{\qi-\qi_{\infty}}&\leq& C\label{mt 19},
\end{eqnarray}
where $q_{\infty}=\sum_{i=0}^3q_{i,\infty}\psi_i$ and the orthogonal properties:
\begin{eqnarray}\label{mt 29}
\br{\ve\psi_i,\wi}=0\ \ \text{for}\ \ i=0,2,3.
\end{eqnarray}
\end{lemma}
\begin{proof}
We divide the proof into several steps:\\
\ \\
Step 1: Weak convergence.\\
We can extend the solution $\gl$ by passing $L\rt\infty$. Hence, we
can always take weakly convergent subsequence
\begin{eqnarray}
\ql_i(\et)&\rt&\qi_i(\et)\ \ \text{in}\ \ L^2_{\text{loc}}([0,\infty)), \\
\wl&\rt&\wi\ \ \text{in}\ \ L^2_{\text{loc}}([0,\infty),L^2(\r^2)).
\end{eqnarray}
Therefore,
\begin{eqnarray}
g=\sum_{i=0}^3\qi_i\psi_i+\wi,
\end{eqnarray}
is a weak solution of the equation (\ref{infinite slab LT}). Also,
by the weak lower semi-continuity, the estimate (\ref{mt 17}) of $\wi$ is
obvious. Also, we can show the orthogonal properties (\ref{mt 29}) when $L\rt\infty$.\\
\ \\
Step 2: Estimate of $\qi_{\infty}$.\\
It is easy to see
\begin{eqnarray}
\qi_1(\et)=m_f[g]=0,
\end{eqnarray}
so we do not need to bother with it. Since $\ll: L^2(\r^2)\rt\nk^{\bot}$ with null space $\nk$ and image $\nk^{\bot}$, we have $\tilde\ll: L^2/\nk\rt\nk^{\bot}$ is bijective, where $L^2/\nk=\nk^{\bot}$ is the quotient space. Then we can define its inverse, i.e. the pseudo-inverse of $\ll$ as
$\ll^{-1}: \nk^{\bot}\rt\nk^{\bot}$ satisfying $\ll\ll^{-1}[f]=f$ for any $f\in\nk^{\bot}$.

We intend to multiply
$\ll^{-1}[\ve\psi_{i}]$ for $i=2,3$ on both sides of
(\ref{infinite slab LT}) and integrating over $\vvv$. Notice that $\ve\psi_{2}\in \nk^{\bot}$, but $\ve\psi_{3}\notin \nk^{\bot}$. Actually, it is easy to verify $\ve(\psi_{3}-\psi_0)\in \nk^{\bot}$. To avoid introducing new notation, we still use $\psi_3$ to denote $\psi_{3}-\psi_0$ in the following proof and it is easy to see there is no confusion. Then we get
\begin{eqnarray}
&&\frac{\ud{}}{\ud{\et}}\br{\ll^{-1}[\psi_i\ve],\ve\g}+G(\et)\br{\ll^{-1}[\psi_i\ve],\bigg(\vp^2\dfrac{\p
\g}{\p\ve}-\ve\vp\dfrac{\p
\g}{\p\vp}\bigg)}=-\br{\ll^{-1}[\psi_i\ve],\ll[\wi]}.
\end{eqnarray}
Since $\ll$ is self-adjoint, combining with the orthogonal properties, we have
\begin{eqnarray}
\br{\ll^{-1}[\psi_i\ve],\ll[\wi]}=\br{\ll\bigg[\ll^{-1}[\psi_i\ve]\bigg],\wi}=\br{\psi_i\ve,\wi}=0.
\end{eqnarray}
Therefore, we have
\begin{eqnarray}\label{mt 12}
&&\frac{\ud{}}{\ud{\et}}\br{\ve\ll^{-1}[\psi_i\ve],\g}+G(\et)\br{\ll^{-1}[\psi_i\ve],\bigg(\vp^2\dfrac{\p
\g}{\p\ve}-\ve\vp\dfrac{\p \g}{\p\vp}\bigg)}=0.
\end{eqnarray}
Since $\psi_1\in\nk$ and $\ll^{-1}[\psi_i\ve]\in\nk^{\bot}$, we
have
\begin{eqnarray}
\br{\psi_1,\ll^{-1}[\psi_i\ve]}=0.
\end{eqnarray}
For $i,k=2,3$, put
\begin{eqnarray}
N_{i,k}&=&\br{\ve\ll^{-1}[\psi_i\ve],\psi_k},\\
P_{i,k}&=&\br{\bigg(\vp^2\dfrac{\p }{\p\ve}-\ve\vp\dfrac{\p
}{\p\vp}+\ve\bigg)\ll^{-1}[\psi_i\ve],\psi_k}.
\end{eqnarray}
Thus,
\begin{eqnarray}
\Omega_i&=&\br{\ve\ll^{-1}[\psi_i\ve],\qi}=\sum_{k=2}^3N_{i,k}\qi_k(\et),
\end{eqnarray}
and
\begin{eqnarray}
\br{\bigg(\vp^2\dfrac{\p }{\p\ve}-\ve\vp\dfrac{\p
}{\p\vp}+\ve\bigg)\ll^{-1}[\psi_i\ve],\qi}&=&\sum_{k=2}^3P_{i,k}\qi_k(\et).
\end{eqnarray}
Since matrix $N$ is invertible (see \cite{Golse.Poupaud1989}), from (\ref{mt 12}) and integration
by parts, we have for $i=2,3$,
\begin{eqnarray}
\frac{\ud{\Omega_i}}{\ud{\et}}&=&-\frac{\ud{}}{\ud{\et}}\br{\ve\ll^{-1}[\psi_i\ve],\wi}\\
&&+\sum_{k=2}^3G(\et)(PN^{-1})_{ik}\Omega_k+G(\et)\br{\bigg(\vp^2\dfrac{\p
}{\p\ve}-\ve\vp\dfrac{\p
}{\p\vp}+\ve\bigg)\ll^{-1}[\hat\psi_i\ve],\wi}.\no
\end{eqnarray}
Denote
\begin{eqnarray}
\Omega'=\exp\bigg(W(\et)PN^{-1}\bigg)\Omega.
\end{eqnarray}
Let $\hat\psi=(\psi_2,\psi_3)^T$, we can solve
\begin{eqnarray}
\Omega'(\et)
&=&\br{\ve\ll^{-1}[\hat\psi\ve],\g}(0)-\exp\bigg(W(\et)PN^{-1}\bigg)\br{\ve\ll^{-1}[\hat\psi\ve],\wi}(\et)\\
&&+\int_0^{\et}\exp\bigg(W(y)PN^{-1}\bigg)G(y)\Bigg(\br{\bigg(\vp^2\dfrac{\p
}{\p\ve}-\ve\vp\dfrac{\p
}{\p\vp}+\ve\bigg)\ll^{-1}[\hat\psi\ve],\wi}(y)\no\\
&&+\sum_{k=2}^3PN^{-1}\br{\ve\ll^{-1}[\hat\psi\ve],\wi}(y)\Bigg)\ud{y}.\no
\end{eqnarray}
By a similar method as in the proof of Lemma \ref{Milne prelim lemma 1} to bound $\theta_i(0)$, we can show
\begin{eqnarray}
\br{\ve\ll^{-1}[\hat\psi\ve],\g}(0)<\infty.
\end{eqnarray}
Since $\wi\in L^2([0,\infty)\times\r^2)$,
considering $W(\et)$ and $PN^{-1}$ are bounded, and $G(\et)\in L^{\infty}$,
we define
\begin{eqnarray}
\\
\Omega'_{\infty}&=&\br{\ve\ll^{-1}[\hat\psi\ve],\g}(0)
+\int_0^{\infty}\exp\bigg(W(y)PN^{-1}\bigg)G(y)\Bigg(\br{\bigg(\vp^2\dfrac{\p
}{\p\ve}-\ve\vp\dfrac{\p
}{\p\vp}+\ve\bigg)\ll^{-1}[\hat\psi\ve],\wi}(y)\no\\
&&+\sum_{k=2}^3PN^{-1}\br{\ve\ll^{-1}[\hat\psi\ve],\wi}(y)\Bigg)\ud{y}.\no
\end{eqnarray}
Let $\hat q=(q_2,q_3)^T$. Then we can define
\begin{eqnarray}
\hat q_{\infty}=N^{-1}\exp\bigg(-W(\infty)PN^{-1}\bigg)\Omega'_{\infty}.
\end{eqnarray}
Finally, we consider $q_0$. Multiplying $\psi_1$ on both sides of
(\ref{infinite slab LT}) and integrating over $\vvv$, we obtain
\begin{eqnarray}
\frac{\ud{}}{\ud{\et}}\br{\psi_1\g,\ve}=-G(\et)\br{\psi_1,\bigg(\vp^2\dfrac{\p
\g}{\p\ve}-\ve\vp\dfrac{\p
\g}{\p\vp}\bigg)}=G(\et)\br{g,(\vp^2-\ve^2)\sqrt{\m}}.
\end{eqnarray}
Then integrating over $[0,\et]$, we obtain
\begin{eqnarray}\label{mt 13}
\br{\psi_1\g,\ve}(\et)&=&\br{\psi_1\g,\ve}(0)+\int_0^{\et}G(y)\br{g,(\vp^2-\ve^2)\sqrt{\m}}\ud{y}\\
&=&\br{\psi_1\g,\ve}(0)+\int_0^{\et}G(y)\br{w,(\vp^2-\ve^2)\sqrt{\m}}(y)\ud{y}.\no
\end{eqnarray}
Since $\wi\in L^2([0,\infty)\times\r^2)$ and we can also bound $\br{\ve\g,\ve}(0)$, we have
\begin{eqnarray}
\lim_{\eta\rt\infty}\br{\psi_1\g,\ve}(\et)=\br{\psi_1\g,\ve}(0)+\int_0^{\infty}G(y)\br{w,(\vp^2-\ve^2)\sqrt{\m}}(y)\ud{y}
\end{eqnarray}
exists. Note that
\begin{eqnarray}
\lim_{\eta\rt\infty}\br{\psi_1\qi_{1,\infty},\ve}(\et)=\lim_{\eta\rt\infty}\br{\psi_1\qi_{2,\infty},\ve}(\et)=0.
\end{eqnarray}
Then we define
\begin{eqnarray}
\qi_{0,\infty}=\frac{\lim_{\eta\rt\infty}\br{\psi_1\g,\ve}(\et)-\qi_{3,\infty}\br{\psi_1\psi_3,\ve}}{\br{\ve\psi_0,\ve}}.
\end{eqnarray}
Then to summarize all above, we have defined $\qi_{\infty}$ which satisfies $\abs{\qi_{i,\infty}}\leq C$ for $i=0,1,2,3$.\\
\ \\
Step 3: $L^2$ Decay of $\wi$.\\
The orthogonal property and zero mass-flux imply
\begin{eqnarray}
\br{\ve\qi,\wi}(\et)=\sum_{k=0}^3\br{\ve\psi_k,\wi}(\et)=0.
\end{eqnarray}
The oddness and zero mass-flux imply
\begin{eqnarray}
\br{\ve\qi,\qi}(\et)=0.
\end{eqnarray}
Therefore, we deduce that
\begin{eqnarray}
\br{\ve\g,\g}(\et)=\br{\ve\wi,\wi}(\et).
\end{eqnarray}
Multiplying $\ue^{2K_0\et}\g$ on both sides of (\ref{infinite slab LT}) and integrating over $\vvv$, we obtain
\begin{eqnarray}
\half\frac{\ud{}}{\ud{\et}}\bigg(\ue^{2K_0\et+W(\et)}\br{\ve\wi,\wi}\bigg)-\ue^{2K_0\et+W(\et)}\bigg(K_0\br{\ve\wi,\wi}
-\br{\wi,\ll\wi}\bigg)=0.
\end{eqnarray}
Since
\begin{eqnarray}
\br{\ll[\wi],\wi}\geq \nu_0\br{(1+\abs{\vvv})\wi,\wi},
\end{eqnarray}
and $W(\et)$ is bounded,
for $K_0$ sufficiently small, we have
\begin{eqnarray}
\br{\ll[\wi],\wi}-\br{K_0\ve\wi,\wi}\geq C\br{\wi,\wi}.
\end{eqnarray}
Then by a similar argument as in Lemma \ref{Milne prelim lemma 1},
we can show
\begin{eqnarray}
\int_0^{\infty}\ue^{2K_0\et}\br{\nu\wi,\wi}(\et)\ud{\et}\leq C.
\end{eqnarray}
\ \\
Step 4: Estimate of $\qi-\qi_{\infty}$.\\
We first consider $\hat\qi=(\qi_2,\qi_3)^T$, which satisfies
\begin{eqnarray}
\hat\qi(\et)=N^{-1}\exp\bigg(-W(\et)PN^{-1}\bigg)\Omega'(\et),
\end{eqnarray}
Let
\begin{eqnarray}
\delta&=&\br{\ve\ll^{-1}[\hat\psi\ve],\g}(0)\\
\Delta&=&G\Bigg(\br{\bigg(\vp^2\dfrac{\p
}{\p\ve}-\ve\vp\dfrac{\p
}{\p\vp}+\ve\bigg)\ll^{-1}[\hat\psi\ve],\wi}
+\sum_{k=2}^3PN^{-1}\br{\ve\ll^{-1}[\hat\psi\ve],\wi}\Bigg)
\end{eqnarray}
Then we have
\begin{eqnarray}
\hat\qi(\et)
&=&N^{-1}\exp\bigg(-W(\et)PN^{-1}\bigg)\delta-N^{-1}\br{\ve\ll^{-1}[\hat\psi\ve],\wi}(\et)\\
&&+\int_0^{\et}\exp\bigg((W(y)-W(\et))PN^{-1}\bigg)\Delta(y)\ud{y}.\no
\end{eqnarray}
Also, we know
\begin{eqnarray}
\hat\qi_{\infty}&=&N^{-1}\exp\bigg(-W(\infty)PN^{-1}\bigg)\delta+\int_0^{\infty}\exp\bigg((W(y)-W(\infty))PN^{-1}\bigg)\Delta(y)\ud{y}.
\end{eqnarray}
Then we have
\begin{eqnarray}
\hat\qi(\et)-\hat\qi_{\infty}&=&N^{-1}\Bigg(\exp\bigg(-W(\et)PN^{-1}\bigg)-\exp\bigg(-W(\infty)PN^{-1}\bigg)\Bigg)\delta-N^{-1}\br{\ve\ll^{-1}[\hat\psi\ve],\wi}(\et)\\
&&+N^{-1}\Bigg(\exp\bigg(-W(\et)PN^{-1}\bigg)-\exp\bigg(-W(\infty)PN^{-1}\bigg)\Bigg)\int_0^{\infty}\exp\bigg(W(y)PN^{-1}\bigg)\Delta(y)\ud{y}\no\\
&&+N^{-1}\int_{\et}^{\infty}\exp\bigg((W(y)-W(\et))PN^{-1}\bigg)\Delta(y)\ud{y}.\no
\end{eqnarray}
Then we have
\begin{eqnarray}\label{mt 31}
&&\tnnm{\hat\qi-\hat\qi_{\infty}}\\
&\leq&\tnnm{N^{-1}\Bigg(\exp\bigg(-W(\et)PN^{-1}\bigg)-\exp\bigg(-W(\infty)PN^{-1}\bigg)\Bigg)\delta}
+\tnnm{N^{-1}\br{\ve\ll^{-1}[\hat\psi\ve],\wi}}\no\\
&&+\tnnm{N^{-1}\Bigg(\exp\bigg(-W(\et)PN^{-1}\bigg)-\exp\bigg(-W(\infty)PN^{-1}\bigg)\Bigg)\int_0^{\infty}\exp\bigg(W(y)PN^{-1}\bigg)\Delta(y)\ud{y}}\no\\
&&+\tnnm{N^{-1}\int_{\et}^{\infty}\exp\bigg((W(y)-W(\et))PN^{-1}\bigg)\Delta(y)\ud{y}}.\no
\end{eqnarray}
We need to estimate each term on the right-hand side of (\ref{mt 31}). By Lemma \ref{Milne force}, we have
\begin{eqnarray}
&&\tnnm{N^{-1}\Bigg(\exp\bigg(-W(\et)PN^{-1}\bigg)-\exp\bigg(-W(\infty)PN^{-1}\bigg)\Bigg)\delta}^2\\
&\leq&C\delta\tnnm{\ue^{-W(\et)}-\ue^{-W(\infty)}}\leq C.
\end{eqnarray}
Since $\wi\in L^2([0,\infty)\times\r^2)$, we have
\begin{eqnarray}
\tnnm{N^{-1}\br{\ve\ll^{-1}[\hat\psi\ve],\wi}}\leq C\tnnm{w}\leq C.
\end{eqnarray}
Similarly, we can show
\begin{eqnarray}
&&\tnnm{N^{-1}\Bigg(\exp\bigg(-W(\et)PN^{-1}\bigg)-\exp\bigg(-W(\infty)PN^{-1}\bigg)\Bigg)\int_0^{\infty}\exp\bigg(W(y)PN^{-1}\bigg)\Delta(y)\ud{y}}\\
&\leq&C\tnnm{\ue^{-W(\et)}-\ue^{-W(\infty)}}\tnnm{r}\leq C.\no
\end{eqnarray}
For the last term, we have to resort to exponential decay of $\wi$. We estimate
\begin{eqnarray}
&&\tnnm{N^{-1}\int_{\et}^{\infty}\exp\bigg((W(y)-W(\et))PN^{-1}\bigg)\Delta(y)\ud{y}}\\
&\leq&C\int_0^{\infty}\bigg(\int_{\et}^{\infty}\Delta(y)\ud{y}\bigg)^2\ud{\et}\leq \int_0^{\infty}\bigg(\int_{\et}^{\infty}\ue^{-2K_0y}\ud{y}\bigg)\bigg(\int_{\et}^{\infty}\wi^2(y)\ue^{2K_0y}\ud{y}\bigg)\ud{\et}\no\\
&\leq&\int_0^{\infty}C\ue^{-2K_0\et}\ud{\et}\leq C.\no
\end{eqnarray}
Collecting all above, we have
\begin{eqnarray}
\tnnm{\hat\qi-\hat\qi_{\infty}}\leq C.
\end{eqnarray}
Then we turn to $\qi_0$. We have
\begin{eqnarray}
\qi_{0}(\et)=\frac{\br{\psi_1\g,\ve}(\et)-\qi_{3}(\et)\br{\psi_1\psi_3,\ve}}{\br{\ve\psi_0,\ve}},
\end{eqnarray}
where
\begin{eqnarray}
\br{\psi_1\g,\ve}(\et)&=&\br{\psi_1\g,\ve}(0)+\int_0^{\et}G(y)\br{w,(\vp^2-\ve^2)\sqrt{\m}}(y)\ud{y}.
\end{eqnarray}
Also, we have
\begin{eqnarray}
\qi_{0,\infty}(\et)=\frac{\lim_{\et\rt\infty}\br{\psi_1\g,\ve}(\et)-\qi_{3,\infty}\br{\psi_1\psi_3,\ve}}{\br{\ve\psi_0,\ve}},
\end{eqnarray}
where
\begin{eqnarray}
\lim_{\et\rt\infty}\br{\psi_1\g,\ve}(\et)&=&\br{\psi_1\g,\ve}(0)+\int_0^{\infty}G(y)\br{w,(\vp^2-\ve^2)\sqrt{\m}}(y)\ud{y}.
\end{eqnarray}
Therefore, we have
\begin{eqnarray}
\qi_0(\et)-\qi_{0,\infty}=\frac{\displaystyle\int_{\et}^{\infty}G(y)\br{w,(\vp^2-\ve^2)\sqrt{\m}}(y)\ud{y}-(\qi_3(\et)-\qi_{3,\infty})\br{\psi_1\psi_3,\ve}}{\br{\ve\psi_0,\ve}}
\end{eqnarray}
Then we can naturally estimate
\begin{eqnarray}
\tnnm{\qi_0-\qi_{0,\infty}}\leq C\tnnm{\int_{\et}^{\infty}G(y)\br{w,(\vp^2-\ve^2)\sqrt{\m}}(y)\ud{y}}+C\tnnm{\qi_3(\et)-\qi_{3,\infty}}
\end{eqnarray}
$\tnnm{\qi_3(\et)-\qi_{3,\infty}}$ is bounded due to the estimate of $\tnnm{\hat\qi(\et)-\hat\qi_{\infty}}$. Then by Cauchy's inequality and Lemma \ref{Milne force}, we obtain
\begin{eqnarray}
\tnnm{\int_{\et}^{\infty}G(y)\br{w,(\vp^2-\ve^2)\sqrt{\m}}(y)\ud{y}}&\leq& \int_0^{\infty}\bigg(\int_{\et}^{\infty}G(y)\tnm{r(y)}\ud{y}\bigg)^2\ud{\et}\\
&\leq&\tnnm{r}\int_0^{\infty}\int_{\et}^{\infty}G^2(y)\ud{y}\ud{\et}\leq C.
\end{eqnarray}
Therefore, we have shown
\begin{eqnarray}
\tnnm{\qi_0(\et)-\qi_{0,\infty}}\leq C.
\end{eqnarray}
In summary, we prove that
\begin{eqnarray}
\tnnm{\qi-\qi_{\infty}}\leq C.
\end{eqnarray}
\ \\
Step 5: Uniqueness.\\
If $\g_1$ and $\g_2$ are two solutions of (\ref{infinite slab LT}), define $\g'=\g_1-\g_2$. Then $\g'$ satisfies the equation
\begin{eqnarray}\label{infinite slab LT difference'}
\left\{
\begin{array}{rcl}\displaystyle
\ve\frac{\p\g'}{\p\et}+G(\et)\bigg(\vp^2\dfrac{\p
\g'}{\p\ve}-\ve\vp\dfrac{\p\g'}{\p\vp}\bigg)+\ll[\g']
&=&0,\\\rule{0ex}{1.0em} \g'(0,\vvv)&=&0\ \
\text{for}\ \ \ve>0,\\\rule{0ex}{1.0em}
\displaystyle\int_{\r^2}\ve\sqrt{\m} \g'(0,\vvv)\ud{\vvv}
&=&0,\\\rule{0ex}{1.0em}
\lim_{\et\rt\infty}\g'(\et,\vvv)&=&\g'_{\infty}(\vvv),
\end{array}
\right.
\end{eqnarray}
Similarly, we can define $\g'=\wi'+\qi'$.
Define the linearized entropy as
\begin{eqnarray}
H[g'](\et)=\br{\ve\g',\g'}(\et).
\end{eqnarray}
Multiplying
$\g'$ on both sides of
(\ref{infinite slab LT difference'}) and integrating over $\vvv$, we get
\begin{eqnarray}
\half\frac{\ud{}}{\ud{\et}}\br{\ve\g',\g'}=\br{\wi',\ll[\wi']}-\half G(\et)\br{\ve\g',\g'}.
\end{eqnarray}
Hence, we have
\begin{eqnarray}\label{mt 81}
\half\frac{\ud{}}{\ud{\et}}\bigg(\ue^{W}\br{\ve\g',\g'}\bigg)=-\ue^{W}\br{\wi',\ll[\wi']},
\end{eqnarray}
which implies $\ue^{W}H[g']$ is decreasing. Furthermore, we have
\begin{eqnarray}\label{mt 20}
\ue^{W(\et)}H[g'](\et)=H[g'](0)-\int_0^{\et}\ue^{W(y)}\br{\wi',\ll[\wi']}(y)\ud{y}<\infty.
\end{eqnarray}
Hence, we can take a subsequence such that $\tnm{\sn\wi'(\et_n)}$ goes to zero. Then we can always assume $\qi'(\et_n)$ goes to $\qi'_{\infty}$.
Therefore, we have
\begin{eqnarray}
\ue^{W(\et_n)}H[g'](\et_n)\rt \br{\ve\qi'_{\infty}, \qi'_{\infty}}.
\end{eqnarray}
Since $m_f[\g']=0$, we naturally obtain
\begin{eqnarray}
\ue^{W(\et_n)}H[g'](\et_n)\rt 0\ \ \text{as}\ \ \et_n\rt\infty.
\end{eqnarray}
Hence, we have
\begin{eqnarray}
\ue^{W(\et)}H[g'](\et)\geq0,
\end{eqnarray}
and
\begin{eqnarray}
\ue^{W(\et)}H[g'](\et)\rt 0\ \ \text{as}\ \ \et\rt\infty.
\end{eqnarray}
In (\ref{mt 81}), integrating over $[0,\infty)$, we achieve
\begin{eqnarray}
-\int_{\ve<0}\ve(\g')^2(0)\ud{\vvv}+\int_0^{\infty}\ue^{W(\et)}\br{\wi',\ll[\wi']}(\et)\ud{\et}=\int_{\ve>0}\ve(\g')^2(0)\ud{\vvv}=0.
\end{eqnarray}
Hence, we have
\begin{eqnarray}
\int_{\ve<0}\ve(\g')^2(0)\ud{\vvv}=\int_0^{\infty}\br{\wi',\ll[\wi']}(\et)\ud{\et}=0,
\end{eqnarray}
which implies $\g'(0)=0$ and $\wi'=0$. Hence, $\g'=\qi'$ and satisfies
\begin{eqnarray}\label{mt 21}
\ve\frac{\p\g'}{\p\et}+G(\et)\bigg(\vp^2\dfrac{\p
\g'}{\p\ve}-\ve\vp\dfrac{\p\g'}{\p\vp}\bigg)=0.
\end{eqnarray}
$m_f[g']=0$ implies $\qi'_{1}=0$. Therefore, multiplying $\ve\psi_i$ for $i\neq1$ on both sides of (\ref{mt 21}) and integrating over $\vvv$,
we obtain a linear system on $\qi'_{k}$ for $k=1,2,3$ with initial data zero, which possesses a unique solution zero. This means $\g'=0$. Hence, the solution is unique.
\end{proof}

\subsubsection{$L^2$ Estimates with general source term and non-vanishing mass-flux}

We consider the Milne problem with general source term and non-vanishing mass-flux.
\begin{eqnarray}\label{infinite slab LT general}
\left\{
\begin{array}{rcl}\displaystyle
\ve\frac{\p \g}{\p\et}+G(\et)\bigg(\vp^2\dfrac{\p
\g}{\p\ve}-\ve\vp\dfrac{\p \g}{\p\vp}\bigg)+\ll[\g]
&=&S,\\\rule{0ex}{1.0em} \g(0,\vvv)&=&h(\vvv)\ \ \text{for}\ \
\ve>0,\\\rule{0ex}{1.0em}\displaystyle\int_{\r^2}\ve
\sqrt{\m}\g(0,\vvv)\ud{\vvv}
&=&m_f[g]\\
\lim_{\et\rt\infty}\g(\et,\vvv)&=&\g_{\infty}(\vvv).
\end{array}
\right.
\end{eqnarray}
\begin{lemma}\label{Milne prelim lemma 3}
There exists a unique solution of the equation (\ref{infinite slab
LT general}) satisfying the estimate
\begin{eqnarray}
\tnnm{\sn\wi}&\leq&C,\\
\abs{\qi_{i,\infty}}{}&\leq&C,\\
\tnnm{\qi-\qi_{\infty}}&\leq& C,
\end{eqnarray}
where $q_{\infty}=\sum_{i=0}^3q_{i,\infty}\psi_i$.
\end{lemma}
\begin{proof}
For the non-vanishing mass flux problem, we can see $\hat\g=\g-m_f[g]\sqrt{\m}\ue^{-\et}\ve$ satisfies the $\e$-Milne problem with zero mass flux
with the source term
\begin{eqnarray}
\hat S=S+m_f[g]\sqrt{\m}\bigg(\ve^2-G(\et)\vp^2\bigg)\ue^{-\et},
\end{eqnarray}
and the boundary data
\begin{eqnarray}
\hat h=h-m_f[g]\sqrt{\m}\ve.
\end{eqnarray}
Therefore, we only need to consider the case with general source term and zero mass-flux. However, if $\displaystyle\int_{\r^2}\sqrt{\m}S(\et,\vvv)\ud{\vvv}\neq0$, the mass-flux is not conserved when $\et$ changes.
The construction of solutions can be divided into several steps:\\
\ \\
Step 1: Decomposition of the source term.\\
We decompose the source term as
\begin{eqnarray}
S=S_Q+S_W,
\end{eqnarray}
where $S_Q\in\nk$ is the kernel part and $S_W=S-S_Q\in\nk^{\bot}$.\\
\ \\
Step 2: Construction of $\g_1$.\\
We first solve the problem with source term $S_W$ as
\begin{eqnarray}
\left\{
\begin{array}{rcl}\displaystyle
\ve\frac{\p \g_1}{\p\et}+G(\et)\bigg(\vp^2\dfrac{\p
\g_1}{\p\ve}-\ve\vp\dfrac{\p \g_1}{\p\vp}\bigg)+\ll[\g_1]
&=&S_W,\\\rule{0ex}{1.0em} \g_1(0,\vvv)&=&h(\vvv)\ \ \text{for}\ \
\ve>0,\\\rule{0ex}{1.0em}\displaystyle\int_{\r^2}\ve
\sqrt{\m}\g_1(0,\vvv)\ud{\vvv}
&=&0\\
\lim_{\et\rt\infty}\g_1(\et,\vvv)&=&\g_{1,\infty}(\vvv).
\end{array}
\right.
\end{eqnarray}
In this case, we apply similar techniques as in the analysis of $S=0$ case. All the results can be generalized in a natural way.
Hence, we know $\g_1$ is well-posed.\\
\ \\
Step 3: Construction of $\g_2$.\\
We try to find a function $\g_2$ such that
\begin{eqnarray}
\ll\left[\ve\frac{\p \g_2}{\p\et}+G(\et)\bigg(\vp^2\dfrac{\p
\g_2}{\p\ve}-\ve\vp\dfrac{\p \g_2}{\p\vp}\bigg)+S_Q\right]=0.
\end{eqnarray}
which further means
\begin{eqnarray}\label{mt 22}
\int_{\r^2}\psi_i\left(\ve\frac{\p \g_2}{\p\et}+G(\et)\bigg(\vp^2\dfrac{\p
\g_2}{\p\ve}-\ve\vp\dfrac{\p \g_2}{\p\vp}\bigg)+S_Q\right)\ud{\vvv}=0.
\end{eqnarray}
for $i=0,1,2,3$. Consider
\begin{eqnarray}
S_Q=\sqrt{\m}\bigg(a(\et)+\vb(\et)\cdot\vvv+c(\et)\abs{\vvv}^2\bigg).
\end{eqnarray}
We make an ansatz that
\begin{eqnarray}
\g_2=\sqrt{\m}\bigg(A(\et)\ve+B_1(\et)+B_2(\et)\ve\vp+C(\et)\ve\abs{\vvv}^2\bigg).
\end{eqnarray}
Plugging this ansatz into the equation (\ref{mt 22}), we obtain a system of linear ordinary differential equations which is well-posed.
Hence, we can naturally obtain $\g_2$. Furthermore, $\g_2$ decays exponentially with respect to $\et$ as long as the boundary data are taken properly.\\
\ \\
Step 4: Construction of $\g_3$.\\
We may directly verify
\begin{eqnarray}
\ll\left[\ve\frac{\p \g_2}{\p\et}+G(\et)\bigg(\vp^2\dfrac{\p
\g_2}{\p\ve}-\ve\vp\dfrac{\p \g_2}{\p\vp}\bigg)+\ll[\g_2]+S_Q\right]=0
\end{eqnarray}
Then we may define $\g_3$ as the solution of the equation
\begin{eqnarray}
\\
\left\{
\begin{array}{rcl}\displaystyle
\ve\frac{\p \g_3}{\p\et}+G(\et)\bigg(\vp^2\dfrac{\p
\g_3}{\p\ve}-\ve\vp\dfrac{\p \g_3}{\p\vp}\bigg)+\ll[\g_3]
&=&\ve\dfrac{\p \g_2}{\p\et}+G(\et)\bigg(\vp^2\dfrac{\p
\g_2}{\p\ve}-\ve\vp\dfrac{\p \g_2}{\p\vp}\bigg)+\ll[\g_2]+S_Q,\\\rule{0ex}{1.5em} \g_3(0,\vvv)&=&-\g_2(0,\vvv)\ \ \text{for}\ \
\ve>0,\\\rule{0ex}{1.0em}\displaystyle\int_{\r^2}\ve
\sqrt{\m}\g_3(0,\vvv)\ud{\vvv}
&=&0\\
\lim_{\et\rt\infty}\g_3(\et,\vvv)&=&\g_{3,\infty}(\vvv).\no
\end{array}
\right.
\end{eqnarray}
We can obtain $\g_3$ is well-posed.\\
\ \\
Step 5: Construction of $\g_4$.\\
We may directly verify $\g_4=\g_2+\g_3$ satisfies the equation
\begin{eqnarray}
\left\{
\begin{array}{rcl}\displaystyle
\ve\frac{\p \g_4}{\p\et}+G(\et)\bigg(\vp^2\dfrac{\p
\g_4}{\p\ve}-\ve\vp\dfrac{\p \g_4}{\p\vp}\bigg)+\ll[\g_4]
&=&S_Q,\\\rule{0ex}{1.0em} \g_4(0,\vvv)&=&h(\vvv)\ \ \text{for}\ \
\ve>0,\\\rule{0ex}{1.0em}\displaystyle\int_{\r^2}\ve
\sqrt{\m}\g_4(0,\vvv)\ud{\vvv}
&=&0\\
\lim_{\et\rt\infty}\g_4(\et,\vvv)&=&\g_{4,\infty}(\vvv).
\end{array}
\right.
\end{eqnarray}
\ \\
In summary, we know $\g=\g_1+\g_4$ satisfies the equation (\ref{infinite slab LT general}) with zero mass-flux and is well-posed.
\end{proof}
\begin{lemma}\label{Milne lemma 1}
Assume (\ref{Milne bounded}) and (\ref{Milne decay}) hold. There
exists a unique solution $\g(\et,\vvv)$ to the $\e$-Milne problem
(\ref{Milne}) satisfying
\begin{eqnarray}
\tnnm{\g-\g_{\infty}}\leq C.
\end{eqnarray}
\end{lemma}
\begin{proof}
Taking $\g_{\infty}=\qi_{\infty}$, we can naturally obtain the
desired result.
\end{proof}
Then we turn to the construction of $\tilde h$ and the well-posedness of the equation (\ref{Milne transform}).
\begin{theorem}\label{Milne theorem 1}
Assume (\ref{Milne bounded}) and (\ref{Milne decay}) hold. There
exists $\tilde h$ satisfying the condition (\ref{Milne transform
compatibility}) such that there exists a unique solution
$\gg(\et,\vvv)$ to the $\e$-Milne problem (\ref{Milne transform})
satisfying
\begin{eqnarray}
\tnnm{\gg}\leq C.
\end{eqnarray}
\end{theorem}
\begin{proof}
The key part is the construction of $\tilde h$. Our main idea is to find $\tilde h\in\nk$ such that the equation
\begin{eqnarray}\label{mt 82}
\left\{
\begin{array}{rcl}\displaystyle
\ve\frac{\p \tilde\g}{\p\et}+G(\et)\bigg(\vp^2\dfrac{\p
\tilde\g}{\p\ve}-\ve\vp\dfrac{\p
\tilde\g}{\p\vp}\bigg)+\ll[\tilde\g]
&=&0,\\
\tilde\g(0,\vvv)&=&\tilde h(\vvv)\ \ \text{for}\ \ \ve>0,\\
\displaystyle\int_{\r^2}\ve\sqrt{\m}
\tilde\g(0,\vvv)\ud{\vvv} &=&\displaystyle\int_{\r^2}\ve\sqrt{\m}
\tilde h(\vvv)\ud{\vvv},\\
\displaystyle\lim_{\et\rt\infty}\tilde\g(\et,\vvv)&=&\tilde\g_{\infty}(\vvv),
\end{array}
\right.
\end{eqnarray}
for $\tilde g(\et,\vvv)$ is well-posed, where
\begin{eqnarray}
\tilde\g_{\infty}(\vvv)=\g_{\infty}(\vvv)=\qi_{0,\infty}\psi_0+\qi_{1,\infty}\psi_1+\qi_{2,\infty}\psi_2+\qi_{3,\infty}\psi_3,
\end{eqnarray}
is given by the equation of $\g$. Note that
\begin{eqnarray}
\tilde h^{\e}(\vvv)=\tilde D_0^{\e}\psi_0+\tilde D_1^{\e}\psi_1+\tilde
D_2^{\e}\psi_2+\tilde D_3^{\e}\psi_3.
\end{eqnarray}
We consider the endomorphism $T$ in $\nk$ defined as $T:\tilde h\rt T[\tilde h]=\tilde g_{\infty}$. Therefore, we only need to study the matrix of $T$ at the basis $\{\psi_0,\psi_1,\psi_2,\psi_3\}$.
It is easy to check when $\tilde
h=\psi_0$ and $\tilde h=\psi_3$, $T$ is an identity mapping, i.e.
\begin{eqnarray}
T[\psi_0]&=&\psi_0\\
T[\psi_3]&=&\psi_3
\end{eqnarray}
Multiplying $\sqrt{\m}$ on both sides of (\ref{mt 82}) and integrating over $\vvv\in\r^2$ imply conserved mass-flux, which further leads to
\begin{eqnarray}
T[\psi_1]&=&\psi_1
\end{eqnarray}
The main obstacle is when $\tilde h=\psi_2$. In this case, define $\tilde\g'=\tilde\g-\psi_2$. Then $\tilde\g'$ satisfies the equation
\begin{eqnarray}
\left\{
\begin{array}{rcl}\displaystyle
\ve\frac{\p \tilde\g'}{\p\et}+G(\et)\bigg(\vp^2\dfrac{\p \tilde\g'}{\p\ve}-\ve\vp\dfrac{\p \tilde\g'}{\p\vp}\bigg)+\ll[\tilde\g']
&=&G(\et)\sqrt{\m}\ve\vp,\\
\tilde\g'(0,\vvv)&=&0\ \ \text{for}\ \ \ve>0,\\
\displaystyle\int_{\r^2}\ve\sqrt{\m} \tilde\g'(0,\vvv)\ud{\vvv} &=&0,\\
\displaystyle\lim_{\et\rt\infty}\tilde\g'(\et,\vvv)&=&\tilde\g'_{\infty}(\vvv).
\end{array}
\right.
\end{eqnarray}
Although $G(\et)\sqrt{\m}\ve\vp$ does not decay exponentially, based on Lemma \ref{Milne force}, $L^1$ and $L^2$ norm of
$G$ can be sufficiently small as $\e\rt0$ and $G(\et)\sqrt{\m}\ve\vp\in\nk$. Using a natural extension of Lemma \ref{Milne prelim lemma 2} for $\ll[S]=0$, we know
$\abs{\tilde\qi'_{\infty}}$ is also sufficiently small, where $\tilde\qi'$ is the projection of $\tilde\g'$ on $\nk$. Note that we do not need exponential decay of source term in order to show the bound of $\tilde\qi_{\infty}$. This means
\begin{eqnarray}
T[\psi_0,\psi_1,\psi_2,\psi_3]=[\psi_0,\psi_1,\psi_2,\psi_3]\left(
\begin{array}{cccc}
1&0&\tilde\qi'_{0,\infty}&0\\
0&1&\tilde\qi'_{1,\infty}&0\\
0&0&1+\tilde\qi'_{2,\infty}&0\\
0&0&\tilde\qi'_{3,\infty}&1\\
\end{array}
\right)
\end{eqnarray}
For $\e$ sufficiently small, this matrix is invertible, which means $T$ is bijective. Therefore, we can always find $\tilde h$ such that
$\tilde\g_{\infty}=\g_{\infty}$, which is desired. Then by Lemma
\ref{Milne lemma 1} and superposition property, when define $\gg^{\e}=\g^{\e}-\tilde\g$, the theorem naturally follows.
\end{proof}
In the $\e$-Milne problem (\ref{Milne transform}), even if the boundary data and source term are determined, we still have the freedom to choose mass-flux $m_f[\g]$. Next theorem shows that we can adjust the mass-flux $m_f[\gg]$ to obtain desired properties of $\tilde h$.
\begin{theorem}\label{Milne adjustment}
Assume (\ref{Milne bounded}) and (\ref{Milne decay}) hold. In the $\e$-Milne problem (\ref{Milne transform}), for any constant $C_0$, there
exists an $m_f[\g]$ such that $\tilde\gamma=\tilde D_0^{\e}+\tilde D_3^{\e}=C_0$.
\end{theorem}
\begin{proof}
The key point is to study the equation for $\bar\g$ as
\begin{eqnarray}\label{Milne flux}
\left\{
\begin{array}{rcl}\displaystyle
\ve\frac{\p \bar\g}{\p\et}+G(\e;\et)\bigg(\vp^2\dfrac{\p
\bar\g}{\p\ve}-\ve\vp\dfrac{\p \bar\g}{\p\vp}\bigg)+\ll[\bar\g]
&=&0,\\
\bar\g(0,\vvv)&=&0\ \ \text{for}\ \
\ve>0,\\\rule{0ex}{1em} \displaystyle\int_{\r^2}
\ve\sqrt{\m}\bar\g(0,\vvv)\ud{\vvv}
&=&m_f[\bar\g],\\\rule{0ex}{1.0em}
\displaystyle\lim_{\et\rt\infty}\bar\g(\et,\vvv)&=&\bar\g_{\infty}(\vvv),
\end{array}
\right.
\end{eqnarray}
where
\begin{eqnarray}
\bar\g_{\infty}(\vvv)&=&E_0\psi_0+E_1\psi_1+E_2\psi_2+E_3\psi_3,
\end{eqnarray}
with zero boundary data and source term but non-vanishing mass-flux, i.e. $m_f[\bar\g]\neq0$.
We claim $E_1+E_3\neq 0$. If this claim is true, then by superposition property, in the equation (\ref{Milne}), we can obtain the desired $\gamma=\g_{0,\infty}+\g_{3,\infty}$ by adding a multiple of the equation (\ref{Milne flux}). Then as in the proof of Theorem \ref{Milne theorem 1}, we know the endomorphism $T$ leads to $\g_{0,\infty}+\g_{3,\infty}=\tilde D_0^{\e}+\tilde D_3^{\e}$. Then our work is done.

Next, we prove this claim by contradiction. Let us assume the claim is not true, i.e $E_1+E_3=0$ for some $m_f[\bar\g]\neq0$.
We decompose $\bar\g=\bar\wi+\bar\qi$. By the construction in the proof of Lemma \ref{Milne prelim lemma 3}, we know
\begin{eqnarray}\label{mt 42}
0<\tnnm{\bar\wi}\leq C.
\end{eqnarray}
Note there the first inequality is valid since we can directly verify $\bar\g\in\nk$ cannot be a solution.
Define the linearized entropy as
\begin{eqnarray}
H[\bar\g](\et)=\br{\ve\bar\g,\bar\g}(\et).
\end{eqnarray}
Multiplying
$\bar\g$ on both sides of
(\ref{Milne flux}) and integrating over $\vvv$, we get
\begin{eqnarray}
\half\frac{\ud{}}{\ud{\et}}\br{\ve\bar\g,\bar\g}=\br{\bar\wi,\ll[\bar\wi]}-\half G(\et)\br{\ve\bar\g,\bar\g}.
\end{eqnarray}
Hence, we have
\begin{eqnarray}\label{mt 41}
\half\frac{\ud{}}{\ud{\et}}\bigg(\ue^{W}\br{\ve\bar\g,\bar\g}\bigg)=-\ue^{W}\br{\bar\wi,\ll[\bar\wi]},
\end{eqnarray}
which implies $\ue^{W}H[\bar g]$ is decreasing. Furthermore, we have
\begin{eqnarray}
\ue^{W(\et)}H[\bar\g](\et)=H[\bar\g](0)-\int_0^{\et}\ue^{W(y)}\br{\bar\wi,\ll[\bar\wi]}(y)\ud{y}<\infty.
\end{eqnarray}
Hence, we can take a subsequence such that $\tnm{\sn\bar\wi(\et_n)}$ goes to zero. Then we can always assume $\bar\qi(\et_n)$ goes to $\bar\qi_{\infty}$.
Therefore, we have
\begin{eqnarray}
\ue^{W(\et_n)}H[\bar\g](\et_n)\rt \br{\ve\bar\qi_{\infty}, \bar\qi_{\infty}}=2E_1(E_0+E_3)=0.
\end{eqnarray}
Then we naturally obtain
\begin{eqnarray}
\ue^{W(\et_n)}H[\bar\g](\et_n)\rt 0\ \ \text{as}\ \ \et_n\rt\infty.
\end{eqnarray}
Hence, we have
\begin{eqnarray}
\ue^{W(\et)}H[\bar\g](\et)\geq0,
\end{eqnarray}
and
\begin{eqnarray}
\ue^{W(\et)}H[\bar\g](\et)\rt 0\ \ \text{as}\ \ \et\rt\infty.
\end{eqnarray}
In (\ref{mt 41}), integrating over $[0,\infty)$, we achieve
\begin{eqnarray}
-\int_{\ve<0}\ve\bar\g^2(0)\ud{\vvv}+\int_0^{\infty}\ue^{W(\et)}\br{\bar\wi,\ll[\bar\wi]}(\et)\ud{\et}=\int_{\ve>0}\ve\bar\g^2(0)\ud{\vvv}=0.
\end{eqnarray}
Hence, we have
\begin{eqnarray}
\int_{\ve<0}\ve\bar\g^2(0)\ud{\vvv}=\int_0^{\infty}\br{\bar\wi,\ll[\bar\wi]}(\et)\ud{\et}=0,
\end{eqnarray}
which implies $\bar\g(0)=0$ and $\bar\wi=0$. This contradicts (\ref{mt 42}). Therefore, the claim is valid.
\end{proof}

\subsection{$L^{\infty}$ Estimates}

\subsubsection{Mild formulation in a finite slab}

Consider the $\e$-transport problem for $\gl(\et,\vvv)$ in a
finite slab
\begin{eqnarray}\label{finite slab LI}
\left\{
\begin{array}{rcl}\displaystyle
\ve\frac{\p \gl}{\p\et}+G(\et)\bigg(\vp^2\dfrac{\p
\gl}{\p\ve}-\ve\vp\dfrac{\p \gl}{\p\vp}\bigg)+\nu\gl
&=&Q(\et,\vvv),\\\rule{0ex}{1.0em} \gl(0,\vvv)&=&h(\vvv)\ \
\text{for}\ \ \ve>0,\\\rule{0ex}{1.0em}
\gl(L,R[\vvv])&=&\gl(\vvv),
\end{array}
\right.
\end{eqnarray}
We define the characteristics starting from
$(\et(0),\ve(0),\vp(0))$ as $(\et(s),\ve(s),\vp(s))$ defined
by
\begin{eqnarray}
\frac{\ud{\et}}{\ud{s}}&=&\ve\\
\frac{\ud{\ve}}{\ud{s}}&=&G(\et)\vp^2\\
\frac{\ud{\vp}}{\ud{s}}&=&-G(\et)\ve\vp
\end{eqnarray}
which leads to
\begin{eqnarray}
\ve^2(s)+\vp^2(s)&=&C_1,\\
\vp(s)\ue^{-W(s)}&=&C_2,
\end{eqnarray}
where $C_1$ and $C_2$ are two constants depending on the starting
point. Along the characteristics, the equation (\ref{finite slab LI}) can be
rewritten as
\begin{eqnarray}
\ve\frac{\p \g}{\p\et}+\nu\g&=&Q.
\end{eqnarray}
Define the energy
\begin{eqnarray}
E(\et,\vvv)=\ve^2(\et)+\vp^2(\et).
\end{eqnarray}
and
\begin{eqnarray}
\vp'(\et,\vvv;\et')&=&\vp e^{W(\et')-W(\et)}.
\end{eqnarray}
For $E\geq\vp'^2$, define
\begin{eqnarray}
\ve'(\et,\vvv;\et')&=&\sqrt{E-\vp'^2(\et,\vvv;\et')},\\
\vvv'(\et,\et')&=&(\ve'(\et,\vvv;\et'),\vp'(\et,\vvv;\et')),\\
R[\vvv'(\et,\et')]&=&(-\ve'(\et,\vvv;\et'),\vp'(\et,\vvv;\et')).
\end{eqnarray}
Basically, this means $(\et,\ve,\vp)$ and $(\et',\ve',\vp')$ are on the same characteristics.
Moreover, define an implicit function $\et^{+}(\et,\vvv)$ by the
equation
\begin{eqnarray}
E(\et,\vvv)=\vp'^2(\et,\vvv;\et^+).
\end{eqnarray}
We know $(\et^+,\vvv)$ at the axis $\ve=0$ is on the same characteristics as $(\et,\vvv)$. Finally put
\begin{eqnarray}
G_{\et,\et'}&=&\int_{\et'}^{\et}\frac{\nu(\vvv'(\et,y))}{\ve'(\et,\vvv,y)}\ud{y},\\
R[G_{\et,\et'}]&=&\int_{\et'}^{\et}\frac{\nu(R[\vvv'(\et,y)])}{\ve'(\et,\vvv,y)}\ud{y}.
\end{eqnarray}
We can rewrite the solution to the equation (\ref{finite slab LI})
along the characteristics
as follows:\\
\ \\
Case I:\\
For $\ve>0$,
\begin{eqnarray}\label{mt 06}
\gl(\et,\vvv)&=&h(\vvv'(\et,\vvv; 0))\exp(-G_{\et,0})+\int_0^{\et}\frac{Q(\et',\vvv(\et,\vvv;\et'))}{\ve'(\et,\vvv,\et')}\exp(-G_{\et,\et'})\ud{\et'}.
\end{eqnarray}
\ \\
Case II:\\
For $\ve<0$ and $\abs{E(\et,\vvv)}\geq \vp'(\et,\vvv;L)$,
\begin{eqnarray}\label{mt 07}
\gl(\et,\vvv)&=&h(\vvv'(\et,\vvv; 0))\exp(-G_{L,0}-R[G_{L,\et}])\\
&+&\bigg(\int_0^{L}\frac{Q(\et',\vvv(\et,\vvv;\et'))}{\ve'(\et,\vvv,\et')}
\exp(-G_{L,\et'}-R[G_{L,\et}])\ud{\et'}\no\\
&&+\int_{\et}^{L}\frac{Q(\et',R[\vvv(\et,\vvv;\et')])}{\ve'(\et,\vvv,\et')}\exp(R[G_{\et,\et'}])\ud{\et'}\bigg)\no.
\end{eqnarray}
\ \\
Case III:\\
For $\ve<0$ and $\abs{E(\et,\vvv)}\leq \vp'(\et,\vvv;L)$,
\begin{eqnarray}\label{mt 08}
\gl(\et,\vvv)&=&h(\vvv'(\et,\vvv; 0))\exp(-G_{\et^+,0}-R[G_{\et^+,\et}])\\
&+&\bigg(\int_0^{\et^+}\frac{Q(\et',\vvv(\et,\vvv;\et'))}{\ve'(\et,\vvv,\et')}
\exp(-G_{\et^+,\et'}-R[G_{\et^+,\et}])\ud{\et'}\no\\
&&+\int_{\et}^{\et^+}\frac{Q(\et',R[\vvv(\et,\vvv;\et')])}{\ve'(\et,\vvv,\et')}\exp(R[G_{\et,\et'}])\ud{\et'}\bigg)\no.
\end{eqnarray}

\subsubsection{Mild formulation in an infinite slab}

Consider the $\e$-transport problem for $\g(\et,\vvv)$ in an
infinite slab
\begin{eqnarray}\label{infinite slab LI}
\left\{
\begin{array}{rcl}\displaystyle
\ve\frac{\p \g}{\p\et}+G(\et)\bigg(\vp^2\dfrac{\p
\g}{\p\ve}-\ve\vp\dfrac{\p \g}{\p\vp}\bigg)+\nu\g
&=&Q(\et,\vvv),\\\rule{0ex}{1.0em} \g(0,\vvv)&=&h(\vvv)\ \
\text{for}\ \ \ve>0,\\\rule{0ex}{1.0em}
\lim_{\et\rt\infty}\g(\et,\vvv)&=&\g_{\infty}(\vvv),
\end{array}
\right.
\end{eqnarray}
We can define the solution via taking limit $L\rt\infty$ in (\ref{mt
06}), (\ref{mt 07}) and (\ref{mt 08}) as follows:
\begin{eqnarray}
\g(\et,\vvv)=\a[h(\vvv)]+\t[Q(\et,\vvv)],
\end{eqnarray}
where\\
\ \\
Case I: \\For $\ve>0$,
\begin{eqnarray}\label{mt 09}
\a[h(\vvv)]&=&h(\vvv'(\et,\vvv; 0))\exp(-G_{\et,0}),\\
\t[Q(\et,\vvv)]&=&\int_0^{\et}\frac{Q(\et',\vvv(\et,\vvv;\et'))}{\ve'(\et,\vvv,\et')}\exp(-G_{\et,\et'})\ud{\et'}.
\end{eqnarray}
\ \\
Case II: \\
For $\ve<0$ and $\abs{E(\et,\vvv)}\geq \vp'(\et,\vvv;\infty)$,
\begin{eqnarray}\label{mt 10}
\a[h(\vvv)]&=&0,\\
\t[Q(\et,\vvv)]&=&\int_{\et}^{\infty}\frac{Q(\et',R[\vvv(\et,\vvv;\et')])}{\ve'(\et,\vvv,\et')}\exp(R[G_{\et,\et'}])\ud{\et'}.
\end{eqnarray}
\ \\
Case III: \\
For $\ve<0$ and $\abs{E(\et,\vvv)}\leq \vp'(\et,\vvv;\infty)$,
\begin{eqnarray}\label{mt 11}
\\
\a[h(\vvv)]&=&h(\vvv'(\et,\vvv; 0))\exp(-G_{\et^+,0}-R[G_{\et^+,\et}]),\no\\
\\
\t[Q(\et,\vvv)]&=&\bigg(\int_0^{\et^+}\frac{Q(\et',\vvv(\et,\vvv;\et'))}{\ve'(\et,\vvv,\et')}
\exp(-G_{\et^+,\et'}-R[G_{\et^+,\et}])\ud{\et'}\no\\
&&+
\int_{\et}^{\et^+}\frac{Q(\et',R[\vvv(\et,\vvv;\et')])}{\ve'(\et,\vvv,\et')}\exp(R[G_{\et,\et'}])\ud{\et'}\bigg)\no.
\end{eqnarray}
Notice that
\begin{eqnarray}
\lim_{L\rt\infty}\exp(-G_{L,\et})=0,
\end{eqnarray}
for $\ve<0$ and $\abs{E(\et,\vvv)}\leq
\vp'(\et,\vvv;\infty)$,. Hence, above derivation is valid. In
order to achieve the estimate of $\g$, we need to control $\a[h]$ and $\t[Q]$.

\subsubsection{Preliminaries}

\begin{lemma}\label{Milne prelim lemma 4}
There is a positive $0<\beta<\nu_0$ such that for any $\vt\geq0$ and
$0\leq\zeta\leq1/4$,
\begin{eqnarray}
\lnm{\ue^{\beta\et}\a[h]}{\vt,\ze}\leq C\lnm{h}{\vt,\ze}.
\end{eqnarray}
\end{lemma}
\begin{proof}
Based on Lemma \ref{Milne property},
we know
\begin{eqnarray}
\frac{\nu(\vvv'(\et,y))}{\ve'(\et,\vvv,y)}&\geq&\nu_0\\
\frac{\nu(R[\vvv'(\et,y)])}{\ve'(\et,\vvv,y)}&\geq&\nu_0
\end{eqnarray}
It follows that
\begin{eqnarray}
\exp(-G_{\et,0})&\leq&\ue^{-\beta\et}\\
\exp(-G_{\et^+,0}-R[G_{\et^+,\et}])&\leq&\ue^{-\beta\et}
\end{eqnarray}
Then our results are obvious.
\end{proof}
\begin{lemma}\label{Milne prelim lemma 5}
For any integer $\vt\geq0$, $0\leq\zeta\leq1/4$ and
$\beta\leq\nu_0/2$, there is a constant $C$ such that
\begin{eqnarray}
\lnnm{\t[Q]}{\vt,\ze}\leq C\lnnm{\frac{Q}{\nu}}{\vt,\ze}.
\end{eqnarray}
Moreover, we have
\begin{eqnarray}
\lnnm{\ue^{\beta\et}\t[Q]}{\vt,\ze}\leq
C\lnnm{\frac{\ue^{\beta\et}Q}{\nu}}{\vt,\ze}.
\end{eqnarray}
\end{lemma}
\begin{proof}
The first inequality is a special case of the second one, so we only need to prove the second inequality.
For $\ve>0$ case, we have
\begin{eqnarray}
\beta(\et-\et')-G_{\et,\et'}\leq\beta(\et-\et')-\frac{\nu_0(\et-\et')}{2}-\frac{G_{\et,\et'}}{2}\leq-\frac{G_{\et,\et'}}{2}.
\end{eqnarray}
It is natural that
\begin{eqnarray}
\int_0^{\et}\frac{\nu(\vvv'(\et,\et'))}{\ve'(\et,\vvv,\et')}
\exp(\beta(\et-\et')-G_{\et,\et'})\ud{\et'}\leq\int_0^{\infty}\exp\bigg(-\frac{z}{2}\bigg)\ud{z}=2.
\end{eqnarray}
Then we estimate
\begin{eqnarray}
\abs{\bvv\ue^{\beta\et}\t[Q]}&\leq& \ue^{\beta\et}\int_0^{\et}\bvv\frac{\abs{Q(\et',\vvv(\et,\vvv;\et'))}}{\ve'(\et,\vvv,\et')}\exp(-G_{\et,\et'})\ud{\et'}\\
&\leq&\lnnm{\frac{\ue^{\beta\et}Q}{\nu}}{\vt,\ze}\int_0^{\et}\frac{\nu(\vvv'(\et,\et'))}{\ve'(\et,\vvv,\et')}
\exp(\beta(\et-\et')-G_{\et,\et'})\ud{\et'}\no\\
&\leq&C\lnnm{\frac{\ue^{\beta\et}Q}{\nu}}{\vt,\ze}.\no
\end{eqnarray}
The $\ve<0$ case can be proved in a similar fashion, so we omit it here.
\end{proof}
\begin{lemma}\label{Milne prelim lemma 6}
For any $\d>0$, $\vt>2$ and $0\leq\zeta\leq1/4$, there is a
constant $C(\d)$ such that
\begin{eqnarray}
\ltnm{\t[Q]}{\ze}\leq C(\d)\tnnm{\nu^{-1/2}Q}+\d\lnnm{Q}{\vt,\ze}.
\end{eqnarray}
\end{lemma}
\begin{proof}
We divide the proof into several cases:\\
\ \\
Case I: For $\ve>0$,
\begin{eqnarray}
\t[Q(\et,\vvv)]&=&\int_0^{\et}\frac{Q(\et',\vvv(\et,\vvv;\et'))}{\ve'(\et,\vvv,\et')}\exp(-G_{\et,\et'})\ud{\et'}.
\end{eqnarray}
We need to estimate
\begin{eqnarray}
\int_{\r^2}\ue^{2\ze\abs{\vvv}^2}\bigg(\int_0^{\et}\frac{Q(\et',\vvv(\et,\vvv;\et'))}{\ve'(\et,\vvv,\et')}\exp(-G_{\et,\et'})\ud{\et'}
\bigg)^2\ud{\vvv}.
\end{eqnarray}
Assume $m>0$ is sufficiently small, $M>0$ is sufficiently large and $\sigma>0$ is sufficiently small which will be determined in the following. We can split the integral into the following parts
\begin{eqnarray}
I=I_1+I_2+I_3+I_4.
\end{eqnarray}
\ \\
Case I - Type I: $\chi_1$: $M\leq\ve'(\et,\vvv,\et')$ or $M\leq\vp'(\et,\vvv,\et')$.\\
By Lemma \ref{Milne property}, we have
\begin{eqnarray}
\abs{\vvv(\et,\vvv;\et')}+1\leq C\nu(\vvv(\et,\vvv;\et')).
\end{eqnarray}
Then for $\vt>2$, since $\abs{\vvv}$ is conserved along the characteristics,  we have
\begin{eqnarray}
I_1&\leq&C\lnnm{Q}{\vt,\ze}^2\int_{\r^2}\chi_1\bigg(\int_0^{\et}\frac{1}{\br{\vvv'}^{\vt}}\frac{\exp(-G_{\et,\et'})}{\ve'(\et,\vvv,\et')}\ud{\et'}
\bigg)^2\ud{\vvv}\\
&\leq&\frac{C}{M^{\vt}}\lnnm{Q}{\vt,\ze}^2\int_{\r^2}\frac{1}{\br{\vvv}^{\vt}}\bigg(\int_0^{\et}\frac{\exp(-G_{\et,\et'})}{\ve'(\et,\vvv,\et')}\ud{\et'}
\bigg)^2\ud{\vvv}\no\\
&\leq&\frac{C}{M^{\vt}}\lnnm{Q}{\vt,\ze}^2.\no
\end{eqnarray}
since
\begin{eqnarray}
\abs{\int_0^{\et}\frac{\exp(-G_{\et,\et'})}{\ve'(\et,\vvv,\et')}\ud{\et'}}&\leq& \abs{\int_0^{\et}\frac{-\nu(\vvv'(\et,\et'))\exp(G_{\et,\et'})}{\ve'(\et,\vvv,\et')}\ud{\et'}}\\
&\leq&\int_0^{\infty}\ue^{-y}\ud{y}=1.\no
\end{eqnarray}
\ \\
Case I - Type II: $\chi_2$: $m\leq\ve'(\et,\vvv,\et')\leq M$ and $\vp'(\et,\vvv,\et')\leq M$.\\
Since along the characteristics, $\abs{\vvv}^2$ can be bounded by
$2M^2$ and the integral domain for $\vvv$ is finite. Then by Cauchy's inequality, we have
\begin{eqnarray}
I_2&\leq&C\frac{\ue^{4\ze M^2}}{m}\int_0^{\et}\frac{Q^2}{\nu}(\et',\vvv(\et,\vvv;\et'))\ud{\et'}
\int_0^{\et}\frac{\nu(\vvv'(\et,\et'))\exp(-2G_{\et,\et'})}{\ve'(\et,\vvv,\et')}\ud{\et'}
\\
&\leq& C\frac{\ue^{4\ze M^2}}{m}\tnnm{\nu^{-1/2}Q}^2,\no
\end{eqnarray}
where
\begin{eqnarray}
\int_0^{\et}\frac{\nu(\vvv'(\et,\et'))}{\ve'(\et,\vvv,\et')}\exp(-2G_{\et,\et'})\ud{\et'}
\ud{\vvv}&\leq&\int_0^{\infty}\ue^{-2y}\ud{y}=\half.
\end{eqnarray}
\ \\
Case I - Type III: $\chi_3$: $0\leq\ve'(\et,\vvv,\et')\leq m$, $\vp'(\et,\vvv,\et')\leq M$ and $\et-\et'\geq\sigma$.\\
In this case, we know
\begin{eqnarray}
G_{\et,\et'}\geq\frac{\sigma}{m}.
\end{eqnarray}
Then after substitution, the integral is not from zero, but from
$-\sigma/m$. Hence, we have
\begin{eqnarray}
I_3&\leq&C\lnnm{Q}{\vt,\ze}^2\int_{\r^2}\chi_3\bigg(\int_0^{\et}\frac{1}{\br{\vvv'}^{\vt}}\frac{\exp(-G_{\et,\et'})}{\ve'(\et,\vvv,\et')}\ud{\et'}
\bigg)^2\ud{\vvv}\\
&\leq&C\lnnm{Q}{\vt,\ze}^2\int_{\r^2}\frac{\chi_3}{\br{\vvv}^{2\vt}}\bigg(\int_0^{\et}\frac{\nu(\vvv'(\et,\et'))\exp(-G_{\et,\et'})}{\ve'(\et,\vvv,\et')}\ud{\et'}
\bigg)^2\ud{\vvv}\no\\
&\leq&C\lnnm{Q}{\vt,\ze}^2\bigg(\int_{-\sigma/m}^{\infty}\ue^{-y}\ud{y}
\bigg)^2\no\\
&\leq& C\ue^{-\frac{\sigma}{m}}\lnnm{Q}{\vt,\ze}^2.\no
\end{eqnarray}
\ \\
Case I - Type IV: $\chi_4$: $0\leq\ve'(\et,\vvv,\et')\leq m$, $\vp'(\et,\vvv,\et')\leq M$ and $\et-\et'\leq\sigma$.\\
For $\et'\leq\et$ and $\et-\et'\leq\sigma$, we have
\begin{eqnarray}
\ve\leq
C\ve'(\et,\vvv,\et')\leq C(m+\sigma).
\end{eqnarray}
Therefore, the integral domain for $\ve$ is very small. We have the estimate
\begin{eqnarray}
I_4&\leq&C\lnnm{Q}{\vt,\ze}\int_{\r^2}\chi_4\bigg(\int_0^{\et}\frac{1}{\br{\vvv'}^{\vt}}\frac{\exp(-G_{\et,\et'})}{\ve'(\et,\vvv,\et')}\ud{\et'}
\bigg)^2\ud{\vvv}\\
&\leq&C\lnnm{Q}{\vt,\ze}\int_{\r^2}\frac{\chi_4}{\br{\vvv}^{\vt}}\ud{\vvv}\no\\
&\leq&C(m+\sigma)\lnnm{Q}{\vt,\ze}.\no
\end{eqnarray}
\ \\
Collecting all four types, we have
\begin{eqnarray}
I\leq C\frac{\ue^{4\ze
M^2}}{m}\tnnm{\nu^{-1/2}Q}+C\bigg(\frac{1}{M^{\vt}}+m+\sigma+\ue^{-\frac{\sigma}{m}}\bigg)\lnnm{Q}{\vt,\ze}.
\end{eqnarray}
Taking $M$ sufficiently large, $\sigma$ sufficiently small and
$m<<\sigma$, this is
the desired result.\\
\ \\
Case II: \\
For $\ve<0$ and $\abs{E(\et,\vvv)}\geq \vp'(\et,\vvv;\infty)$,
\begin{eqnarray}
\t[Q(\et,\vvv)]&=&\int_{\et}^{\infty}\frac{Q(\et',R[\vvv(\et,\vvv;\et')])}{\ve'(\et,\vvv,\et')}\exp(R[G_{\et,\et'}])\ud{\et'}.
\end{eqnarray}
We need to estimate
\begin{eqnarray}
\int_{\r^2}\ue^{2\ze\abs{\vvv}^2}\bigg(\int_{\et}^{\infty}\frac{Q(\et',R[\vvv(\et,\vvv;\et')])}{\ve'(\et,\vvv,\et')}
\exp(R[G_{\et,\et'}])\ud{\et'}
\bigg)^2\ud{\vvv}.
\end{eqnarray}
We can split the integral into the following types:
\begin{eqnarray}
II=II_1+II_2+II_3.
\end{eqnarray}
\ \\
Case II - Type I: $\chi_1$: $M\leq\ve'(\et,\vvv,\et')$ or $M\leq\vp'(\et,\vvv,\et')$.\\
Similar to Case I - Type I, we have
\begin{eqnarray}
II_1&\leq&C\lnnm{Q}{\vt,\ze}\int_{\r^2}\chi_1\bigg(\int_{\et}^{\infty}\frac{1}{\br{\vvv'}^{\vt}}\frac{\exp(R[G_{\et,\et'}])}{\ve'(\et,\vvv,\et')}\ud{\et'}
\bigg)^2\ud{\vvv}\\
&\leq&C\frac{1}{M^{\vt}}\lnnm{Q}{\vt,\ze}.\no
\end{eqnarray}
\ \\
Case II - Type II: $\chi_2$: $m\leq\ve'(\et,\vvv,\et')\leq M$ and $\vp'(\et,\vvv,\et')\leq M$.\\
Similar to Case I - Type II, by Cauchy's inequality, we have
\begin{eqnarray}
II_2&\leq&C\frac{\ue^{4\ze M^2}}{m}\int_{\et}^{\infty}\frac{Q^2}{\nu}(\et',\vvv(\et,\vvv;\et'))\ud{\et'}
\int_{\et}^{\infty}\frac{\exp(2R[G_{\et,\et'}])}{\ve'(\et,\vvv,\et')}\ud{\et'}
\\
&\leq& C\frac{\ue^{4\ze M^2}}{m}\tnnm{\nu^{-1/2}Q}.\no
\end{eqnarray}
\ \\
Case I - Type III: $\chi_3$: $0\leq\ve'(\et,\vvv,\et')\leq m$ and $\vp'(\et,\vvv,\et')\leq M$.\\
In this case, we can directly verify the fact
\begin{eqnarray}
\ve\leq\ve'(\et,\vvv,\et')
\end{eqnarray}
for $\et\leq\et'$. Then we know the integral of $\ve$ is always in a small domain.
Similar to Case I - Type IV, we have the estimate
\begin{eqnarray}
II_3\leq Cm\lnnm{Q}{\vt,\ze}.
\end{eqnarray}
Hence, collecting all three types, we obtain
\begin{eqnarray}
II\leq C\frac{\ue^{4\ze
M^2}}{m}\tnnm{\nu^{-1/2}Q}+C\bigg(\frac{1}{M^{\vt}}+m\bigg)\lnnm{Q}{\vt,\ze}.
\end{eqnarray}
Taking $M$ sufficiently large and
$m$ sufficiently small, this is
the desired result.\\
\ \\
Case III: \\
For $\ve<0$ and $\abs{E(\et,\vvv)}\leq \vp'(\et,\vvv;\infty)$,
\begin{eqnarray}
\\
\t[Q(\et,\vvv)]&=&\bigg(\int_0^{\et^+}\frac{Q(\et',\vvv(\et,\vvv;\et'))}{\nu(\vvv(\et,\vvv;\et'))}
\exp(-G_{\et^+,\et'}-R[G_{\et^+,\et}])\ud{\et'}\no\\
&&+
\int_{\et}^{\et^+}\frac{Q(\et',R[\vvv(\et,\vvv;\et')])}{\nu(\vvv(\et,\vvv;\et'))}\exp(R[G_{\et,\et'}])\ud{\et'}\bigg)\no.
\end{eqnarray}
This is a combination of Case I and Case II, so it naturally holds.
\end{proof}

\subsubsection{Estimates of $\e$-Milne problem}

\begin{lemma}\label{Milne lemma 2}
Assume (\ref{Milne bounded}) and (\ref{Milne decay}) hold. The
solution $\g(\et,\vvv)$ to the $\e$-Milne problem (\ref{Milne})
satisfies for $\vt>2$ and $0\leq\zeta\leq1/4$,
\begin{eqnarray}
\lnnm{\g-\g_{\infty}}{\vt,\ze}\leq
C+C\tnnm{\g-\g_{\infty}}.
\end{eqnarray}
\end{lemma}
\begin{proof}
Define $u=\g-\g_{\infty}$. Then $u$ satisfies the equation
\begin{eqnarray}
\left\{
\begin{array}{rcl}\displaystyle
\ve\frac{\p u}{\p\et}+G(\et)\bigg(\vp^2\dfrac{\p
u}{\p\ve}-\ve\vp\dfrac{\p u}{\p\vp}\bigg)+\ll[u]
&=&S(\et,\vvv)+\g_{2,\infty}G(\et)\sqrt{\m}\ve\vp=\tilde
S,\\\rule{0ex}{1.0em} u(0,\vvv)&=&(h-\g_{\infty})(\vvv)=p(\vvv)\ \
\text{for}\ \ \ve>0,\\\rule{0ex}{1.0em}
\displaystyle\int_{\r^2}\ve\sqrt{\m} u(0,\vvv)\ud{\vvv}
&=&m_f[\g]-\g_{1,\infty},\\\rule{0ex}{1.5em}
\lim_{\et\rt\infty}u(\et,\vvv)&=&0,
\end{array}
\right.
\end{eqnarray}
Since $u=\a[p]+\t[K[u]+\tilde S]$,
based on Lemma \ref{Milne prelim lemma 6}, we have
\begin{eqnarray}
\\
\ltnm{u-\a[p]}{\ze}&=&\ltnm{\t[K[u]+\tilde S]}\no\\
&\leq&
C(\d)\bigg(\tnnm{\nu^{-1/2}K[u]}+\tnnm{\nu^{-1/2}\tilde
S}\bigg)+\d\bigg(\lnnm{K[u]}{\vt,\ze}+\lnnm{\tilde
S}{\vt,\ze}\bigg)\no\\
&\leq&
C(\d)\bigg(\tnnm{u}+\tnnm{\tilde
S}\bigg)+\d\bigg(\lnnm{K[u]}{\vt,\ze}+\lnnm{\tilde
S}{\vt,\ze}\bigg),\no
\end{eqnarray}
where we can directly verify
\begin{eqnarray}
\tnnm{\nu^{-1/2}K[u]}&\leq&\tnnm{u},\\
\tnnm{\nu^{-1/2}\tilde S}&\leq&\tnnm{\tilde S}.
\end{eqnarray}
In \cite[Lemma 3.3.1]{Glassey1996}, it is shown that
\begin{eqnarray}
\lnnm{K[u]}{\vt,\ze}&\leq&\lnnm{u}{\vt-1,\ze},\\
\lnnm{K[u]}{0,\ze}&\leq&\ltnm{u}{\ze}.
\end{eqnarray}
Since $u=\a[p]+\t[K[u]+\tilde S]$, for $\e$ and $\d$ sufficiently
small, we can estimate
\begin{eqnarray}
\\
\lnnm{u}{\vt,\ze}&\leq&C\bigg(\lnnm{\t[K[u]]}{\vt,\ze}+\lnnm{\t[S]}{\vt,\ze}+\lnm{\a[p]}{\vt,\ze}\bigg)\no\\
&\leq&C\bigg(\lnnm{K[u]}{\vt,\ze}+\lnnm{S}{\vt,\ze}+\lnm{\a[p]}{\vt,\ze}\bigg)\no\\
&\leq&C\bigg(\lnnm{u}{\vt-1,\ze}+\lnnm{S}{\vt,\ze}+\lnm{\a p}{\vt,\ze}\bigg)\no\\
&\leq&\ldots\no\\
&\leq&C\bigg(\lnnm{K[u]}{0,\ze}+\lnnm{S}{\vt,\ze}+\lnm{\a[p]}{\vt,\ze}\bigg)\no\\
&\leq&C\bigg(\ltnm{u}{\ze}+\lnnm{S}{\vt,\ze}+\lnm{\a[p]}{\vt,\ze}\bigg)\no\\
&\leq&C(\d)\bigg(\tnnm{\nu^{-1/2}K[u]}+\tnnm{\nu^{-1/2}\tilde
S}\bigg)+\d\bigg(\lnnm{K[u]}{\vt,\ze}+\lnnm{\tilde
S}{\vt,\ze}\bigg)\no\\
&&+C\bigg(\lnnm{S}{\vt,\ze}+\lnm{\a[p]}{\vt,\ze}\bigg).\no
\end{eqnarray}
Therefore, absorbing $\d\lnnm{K[u]}{\vt,\ze}$ into the right-hand side of the second inequality implies
\begin{eqnarray}
\lnnm{K[u]}{\vt,\ze}\leq C\bigg(\tnnm{\nu^{-1/2}K[u]}+\lnnm{S}{\vt,\ze}+\lnm{p}{\vt,\ze}\bigg)
\end{eqnarray}
Therefore, we have
\begin{eqnarray}
\lnnm{u}{\vt,\ze}
&\leq&C\bigg(\lnnm{K[u]}{\vt,\ze}+\lnnm{S}{\vt,\ze}+\lnm{\a[p]}{\vt,\ze}\bigg)\no\\
&\leq&C\bigg(\tnnm{u}+\lnnm{S}{\vt,\ze}+\lnm{p}{\vt,\ze}\bigg).\no
\end{eqnarray}
Then our result naturally follows.
\end{proof}
\begin{lemma}\label{Milne lemma 3}
Assume (\ref{Milne bounded}) and (\ref{Milne decay}) hold. There
exists a unique solution $\g(\et,\vvv)$ to the $\e$-Milne problem
(\ref{Milne}) satisfying for $\vt>2$ and $0\leq\zeta\leq1/4$,
\begin{eqnarray}
\lnnm{\g-\g_{\infty}}{\vt,\ze}\leq C.
\end{eqnarray}
\end{lemma}
\begin{proof}
Based on Lemma \ref{Milne lemma 1} and Lemma \ref{Milne lemma 2}, this is obvious.
\end{proof}
\begin{theorem}\label{Milne theorem 2}
Assume (\ref{Milne bounded}) and (\ref{Milne decay}) hold. There
exists a unique solution $\gg(\et,\vvv)$ to the $\e$-Milne
problem (\ref{Milne transform}) satisfying for $\vt>2$ and
$0\leq\zeta\leq1/4$,
\begin{eqnarray}
\lnnm{\gg}{\vt,\ze}\leq C.
\end{eqnarray}
\end{theorem}
\begin{proof}
Based on Theorem \ref{Milne theorem 1} and Lemma \ref{Milne lemma
3}, this is obvious.
\end{proof}

\subsection{Exponential Decay}

\begin{theorem}\label{Milne theorem 3}
Assume (\ref{Milne bounded}) and (\ref{Milne decay}) hold. For sufficiently small $K_0$, there
exists a unique solution $\gg(\et,\vvv)$ to the $\e$-Milne
problem (\ref{Milne transform}) satisfying for $\vt>2$ and
$0\leq\zeta\leq1/4$,
\begin{eqnarray}
\lnnm{\ue^{K_0\et}\gg}{\vt,\ze}\leq C.
\end{eqnarray}
\end{theorem}
\begin{proof}
Define $U=\ue^{K_0\et}\gg$. Then $U$ satisfies the equation
\begin{eqnarray}\label{exponential}
\left\{
\begin{array}{rcl}\displaystyle
\ve\frac{\p U}{\p\et}+G(\et)\bigg(\vp^2\dfrac{\p
U}{\p\ve}-\ve\vp\dfrac{\p U}{\p\vp}\bigg)+\ll[U]
&=&\ue^{K_0\et}S(\et,\vvv)+K_0\ve U,\\\rule{0ex}{1.0em}
U(0,\vvv)&=&\ue^{K_0\et}(h-\tilde h)(\vvv)\ \ \text{for}\ \
\ve>0,\\\rule{0ex}{1.0em} \displaystyle\int_{\r^2}\ve\sqrt{\m}
U(0,\vvv)\ud{\vvv}
&=&m_f[\g]-\displaystyle\int_{\r^2}\sqrt{\m}\tilde h(\vvv)\ud{\vvv},\\\rule{0ex}{1.0em}
\lim_{\et\rt\infty}U(\et,\vvv)&=&0,
\end{array}
\right.
\end{eqnarray}
We divide the proof into several steps:\\
\ \\
Step 1: $L^2$ Estimates for $S=0$ and $m_f[U]=0$.\\
In the proof of Lemma \ref{Milne prelim lemma 2},
we already show
\begin{eqnarray}
\int_0^{\infty}\ue^{2K_0\et}\br{\wi_{\gg},\wi_{\gg}}(\et)\ud{\et}\leq C.
\end{eqnarray}
We can decompose $\gg=\wi_{\gg}+\qi_{\gg}$. Since $\lim_{\et\rt\infty}\gg(\et,\vvv)=0$, we naturally have $\qi_{\gg}=0$.
Then using the orthogonal relation and zero mass-flux, we have
\begin{eqnarray}
&&\int_0^{\infty}\ue^{2K_0\et}\int_{\r^2}\gg^2(\et,\vvv)\ud{\vvv}\ud{\et}=
\int_0^{\infty}\ue^{2K_0\et}\br{\qi_{\gg},\qi_{\gg}}(\et)\ud{\et}+\int_0^{\infty}\ue^{2K_0\et}\br{\wi_{\gg},\wi_{\gg}}(\et)\ud{\et}.
\end{eqnarray}
Similar to Step 4 in the proof of Lemma \ref{Milne prelim lemma 2}, using the exponential decay of $\wi_{\gg}$, we have
\begin{eqnarray}
\int_0^{\infty}\ue^{2K_0\et}\br{\qi_{\gg},\qi_{\gg}}(\et)\ud{\et}\leq C\int_0^{\infty}\ue^{2K_0\et}\br{\wi_{\gg},\wi_{\gg}}(\et)\ud{\et}
\end{eqnarray}
This shows
\begin{eqnarray}
\tnnm{U}<C.
\end{eqnarray}
\ \\
Step 2: $L^2$ Estimates for general source term and mass flux.\\
We follow the idea in the proof of Lemma \ref{Milne prelim lemma 3}. Note
that all the auxiliary functions we construct decays exponentially.
Hence, the result naturally follows.\\
\ \\
Step 3: $L^{\infty}$ Estimates.\\
By a similar argument of Lemma \ref{Milne lemma 2},
Since $u=\a[p]+\t[K[u]+\tilde S]$, similar to the proof of Lemma \ref{Milne lemma 2}, we have
\begin{eqnarray}
\lnnm{\ue^{K_0\et}\gg}{\vt,\ze}\leq C\bigg(\tnnm{\ue^{K_0\et}\gg}+\lnnm{\ue^{K_0\et}S}{\vt,\ze}+\lnm{p}{\vt,\ze}\bigg)
\end{eqnarray}
Then we naturally obtain the result.
\end{proof}
Our results can also be applied to the Milne problem without geometric correction.
\begin{remark}\label{Milne remark}
Taking $G=0$, we consider the Milne problem
\begin{eqnarray}\label{classical Milne}
\left\{
\begin{array}{rcl}\displaystyle
\ve\frac{\p \g}{\p\et}+\ll[\g]
&=&S(\et,\ph,\vvv),\\
\g(0,\ph,\vvv)&=&h(\ph,\vvv)\ \ \text{for}\ \
\ve>0,\\\rule{0ex}{1.0em} \displaystyle\int_{\r^2}\ve\sqrt{\m}
\g(0,\ph,\vvv)\ud{\vvv}
&=&m_f[\g],\\\rule{0ex}{1.0em}
\displaystyle\lim_{\et\rt\infty}\g(\et,\ph,\vvv)&=&\g_{\infty}(\ph,\vvv).
\end{array}
\right.
\end{eqnarray}
Then there exists
\begin{eqnarray}
\tilde h(\ph,\vvv)=\tilde D_0(\ph)\psi_0+\tilde D_1(\ph)\psi_1+\tilde
D_2(\ph)\psi_2+\tilde D_3(\ph)\psi_3,
\end{eqnarray}
such that the Milne problem for
$\gg(\et,\ph,\vvv)$ in the domain
$(\et,\ph,\vvv)\in[0,\infty)\times[-\pi,\pi)\times\r^2$
\begin{eqnarray}\label{classical Milne transform}
\left\{
\begin{array}{rcl}\displaystyle
\ve\frac{\p \gg}{\p\et}+\ll[\gg]
&=&S(\et,\ph,\vvv),\\
\gg(0,\ph,\vvv)&=&h(\ph,\vvv)-\tilde h(\ph,\vvv)\ \
\text{for}\ \ \ve>0,\\\rule{0ex}{1.0em}
\displaystyle\int_{\r^2}\ve\sqrt{\m}
\gg(0,\ph,\vvv)\ud{\vvv}
&=&m_f[\g]-\displaystyle\int_{\r^2}\ve\sqrt{\m}\tilde h(\ph,\vvv)\ud{\vvv},\\\rule{0ex}{1.0em}
\displaystyle\lim_{\et\rt\infty}\gg(\et,\ph,\vvv)&=&0.
\end{array}
\right.
\end{eqnarray}
is well-posed in $L^{\infty}$ and decays exponentially.
\end{remark}

\section{Diffusive Limit and Well-Posedness}

We prove the diffusive limit and well-posedness of the Boltzmann equation (\ref{small system}).
\begin{theorem}
For given $\bb^{\e}>0$ satisfying (\ref{smallness assumption}) and $0<\e<<1$, there exists a unique positive
solution $F^{\e}=\m+\sqrt{\m}f^{\e}$ to the stationary Boltzmann equation (\ref{large system}), where
\begin{eqnarray}
f^{\e}=\e^3 R_N+\bigg(\sum_{k=1}^{N}\e^{k}\f_k\bigg)+\bigg(\sum_{k=1}^{N}\e^{k}\fb_k\bigg),
\end{eqnarray}
for $N\geq3$ and $R_N$ satisfies
\begin{eqnarray}\label{remainder}
\left\{
\begin{array}{rcl}
\e\vv\cdot\nx
R_N+\ll[R_N]&=&\e^3\Gamma[R_N,R_N]+2\Gamma[R_N,\q_N+\qb_N]+S_N\ \
\text{in}\ \ \Omega,\\\rule{0ex}{0em}
R_N(\vx_0,\vv)&=&h_N\ \ \text{for}\ \ \vv\cdot\vn<0\ \ \text{and}\ \
\vx_0\in\p\Omega,
\end{array}
\right.
\end{eqnarray}
with
\begin{eqnarray}
S_N&=&-\sum_{1\leq i,j\leq
N}^{i+j\geq N+1}\e^{i+j-3}\Gamma[\f_i,\f_j]-\Upsilon_0\sum_{1\leq i,j\leq N}^{i+j\geq N+1}\e^{i+j-3}\bigg(\Gamma[\fb_i,\fb_j]+2\Gamma[\f_i,\fb_j]\bigg)\no\\
&&-\e^{N-2}\vv\cdot\nx \f_N-\sum_{k=1}^N\e^{k-3}\ve
\frac{\p\Upsilon_0}{\p\et}\gg^{\e}_k+\e^{N-2}\frac{\Upsilon_0}{1-\e\et}\vp\frac{\p
\gg^{\e}_N}{\p\ph},\no
\end{eqnarray}
and
\begin{eqnarray}
h_N&=&\sum_{k=N+1}^{\infty}\e^{k-3}\b_k\no,
\end{eqnarray}
$\f_k$ and $\fb_k$ satisfy (\ref{at 14}) and (\ref{at 11}). Also, there exists a $C>0$ such that $f^{\e}$ satisfies
\begin{eqnarray}
\im{\bv f^{\e}}\leq C\e,
\end{eqnarray}
for any $\vartheta>2$ and $0\leq\zeta\leq1/4$
\end{theorem}.
\begin{proof}
We divide the proof into several steps:\\
\ \\
Step 1: Remainder definitions.\\
We combine the interior solution and boundary layer as follows:
\begin{eqnarray}
f^{\e}&\sim&\sum_{k=1}^{\infty}\e^{k}\f_k+\sum_{k=1}^{\infty}\e^{k}\fb_k.
\end{eqnarray}
Define the remainder as
\begin{eqnarray}\label{pf 1}
R_N&=&\frac{1}{\e^3}\bigg(f^{\e}-\sum_{k=1}^{N}\e^{k}\f_k-\sum_{k=1}^{N}\e^{k}\fb_k\bigg)=\frac{1}{\e^3}\bigg(f^{\e}-\q_N-\qb_N\bigg),
\end{eqnarray}
where
\begin{eqnarray}
\q_N&=&\sum_{k=1}^{N}\e^{k}\f_k,\\
\qb_N&=&\sum_{k=1}^{N}\e^{k}\fb_k.
\end{eqnarray}
Noting the equation (\ref{boundary layer system}) is equivalent to
the equation (\ref{small system}), we write $\lll$ to denote the
linearized Boltzmann operator as follows:
\begin{eqnarray}
\lll[f]&=&\e\vv\cdot\nx u+\ll[f]\\
&=&\ve\dfrac{\p f}{\p\et}-\dfrac{\e}{1-\e\et}\bigg(\vp\dfrac{\p
f}{\p\ph}+\vp^2\dfrac{\p f}{\p\ve}-\ve\vp\dfrac{\p
f}{\p\vp}\bigg)+\ll[f].\nonumber
\end{eqnarray}
\ \\
Step 2: Estimates of $\lll[R_N]$.\\
The interior contribution can be estimated as
\begin{eqnarray}
\lll[\q_N]&=&\e\vv\cdot\nx \q_N+\ll[\q_N]=\sum_{k=1}^N\e^{k}\bigg(\e\vv\cdot\nx \f_k+\ll[\f_k]\bigg)\\
&=&\ll[\f_1]+\sum_{k=2}^N\e^{k}\bigg(\vv\cdot\nx \f_{k-1}+\ll[\f_k]\bigg)+\e^{N+1}\vv\cdot\nx
\f_N\no\\
&=&\e^{N+1}\vv\cdot\nx
\f_N+\sum_{1\leq i,j\leq
N}^{i+j\leq N}\e^{i+j}\Gamma[\f_i,\f_j].\no
\end{eqnarray}
The boundary layer is
$\fb_k=\gg^{\e}_k\cdot\Upsilon_0$ where
$\gg_k^{\e}$ solves the $\e$-Milne problem. Notice $\Upsilon_0\Upsilon=\Upsilon_0$, so the
boundary layer contribution can be estimated as
\begin{eqnarray}\label{remainder temp 1}
\lll[\qb_N]&=&\ve\dfrac{\p
\qb_N}{\p\et}-\dfrac{\e}{1-\e\et}\bigg(\vp\dfrac{\p
\qb_N}{\p\ph}+\vp^2\dfrac{\p \qb_N}{\p\ve}-\ve\vp\dfrac{\p
\qb_N}{\p\vp}\bigg)+\ll[\qb_N]\\
&=&\sum_{k=1}^N\e^{k}\Bigg(\ve\dfrac{\p
\fb_k}{\p\et}-\dfrac{\e}{1-\e\et}\bigg(\vp\dfrac{\p
\fb_k}{\p\ph}+\vp^2\dfrac{\p \fb_k}{\p\ve}-\ve\vp\dfrac{\p
\fb_k}{\p\vp}\bigg)+\ll[\fb_k]\Bigg)\no\\
&=&\sum_{k=1}^N\e^{k}\Bigg(\ve\dfrac{\p
\gg^{\e}_k}{\p\et}\Upsilon_0-\dfrac{\e\Upsilon_0}{1-\e\et}\bigg(\vp\dfrac{\p
\gg^{\e}_k}{\p\ph}+\vp^2\dfrac{\p \gg^{\e}_k}{\p\ve}-\ve\vp\dfrac{\p
\gg^{\e}_k}{\p\vp}\bigg)+\Upsilon_0\ll[\gg^{\e}_k]\Bigg)\no\\
&&+\sum_{k=1}^N\e^{k}\ve\dfrac{\p
\Upsilon_0}{\p\et}\gg^{\e}_k\no\\
&=&\sum_{k=1}^N\e^{k}\Bigg(\ve\dfrac{\p
\gg^{\e}_k}{\p\et}\Upsilon_0-\dfrac{\e\Upsilon_0\Upsilon}{1-\e\et}\bigg(\vp\dfrac{\p
\gg^{\e}_k}{\p\ph}+\vp^2\dfrac{\p \gg^{\e}_k}{\p\ve}-\ve\vp\dfrac{\p
\gg^{\e}_k}{\p\vp}\bigg)+\Upsilon_0\ll[\gg^{\e}_k]\Bigg)\no\\
&&+\sum_{k=1}^N\e^{k}\ve\dfrac{\p
\Upsilon_0}{\p\et}\gg^{\e}_k\no\\
&=&\Upsilon_0\sum_{k=1}^N\e^{k}\Bigg(\ve\dfrac{\p
\gg^{\e}_k}{\p\et}-\dfrac{\e\Upsilon}{1-\e\et}\bigg(\vp^2\dfrac{\p \gg^{\e}_k}{\p\ve}-\ve\vp\dfrac{\p
\gg^{\e}_k}{\p\vp}\bigg)+\ll[\gg^{\e}_k]\Bigg)\no\\
&&+\sum_{k=1}^N\e^{k}\ve\dfrac{\p
\Upsilon_0}{\p\et}\gg^{\e}_k-\sum_{k=1}^N\e^{k+1}\dfrac{\Upsilon_0}{1-\e\et}\vp\dfrac{\p
\gg^{\e}_k}{\p\ph}\no\\
&=&\sum_{k=1}^N\e^{k}\ve\dfrac{\p
\Upsilon_0}{\p\et}\gg^{\e}_k-\e^{N+1}\frac{\Upsilon_0}{1-\e\et}\vp\frac{\p
\gg^{\e}_N}{\p\ph}+\Upsilon_0\sum_{1\leq i,j\leq N}^{i+j\leq
N}\e^{i+j}\bigg(\Gamma[\fb_i,\fb_j]+2\Gamma[\f_i,\fb_j]\bigg).\no
\end{eqnarray}
Note that for any $f,g\in L^2$,
\begin{eqnarray}\label{ft 01}
\pk[\Gamma(f,g)]=0.
\end{eqnarray}
Since
\begin{eqnarray}
\lll[f^{\e}]=\Gamma[f^{\e},f^{\e}],
\end{eqnarray}
then we can naturally obtain
\begin{eqnarray}
\lll[R_N]&=&\frac{1}{\e^3}\lll[f^{\e}-\q_N-\qb_N]=\frac{1}{\e^3}\lll[f^{\e}]-\frac{1}{\e^3}\lll[\q_N]-\frac{1}{\e^3}\lll[\qb_N]\\
&=&\frac{1}{\e^3}\Gamma[\q_N+\qb_N+\e^3R_N,\q_N+\qb_N+\e^3R_N]\no\\
&&-\sum_{1\leq i,j\leq
N}^{i+j\leq N}\e^{i+j-3}\Gamma[\f_i,\f_j]-\psi_0\sum_{1\leq i,j\leq N}^{i+j\leq
N}\e^{i+j-3}\bigg(\Gamma[\fb_i,\fb_j]+2\Gamma[\f_i,\fb_j]\bigg)\no\\
&&-\e^{N-2}\vv\cdot\nx \f_N-\sum_{k=1}^N\e^{k-3}\ve
\frac{\p\Upsilon_0}{\p\et}\gg^{\e}_k+\e^{N-2}\frac{\Upsilon_0}{1-\e\et}\vp\frac{\p
\gg^{\e}_N}{\p\ph}\no\\
&=&\e^3\Gamma[R_N,R_N]+2\Gamma[R_N,\q_N+\qb_N]\no\\
&&-\sum_{1\leq i,j\leq
N}^{i+j\geq N+1}\e^{i+j-3}\Gamma[\f_i,\f_j]-\psi_0\sum_{1\leq i,j\leq N}^{i+j\geq N+1}\e^{i+j-3}\bigg(\Gamma[\fb_i,\fb_j]+2\Gamma[\f_i,\fb_j]\bigg)\no\\
&&-\e^{N-2}\vv\cdot\nx \f_N-\sum_{k=1}^N\e^{k-3}\ve
\frac{\p\Upsilon_0}{\p\et}\gg^{\e}_k+\e^{N-2}\frac{\Upsilon_0}{1-\e\et}\vp\frac{\p
\gg^{\e}_N}{\p\ph}.\no
\end{eqnarray}
\ \\
Step 3: Estimates of $R_N$.\\
$R_N$ satisfies the equation
\begin{eqnarray}\label{remainder.}
\left\{
\begin{array}{rcl}
\e\vv\cdot\nx
R_N+\ll[R_N]&=&\e^3\Gamma[R_N,R_N]+2\Gamma[R_N,\q_N+\qb_N]+S_N\ \
\text{in}\ \ \Omega,\\\rule{0ex}{2em}
R_N(\vx_0,\vv)&=&h_N\ \ \text{for}\ \ \vv\cdot\vn<0\ \ \text{and}\ \
\vx_0\in\p\Omega,
\end{array}
\right.
\end{eqnarray}
where
\begin{eqnarray}
S_N&=&-\sum_{1\leq i,j\leq
N}^{i+j\geq N+1}\e^{i+j-3}\Gamma[\f_i,\f_j]-\psi_0\sum_{1\leq i,j\leq N}^{i+j\geq N+1}\e^{i+j-3}\bigg(\Gamma[\fb_i,\fb_j]+2\Gamma[\f_i,\fb_j]\bigg)\no\\
&&-\e^{N-2}\vv\cdot\nx \f_N-\sum_{k=1}^N\e^{k-3}\ve
\frac{\p\Upsilon_0}{\p\et}\gg^{\e}_k+\e^{N-2}\frac{\Upsilon_0}{1-\e\et}\vp\frac{\p
\gg^{\e}_N}{\p\ph},\no
\end{eqnarray}
and
\begin{eqnarray}
h_N&=&\sum_{k=N+1}^{\infty}\e^{k-3}\b_k\no.
\end{eqnarray}
By the classical estimate of two-dimensional Stokes-Fourier equations, exponential decay of $\fb_k$ and (\ref{ft 01}), we can directly verify
\begin{eqnarray}
\im{\bv S_N}&\leq& C\e^{N-2},\\
\iss{\bv h_N}{-}&\leq& C\e^{N-2}.
\end{eqnarray}
Based on Theorem \ref{wellposedness LT estimate}, we have
\begin{eqnarray}\label{remainder temp 2}
\\
\tm{R_N}+\tss{R_N}{+}
&\leq&C\bigg(\frac{1}{\e^2}\tm{S_N}+\frac{1}{\e}\tm{\e^3\Gamma[R_N,R_N]}+\frac{1}{\e}\tm{2\Gamma[R_N,\q_N+\qb_N]}+\frac{1}{\e^{1/2}}\tss{h_N}{-}\bigg).\no
\end{eqnarray}
Then since
\begin{eqnarray}
\is{\bv(\bb^{\e}-\m)}\leq C_0\e,
\end{eqnarray}
for $C_0$ is sufficiently small, we deduce $\b_1$ is sufficiently small. Hence, small boundary data naturally yields
\begin{eqnarray}
\im{\bv(\q_N+\qb_N)}\leq \delta\e,
\end{eqnarray}
which further implies
\begin{eqnarray}
\frac{1}{\e}\tm{2\Gamma[R_N,\q_N+\qb_N]}&\leq&\frac{C}{\e}\tm{R_N}\im{\bv
(\q_N+\qb_N)}\leq\d\tm{R_N},
\end{eqnarray}
for some small $\d>0$. Hence, absorbing them into the left-hand side of (\ref{remainder temp 2})
yields
\begin{eqnarray}
&&\tm{R_N}+\tss{R_N}{+}\leq C\bigg(\frac{1}{\e^2}\tm{S_N}+\e\tm{\Gamma[R_N,R_N]}+\frac{1}{\e^{1/2}}\tss{h_N}{-}\bigg).
\end{eqnarray}
Based on Theorem \ref{wellposedness LI estimate}, we have
\begin{eqnarray}
&&\im{\bv R_N}+\iss{\bv R_N}{+}\\
&\leq&C\bigg(\frac{1}{\e}\tm{R_N}+\im{\bv S_N}+\iss{\bv h_N}{-}\no\\
&&+\im{\bv \e^2\Gamma[R_N,R_N]}+\im{2\bv \Gamma[R_N,\q_N+\qb_N]}\bigg)\no\\
&\leq&C\bigg(\e\tm{\Gamma[R_N,R_N]}+\frac{1}{\e^3}\tm{S_N}+\frac{1}{\e^{3/2}}\tss{h_N}{-}+\im{\bv S_N}+\iss{\bv h_N}{-}\no\\
&&+\e^3\im{\bv \Gamma[R_N,R_N]}+\im{2\bv \Gamma[R_N,\q_N+\qb_N]}\bigg)\no\\
&\leq&C\bigg(\e\tm{\Gamma[R_N,R_N]}+\e^3\im{\bv \Gamma[R_N,R_N]}+\e\im{\bv
R_N}+\e^{N-5}\bigg).\no
\end{eqnarray}
Moreover, we can directly estimate
\begin{eqnarray}
\tm{\e\Gamma[R_N,R_N]}&\leq& C\e\im{\bv R_N}^2\\
\e^3\im{\bv \Gamma[R_N,R_N]}&\leq&C\e^3\im{\bv R_N}^2.
\end{eqnarray}
Then if $N\geq3$, we obtain
\begin{eqnarray}
\im{\bv R_N}+\iss{\bv R_N}{+}&\leq&C\e^{N-5}+C\e\im{\bv R_N}^2,
\end{eqnarray}
which further implies
\begin{eqnarray}
\im{\bv R_N}+\iss{\bv R_N}{+}&\leq&\frac{C}{\e^2}.
\end{eqnarray}
for $\e$ sufficiently small. This means we have shown
\begin{eqnarray}
\frac{1}{\e^3}\im{\bv\bigg(f^{\e}-\sum_{k=1}^{N}\e^k\f_k-\sum_{k=1}^{N}\e^k\fb_k\bigg)}=O\bigg(\frac{1}{\e^2}\bigg),
\end{eqnarray}
which naturally leads to the desired result.
\end{proof}
Hence, combining the estimate of steady Navier-Stokes-type equations and $\e$-Milne problem, we have
\begin{eqnarray}
f^{\e}=\e^3 R_N+\q_N+\qb_N,
\end{eqnarray}
exists and is well-posed. The uniqueness and positivity follows from a standard argument as in \cite{Esposito.Guo.Kim.Marra2013}.

\section{Counterexample for Classical Approach}

In this section, we present the classical approach with the idea in \cite{Sone2002, Sone2007} to construct asymptotic expansion, especially the boundary layer expansion, and provide counterexamples to show this method is problematic.

\subsection{Interior Expansion}

Basically, the expansion for interior solution is identical to our method, so we omit the details and only present the notation.
We define the interior expansion
\begin{eqnarray}\label{interior expansion.}
\fc\sim\sum_{k=1}^{\infty}\e^{k}\fc_k(\vx,\vv),
\end{eqnarray}
with
\begin{eqnarray}
\fc_k(\vx,\vv)=A_k(\vx,\vv)+B_k(\vx,\vv)+C_k(\vx,\vv),
\end{eqnarray}
where
\begin{eqnarray}
A_k(\vx,\vv)=\sqrt{\m}\bigg(A_{k,0}(\vx)+A_{k,1}(\vx)\v_1+A_{k,2}(\vx)\v_2+A_{k,3}(\vx)\left(\frac{\abs{\vv}^2-2}{2}\right)\bigg),
\end{eqnarray}
\begin{eqnarray}
B_k(\vx,\vv)=\sqrt{\m}\bigg(B_{k,0}(\vx)+B_{k,1}(\vx)\v_1+B_{k,2}(\vx)\v_2+B_{k,3}(\vx)\left(\frac{\abs{\vv}^2-2}{2}\right)\bigg),
\end{eqnarray}
with $B_{k}$ depending on $A_{s,i}$ in $1\leq s\leq k-1$ and $i=0,1,2,3$ as
\begin{eqnarray}
B_{k,0}&=&0,\\
B_{k,1}&=&\sum_{i=1}^{k-1}A_{i,0}A_{k-i,1},\\
B_{k,2}&=&\sum_{i=1}^{k-1}A_{i,0}A_{k-i,2},\\
B_{k,3}&=&\sum_{i=1}^{k-1}\bigg(A_{i,0}A_{k-i,3}+A_{i,1}A_{k-i,1}+A_{i,2}A_{k-i,2}\\
&&+\sum_{j=1}^{k-1-i}A_{i,0}(A_{j,1}A_{k-i-j,1}+A_{j,2}A_{k-i-j,2})\bigg),\no
\end{eqnarray}
and $C_k(\vx,\vv)$ satisfies
\begin{eqnarray}
\int_{\r^2}\sqrt{\m(\vv)}C_k(\vx,\vv)\left(\begin{array}{c}1\\\vv\\\abs{\vv}^2
\end{array}\right)\ud{\vv}=0,
\end{eqnarray}
with
\begin{eqnarray}
\ll[C_k]&=&-\vv\cdot\nx\f_{k-1}+\sum_{i=1}^{k-1}\Gamma[\f_i,\f_{k-i}],
\end{eqnarray}
which can be solved explicitly at any fixed $\vx$. We define
\begin{eqnarray}
A_k=\sqrt{\m}\bigg(\rh_k+\u_{k,1}\v_1+\u_{k,2}\v_2+\th_k\left(\abs{\vv}^2-1\right)\bigg),
\end{eqnarray}
Then $A_k$ satisfies the equations as follows:\\
\ \\
$0^{th}$ order equations:
\begin{eqnarray}
P_1-(\rh_1+\th_1)&=&0,\\
\nx P_1&=&0,
\end{eqnarray}
$1^{st}$ order equations:
\begin{eqnarray}
P_2-(\rh_2+\th_2+\rh_1\th_1)&=&0,\\
\uh\cdot\nx\uh_1-\gamma_1\dx\uh_1+\nx P_2&=&0,\\
\nx\cdot\uh_1&=&0,\\
\uh_1\cdot\nx\th_1-\gamma_2\dx\th_1&=&0,
\end{eqnarray}
$k^{th}$ order equations:
\begin{eqnarray}
P_{k+1}-\left(\rh_{k+1}+\th_{k+1}+\sum_{i=1}^{k+1-i}\rh_i\th_{k+1-i}\right)&=&0,\\
\sum_{i=1}^{k}\uh_i\cdot\nx\uh_{k+1-i}-\gamma_1\dx\uh_k+\nx P_{k+1}&=&G_{k,1},\\
\nx\cdot\uh_k&=&G_{k,2},\\
\sum_{i=1}^{k}\uh_i\cdot\nx\th_{k+1-i}-\gamma_2\dx\th_k&=&G_{k,3},
\end{eqnarray}
where
\begin{eqnarray}
\\
G_{k,j}=G_{k,j}[\vx,\vv; \rh_1,\ldots,\rh_{k-1};
\th_1,\ldots,\th_{k-1}; \uh_1,\ldots,\uh_{k-1}]\no,
\end{eqnarray}
is explicit functions depending on lower order terms, and $\gamma_1$ and $\gamma_2$ are two positive constants.

\subsection{Boundary Layer Expansion}

By the idea in \cite{Sone2002, Sone2007}, the boundary layer expansion can be defined by introducing substitutions (\ref{substitution 1}), (\ref{substitution 2}) and (\ref{substitution 3}). Note that we terminate here and do not further use substitution (\ref{substitution 4}). Hence,  we have transformed
the equation (\ref{small system}) into
\begin{eqnarray}\label{boundary layer system.}
\left\{
\begin{array}{rcl}
(\vr\cdot\vn)\dfrac{\p f^{\e}}{\p\et}-\dfrac{\e}{1-\e\et}(\vr\cdot\ta)\dfrac{\p
f^{\e}}{\p\ph}+\ll[f^{\e}]=\Gamma[f^{\e},f^{\e}],\\\rule{0ex}{1.5em}
f^{\e}(0,\ph,\vr)=\b^{\e}(0,\ph,\vr) \ \ \text{for}\ \
\vr\cdot\vn>0.
\end{array}
\right.
\end{eqnarray}
We define the boundary layer expansion
\begin{eqnarray}\label{boundary layer expansion.}
\fbc\sim\sum_{k=1}^{\infty}\e^{k}\fbc_k(\et,\ph,\vr),
\end{eqnarray}
where $\fbc_k$ can be determined by plugging it into the equation (\ref{boundary layer system.}) and comparing the order of $\e$. Thus in a neighborhood of the boundary, we have
\begin{eqnarray}
\label{boundary expansion 1.} (\vr\cdot\vn)\frac{\p
\fbc_1}{\p\et}+\ll[\fbc_1]
&=&0,\\
\no\\
\label{boundary expansion 2.}(\vr\cdot\vn)\frac{\p
\fbc_2}{\p\et}+\ll[\fbc_2]
&=&\frac{1}{1-\e\eta}(\vr\cdot\ta)\dfrac{\p \fbc_1}{\p\ph}+\Gamma[\fbc_1,\fbc_1]+2\Gamma[\fc_1,\fbc_1],
\\
\ldots\no\\
\no\\
\label{boundary expansion 3.}(\vr\cdot\vn)\frac{\p
\fbc_k}{\p\et}+\ll[\fbc_k]
&=&\frac{1}{1-\e\eta}(\vr\cdot\ta)\dfrac{\p \fbc_{k-1}}{\p\ph}+\sum_{i=1}^{k-1}\Gamma[\fbc_i,\fbc_{k-i}]+2\sum_{i=1}^{k-1}\Gamma[\fc_i,\fbc_{k-i}].
\end{eqnarray}
The bridge between the interior solution and boundary layer
is the boundary condition
\begin{eqnarray}
f^{\e}(\vx_0,\vv)&=&\b^{\e}(\vx_0,\vv).
\end{eqnarray}
Plugging the combined expansion
\begin{eqnarray}
f^{\e}\sim\sum_{k=1}^{\infty}\e^k(\fc_k+\fbc_k),
\end{eqnarray}
into the boundary condition and comparing the order of $\e$, we obtain
\begin{eqnarray}
\fc_1+\fbc_1&=&\b_1,\\
\fc_2+\fbc_2&=&\b_2,\\
\ldots\no\\
\fc_k+\fbc_k&=&\b_k.
\end{eqnarray}
This is the boundary conditions $\fc_k$ and $\fbc_k$ need to satisfy.

\subsection{Classical Approach to Construct Asymptotic Expansion}

We divide the construction of asymptotic expansion into several steps for each $k\geq1$:\\
\ \\
Step 1: $\e$-Milne Problem.\\
We solve the $\e$-Milne problem
\begin{eqnarray}
\left\{
\begin{array}{rcl}\displaystyle
(\vr\cdot\vn)\frac{\p \g_k}{\p\et}+\ll[\g_k]
&=&S_k(\et,\ph,\vr),\\
\g_k(0,\ph,\vr)&=&h_k(\ph,\vr)\ \ \text{for}\ \
\vr\cdot\vn>0,\\\rule{0ex}{1.0em} \displaystyle\int_{\r^2}
(\vr\cdot\vn)\sqrt{\m}\g_k(0,\ph,\vr)\ud{\vr}
&=&m_f[\g_k](\ph),\\\rule{0ex}{1.0em}
\displaystyle\lim_{\et\rt\infty}\g_k(\et,\ph,\vr)&=&\g_{k}(\infty,\ph,\vr),
\end{array}
\right.
\end{eqnarray}
for $\g_k(\et,\ph,\vr)$ with the in-flow boundary data
\begin{eqnarray}
h_k&=&\b_k-(B_k+C_k)
\end{eqnarray}
and source term
\begin{eqnarray}
S_k&=&\frac{\Upsilon(\sqrt{\e}\et)}{1-\e\et}(\vr\cdot\vec\tau)\dfrac{\p
\fbc_{k-1}}{\p\ph}
+\sum_{i=1}^{k-1}\Gamma[\fbc_i,\fbc_{k-i}]+2\sum_{i=1}^{k-1}\Gamma[\fc_i,\fbc_{k-i}],
\end{eqnarray}
where
\begin{eqnarray}\label{cutoff 1.}
\Upsilon(z)=\left\{
\begin{array}{ll}
1&0\leq z\leq1/2,\\
0&3/4\leq z\leq\infty.
\end{array}
\right.
\end{eqnarray}
Here the mass-flux $m_f[\g_k](\ph)$ will be determined later. Based on Remark \ref{Milne remark}, there exist
\begin{eqnarray}
\tilde h_k=\sqrt{\m}\bigg(\tilde
D_{k,0}+\tilde D_{k,1}\v_{r,1}+\tilde D_{k,2}\v_{r,2}+\tilde
D_{k,3}^{\e}\bigg(\frac{\abs{\vr}^2-2}{2}\bigg)\bigg),
\end{eqnarray}
such that the problem
\begin{eqnarray}\label{at 11.}
\left\{
\begin{array}{rcl}\displaystyle
(\vr\cdot\vn)\frac{\p \gg_k}{\p\et}+\ll[\gg_k]
&=&S_k(\et,\ph,\vr),\\
\gg_k(0,\ph,\vr)&=&h_k(\ph,\vr)-\tilde h_k(\ph,\vr)\ \
\text{for}\ \ \vr\cdot\vn>0,\\\rule{0ex}{1.0em}
\displaystyle\int_{\r^2}(\vr\cdot\vn)\sqrt{\m}
\gg_k(0,\ph,\vr)\ud{\vr}
&=&m_f[\gg_k](\phi),\\\rule{0ex}{1.0em}
\displaystyle\lim_{\et\rt\infty}\gg_k(\et,\ph,\vr)&=&0,
\end{array}
\right.
\end{eqnarray}
is well-posed.
\ \\
Step 2: Definition of Interior Solution and Boundary Layer with Geometric Correction.\\
Define
\begin{eqnarray}
\fbc_k=\gg_k\cdot\Upsilon_0(\e^{1/2}\et)
\end{eqnarray}
where $\gg_k$ the solution of $\e$-Milne problem (\ref{at 11.}) and
\begin{eqnarray}\label{cutoff 2.}
\Upsilon_0(z)=\left\{
\begin{array}{ll}
1&0\leq z\leq1/4,\\
0&1/2\leq z\leq\infty.
\end{array}
\right.
\end{eqnarray}
Naturally, we have
\begin{eqnarray}
\lim_{\et\rt0}\fbc_k(\et,\ph,\vr)=0.
\end{eqnarray}
The interior solution
\begin{eqnarray}
\fc_k(\vx,\vv)=A_k(\vx,\vv)+B_k(\vx,\vv)+C_k(\vx,\vv),
\end{eqnarray}
where $A_k$ satisfies
\begin{eqnarray}
A_k=\sqrt{\m}\left(\rh_k+\u_{k,1}\v_1+\u_{k,2}\v_2+\th_k\left(\frac{\abs{\vv}^2-1}{2}\right)\right),
\end{eqnarray}
and
\begin{eqnarray}\label{at 14.}
P_{k+1}-\left(\rh_{k+1}+\th_{k+1}+\sum_{i=1}^{k+1-i}\rh_i\th_{k+1-i}\right)&=&0,\\
\sum_{i=1}^{k}\uh_i\cdot\nx\uh_{k+1-i}-\gamma_1\dx\uh_k+\nx P_{k+1}&=&H_{k,1},\\
\nx\cdot\uh_k&=&H_{k,2},\\
\sum_{i=1}^{k}\uh_i\cdot\nx\th_{k+1-i}-\gamma_2\dx\th_k&=&H_{k,3},
\end{eqnarray}
with boundary condition
\begin{eqnarray}
A_{k,0}&=&\tilde D_{k,0},\\
A_{k,1}&=&-\tilde D_{k,0},\\
A_{k,2}&=&-\tilde D_{k,0},\\
A_{k,3}&=&\tilde D_{k,3}.
\end{eqnarray}
where $\tilde D_{k,i}$ comes from the boundary data of Milne problem $\tilde h_k$.
This determines $A_{k,0}$, $A_{k,1}$, $A_{k,2}$ and $A_{k,3}$. Now it is easy to verify the boundary data are satisfied as
\begin{eqnarray}
\fc_k+\fbc_k&=&\b_k.
\end{eqnarray}
\ \\
Step 3: Boussinesq relation and Vanishing Mass-Flux.\\
Similarly, the free mass-flux $m_f[\gg_{k}](\ph)$ can help to enforce two relations: the Boussinesq relation
\begin{eqnarray}
\rh_{k}+\th_{k}=E_k-\sum_{i=1}^{k-i}\rh_i\th_{k-i},
\end{eqnarray}
and vanishing mass-flux relation
\begin{eqnarray}
\int_{\p\Omega}\int_{\r^2}\fc_k(\vx,\vv)\ud{\vv}\ud{\gamma}=0.
\end{eqnarray}
Therefore, $m_f[\gg_{k}](\ph)$ is completed determined and so are $\fc_k$ and $\fbc_k$.

The analysis in \cite{Sone2002, Sone2007} anticipates this process can be generalized to arbitrary $k$. However, In order to show the hydrodynamic limit, we at least need to expand to $k=2$ at least. Therefore, based on Remark \ref{Milne remark}, we require $S_2\in L^{\infty}$ to obtain a well-posed $\fbc_2$, i.e. we need
\begin{eqnarray}
\dfrac{\p \fbc_1}{\p\ph}\in L^{\infty},
\end{eqnarray}
which further requires
\begin{eqnarray}
\dfrac{\p \fbc_1}{\p\et}\in L^{\infty}.
\end{eqnarray}
Theorem \ref{counter theorem 1} states that for certain boundary data $\bb^{\e}$, this is invalid. Hence, this formulation breaks down.

\subsection{Singularity in Derivative of Milne Problem}

Now we present the singularity of the normal derivative in the Milne problem. For convenience, we use the notation $\vvv=(\ve,\vp)$.
\begin{theorem}\label{counter theorem 1}
For the Milne problem
\begin{eqnarray}\label{counter equation}
\left\{
\begin{array}{rcl}\displaystyle
\ve\frac{\p \g}{\p\et}+\ll[\g]
&=&0,\\
\g(0,\vvv)&=&h(\vvv)\ \ \text{for}\ \
\ve>0,\\ \displaystyle\int_{\r^2}\ve\sqrt{\m}
\g(0,\vvv)\ud{\vvv}
&=&0,\\
\displaystyle\lim_{\et\rt\infty}\g(\et,\vvv)&=&\g_{\infty}(\vvv).
\end{array}
\right.
\end{eqnarray}
with
\begin{eqnarray}
h=\vp\ue^{-(\vp^2-1)-M\ve^2},
\end{eqnarray}
where $\ve$ and $\vp$ are defined as in (\ref{substitution 3}) and $M$ is sufficiently large such that
\begin{eqnarray}
h(0,1)&=&1,\\
\tss{h}{-}&<<&1,
\end{eqnarray}
then we have
\begin{eqnarray}
\im{\frac{\p\g}{\p\et}}\notin L^{\infty}([0,\infty)\times\r^2).
\end{eqnarray}
\end{theorem}
\begin{proof}
We divide the proof into several steps:
We first assume $\p_{\et}\g\in
L^{\infty}([0,\infty)\times\r^2)$ and then show it can lead to a contradiction.\\
\ \\
Step 1: Definition of trace.\\
It is easy to see $\p_{\et}\g$ satisfies the Milne problem
\begin{eqnarray}
\ve\frac{\p(\p_{\et}\g)}{\p\eta}+\ll[\p_{\et}\g]&=&0.
\end{eqnarray}
Since $k(\vvu,\vvv)=k_2(\vvu,\vvv)-k_1(\vvu,\vvv)$ is in $L^{1}$ with respect to $\vvu$ uniformly in $\vvv$, then we have $K[\p_{\et}\g]\in L^{\infty}([0,\infty)\times\r^2)$. For fixed $N>0$, $\nu(\vvv)$ is bounded in the domain $S=\{\abs{\vvv}\leq N\}$. Hence,
we have $\nu(\vvv)\p_{\et}\g\in L^{\infty}([0,\infty)\times S)$, which further implies $\ll[\p_{\et}\g]\in L^{\infty}([0,\infty)\times S)$. Therefore,
by a standard cut-off argument and Ukai's trace theorem, we deduce $\p_{\et}\g(0)\in L^{\infty}(S)$ is well-defined.\\
\ \\
However, we can define the trace of $\p_{\et}\g$ in another fashion. For
any $\ve\neq0$, since we have $\nu(\vvv)\g\in
L^{\infty}([0,\infty)\times S)$ as well as
$K[\g]\in L^{\infty}[0,\infty)\times S$, by the Milne
problem (\ref{counter equation}), it is naturally to define for
$\et>0$
\begin{eqnarray}
\p_{\et}\g(\et,\vvv)=\frac{K[\g](\et,\vvv)-\nu\g(\et,\vvv)}{\ve}.
\end{eqnarray}
Since $\p_{\et}\g\in
L^{\infty}([0,\infty)\times S$, we know $\g$
is continuous with respect to $\et$ for a.e. $\vv$. Taking
$\et\rt0$ defines the trace for $\p_{\et}\g$ at$(0,\vvv)$
\begin{eqnarray}
\p_{\et}\g(0,\vvv)=\frac{K[\g](0,\vvv)-\nu\g(0,\vvv)}{\ve}.
\end{eqnarray}
Since the grazing set $\{\vvv: \ve=0\}$ is zero-measured on the boundary $\et=0$, then we
have the trace of $\p_{\et}\g$ is a.e. well-defined.\\
\ \\
By the uniqueness of trace of $\p_{\et}\g$, above two types of traces
must coincide with each other a.e.. Then we may combine them both
and obtain $\p_{\et}\g(0,\vvv)\in
L^{\infty}(S)$ is a.e. well-defined and
satisfies the formula
\begin{eqnarray}
\p_{\et}\g(0,\vvv)=\frac{K[\g](0,\vvv)-\nu\g(0,\vvv)}{\ve}.
\end{eqnarray}
\ \\
Step 2: Limiting Process.\\
Therefore, we may consider the limiting process
\begin{eqnarray}
\lim_{\vvv\rt(0,1)}\frac{\p
\g}{\p\eta}(0,\vvv)=\lim_{\vvv\rt(0,1)}\frac{K[\g](0,\vvv)-\nu\g(0,\vvv)}{\ve}.
\end{eqnarray}
Based on \cite[Lemma 3.3.1]{Glassey1996}, we have
\begin{eqnarray}
\lnm{K[\g](0)}{0,0}\leq\lnnm{K[\g]}{0,0}\leq C\ltnm{\g}{0}.
\end{eqnarray}
By Lemma \ref{Milne prelim lemma 6}, we have
\begin{eqnarray}
\ltnm{\g}{0}\leq C(\d)\tnnm{\g}+\d\lnnm{\g}{\vt,0},
\end{eqnarray}
for $\d>0$ sufficiently small and $\vt>2$.
Combining this with Theorem \ref{Milne theorem 1} and Theorem \ref{Milne theorem 2}, we know
\begin{eqnarray}
\tnnm{\g}&\leq&C\tnm{h}\\
\lnnm{\g}{\vt,0}&\leq& C<\infty.
\end{eqnarray}
Taking $\d$ sufficiently small, and then taking $M$ sufficiently large, we have
\begin{eqnarray}
\ltnm{\g}{0}<<1.
\end{eqnarray}
On the other hand, we can see
\begin{eqnarray}
\nu(0,1)\g(0,0,1)\geq Ch(0,1)\geq C_0>0,
\end{eqnarray}
for some positive constant $C_0$.
Therefore, we have shown
\begin{eqnarray}
\ll[g](0,\vvv)\geq\frac{C_0}{2}>0,
\end{eqnarray}
when $\vvv\rt(0,1)$ which implies $\ve\rt0$. Hence, we can solve the normal derivative as
\begin{eqnarray}
\frac{\p \g}{\p\et}=-\frac{\ll[\g](0)}{\ve}\rt\infty,
\end{eqnarray}
which contradicts our assumption that $\p_{\et}\g(0,\vvv)\in
L^{\infty}(S)$.
\end{proof}

\subsection{Counterexample to Classical Approach}

We present a counterexample to show this classical approach can lead to wrong result.
\begin{theorem}
For given $\bb^{\e}>0$ satisfying (\ref{smallness assumption}) with
\begin{eqnarray}
\frac{\b_1}{\sqrt{\m}}=\bigg(\vp\ue^{-(\vp^2-1)-M\ve^2}\bigg)=h(\ve,\vp),
\end{eqnarray}
where $\ve$ and $\vp$ are defined as in (\ref{substitution 3}) and we take $M$ sufficiently large such that
\begin{eqnarray}
h(0,1)&=&1,\\
\tss{h}{-}&<<&1,
\end{eqnarray}
there exists $C>0$ such that
\begin{eqnarray}
\im{f^{\e}-(\fc_1+\fbc_1)}\geq C\e,
\end{eqnarray}
where the interior solution $\fc_1$ is defined in (\ref{at 14.}) and boundary layer $\fbc_1$ is defined in (\ref{at 11.}).
\end{theorem}
\begin{proof}
Define $\w=\f_1+\fb_1$ and $\wc=\fc_1+\fbc_1$. Consider the $\e$-Milne problem
\begin{eqnarray}\label{compare temp 1}
\left\{
\begin{array}{rcl}\displaystyle
\ve\frac{\p \w}{\p\et}+G(\e;\et)\bigg(\vp^2\dfrac{\p
\w}{\p\ve}-\ve\vp\dfrac{\p \w}{\p\vp}\bigg)+\ll[\w]
&=&0,\\
\w(0,\ve,\vp)&=&h(\ve,\vp)\ \ \text{for}\ \
\ve>0,\\\displaystyle\int_{\r^2}\ve\sqrt{\m}
\w(0,\ve,\vp)\ud{\ve}\ud{\vp}
&=&m_f(\ph),\\
\displaystyle\lim_{\et\rt\infty}\w(\et,\ve,\vp)&=&\w_{\infty}(\ve,\vp),
\end{array}
\right.
\end{eqnarray}
and Milne problem
\begin{eqnarray}\label{compare temp 2}
\left\{
\begin{array}{rcl}\displaystyle
\ve\frac{\p \wc}{\p\et}+\ll[\wc]
&=&0,\\
\wc(0,\ve,\vp)&=&h(\ve,\vp)\ \ \text{for}\ \
\ve>0,\\ \displaystyle\int_{\r^2}\ve\sqrt{\m}
\wc(0,\ve,\vp)\ud{\ve}\ud{\vp}
&=&m_f(\ph),\\
\displaystyle\lim_{\et\rt\infty}\wc(\et,\ve,\vp)&=&\wc_{\infty}(\ve,\vp).
\end{array}
\right.
\end{eqnarray}
For convenience, we use the same velocity variables. Note that $\w$ actually satisfies an $\e$-Milne problem with non-trivial source term. However, based on the proof of Theorem \ref{Milne theorem 1}, this source term will add a $O(\e^{1/2})$ perturbation to $\w$, so we can omit it and concentrate on above simpler form.\\
\ \\
We divide the proof into several steps:\\
\ \\
Step 1: Continuity of $K[\w]$ and $K[\wc]$ at $\et=0$.\\
For any $R_0>r_0>0$ and $\vvu=(\mathfrak{u}_{\et},\mathfrak{u}_{\ph})$, we have
\begin{eqnarray}
&&\abs{K[\wc](0,\vvv)-K[\wc](\et,\vvv)}\\
&\leq&\int_{\mathfrak{u}_{\et}\leq r_0}\abs{k(\vvu,\vvv)}\abs{\wc(0,\vvu)-\wc(\et,\vvu)}\ud{\vvu}+\int_{\mathfrak{u}_{\et}\geq R_0}\abs{k(\vvu,\vvv)}\abs{\wc(0,\vvu)-\wc(\et,\vvu)}\ud{\vvu}\no\\
&&+\int_{r_0\leq\mathfrak{u}_{\et}\leq R_0}\abs{k(\vvu,\vvv)}\abs{\wc(0,\vvu)-\wc(\et,\vvu)}\ud{\vvu}\no
\end{eqnarray}
Since we know $\wc\in L^{\infty}([0,\infty)\times\r^2)$, then for any $\delta>0$ we can take $r_0$ sufficiently small such that
\begin{eqnarray}
\int_{\mathfrak{u}_{\et}\leq r_0}\abs{k(\vvu,\vvv)}\abs{\wc(0,\vvu)-\wc(\et,\vvu)}\ud{\vvu}&\leq&C\int_{\mathfrak{u}_{\et}\leq r_0}\abs{k(\vvu,\vvv)}\ud{\vvu}\leq\frac{\delta}{3}.
\end{eqnarray}
Since we know
\begin{eqnarray}
\im{\bv (\wc-\wc_{\infty})}\leq C<\infty,
\end{eqnarray}
then there exists a $R_0>0$, such that for $\mathfrak{u}_{\et}\geq R_0$,
\begin{eqnarray}
\abs{\wc(\et,\vvu)}\leq \tilde\d,
\end{eqnarray}
where $\tilde\d$ is sufficiently small. Therefore, we have
\begin{eqnarray}
\int_{\mathfrak{u}_{\et}\geq R_0}\abs{k(\vvu,\vvv)}\abs{\wc(0,\vvu)-\wc(\et,\vvu)}\ud{\vvu}&\leq&2\tilde\d\int_{\mathfrak{u}_{\et}\geq R_0}\abs{k(\vvu,\vvv)}\ud{\vvu}\leq\frac{\delta}{3}.
\end{eqnarray}
For fixed $r_0$ and $R_0$ satisfying above requirement, we estimate the integral on $r_0\leq\mathfrak{u}_{\et}\leq R_0$. By Ukai's trace theorem, we have $\wc(0,\vvv)$ is well-defined and
\begin{eqnarray}
\p_{\et}\wc(0,\vvu)=\frac{K[\wc](0,\vvv)-\nu(\vvv)\wc(0,\vvv)}{\ve}.
\end{eqnarray}
The in $r_0\leq\mathfrak{u}_{\et}\leq R_0$, $\p_{\et}\wc$ is bounded, which implies $\wc(\et,\vvv)$ is uniformly continuous at $\et=0$. Then there exists a $\et_0$ such that for $0\leq\et\leq\et_0$,
\begin{eqnarray}
\int_{r_0\leq\mathfrak{u}_{\et}\leq R_0}\abs{k(\vvu,\vvv)}\abs{\wc(0,\vvu)-\wc(\et,\vvu)}\ud{\vvu}&\leq&C\tilde\d\int_{r_0\leq\mathfrak{u}_{\et}\leq R_0}\abs{k(\vvu,\vvv)}\ud{\vvu}\leq\frac{\delta}{3}.
\end{eqnarray}
In summary, we have shown for any $\delta>0$, there exists a $\et_0>0$ such that for any $0\leq\et\leq\et_0$ and fixed $\vvv$,
\begin{eqnarray}
\abs{K[\wc](0,\vvv)-K[\wc](\et,\vvv)}&\leq&\delta.
\end{eqnarray}
Therefore, $K[\wc]$ is continuous at $\et=0$. A similar argument can be implemented to $\w$. It is easy to see above estimate is uniform in $\vvv$ since $L^1$ estimate of $k(\vvu,\vvv)$ in $\vvu$ is uniform with respect to $\vvv$. Also, it is obvious to see $K$ is continuous with respect to $\vvv$ at $\et=0$.\\
\ \\
Step 2: Milne formulation.\\
We consider the solution at a specific point $\et=n\e$, $\ve=\e$ and $\vp=\sqrt{1-\e^2}$  for some
fixed $n>0$. The solution along the characteristics can be rewritten
as follows:
\begin{eqnarray}\label{compare temp 3}
\wc(n\e,\e,\sqrt{1-\e^2})=h(\e,\sqrt{1-\e^2})\ue^{-\frac{\nu(1)}{\e}n\e}
+\int_0^{n\e}\ue^{-\frac{\nu(1)}{\e}(n\e-\k)}\frac{1}{\e}K[\wc](\k,\e,\sqrt{1-\e^2})\ud{\k},
\end{eqnarray}
\begin{eqnarray}\label{compare temp 4}
\\
\w(n\e,\e,\sqrt{1-\e^2})=h(\e_0,\sqrt{1-\e^2_0})\ue^{-\int_0^{n\e}\frac{\nu(1)}{\ve(\zeta)}\ud{\zeta}}
+\int_0^{n\e}\ue^{-\int_{\k}^{n\e}\frac{\nu(1)}{\ve(\zeta)}\ud{\zeta}}\frac{1}{\ve(\k)}K[\w](\k,\ve(\k),\vp(\k))\ud{\k},\no
\end{eqnarray}
where $\nu(1)$ denote the value of $\nu(\vvv)$ at $\abs{\vvv}=1$ and we have the conserved energy along the characteristics
\begin{eqnarray}
E(\et,\ve,\vp)=\vp e^{-W(\et)},
\end{eqnarray}
in which $(0,\e_0,\sqrt{1-\e^2_0})$ and $(\zeta,\ve(\zeta),\sqrt{1-\ve^2(\zeta)})$ are in the same
characteristics of $(n\e,\e,\sqrt{1-\e^2})$.\\
\ \\
Step 3: Estimates of (\ref{compare temp 3}).\\
We turn to the Milne problem for $\wc$. We have the natural estimate
\begin{eqnarray}
\int_0^{n\e}\ue^{-\frac{\nu(1)}{\e}(n\e-\k)}\frac{1}{\e}\ud{\k}&=&\ue^{-n\nu(1)}\int_0^{n\e}
\ue^{\frac{\nu(1)\k}{\e}}\frac{1}{\e}\ud{\k}=\ue^{-n\nu(1)}\int_0^ne^{\nu(1)\zeta}\ud{\zeta}=\frac{1}{\nu(1)}\bigg(1-\ue^{-n\nu(1)}\bigg).
\end{eqnarray}
Then for $0<\e\leq\et_0$, we have $\abs{K[\wc](0,0,1)-K[\wc](\k,\e,\sqrt{1-\e^2})}\leq\delta+\e$, which implies
\begin{eqnarray}
\int_0^{n\e}\ue^{-\frac{\nu(1)}{\e}(n\e-\k)}\frac{1}{\e}K[\wc](\k,\e,\sqrt{1-\e^2})\ud{\k}&=&
\int_0^{n\e}\ue^{-\frac{\nu(1)}{\e}(n\e-\k)}\frac{1}{\e}K[\wc](0,0,1)\ud{\k}+O(\delta)+O(\e)\\
&=&\frac{1}{\nu(1)}(1-\ue^{-n\nu(1)})K[\wc](0,0,1)+O(\delta)+O(\e)\nonumber.
\end{eqnarray}
For the boundary data term, it is easy to see
\begin{eqnarray}
h(\e,\sqrt{1-\e^2})\ue^{-\frac{\nu(1)}{\e}n\e}&=&\ue^{-n\nu(1)}h(\e,\sqrt{1-\e^2}).
\end{eqnarray}
In summary, we have
\begin{eqnarray}
\wc(n\e,\e,\sqrt{1-\e^2})=\frac{1}{\nu(1)}(1-\ue^{-n\nu(1)})K[\wc](0,0,1)+\ue^{-n\nu(1)}h(0,1)+O(\delta)+O(\e).
\end{eqnarray}
\ \\
Step 4: Estimates of (\ref{compare temp 4}).\\
We consider the $\e$-Milne problem for $\w$. For $\e<<1$ sufficiently
small, $\psi(\e)=1$. Then we may estimate
\begin{eqnarray}
\vp(\zeta)e^{-W(\zeta)}=\sqrt{1-\e^2} e^{-W(n\e)},
\end{eqnarray}
which implies
\begin{eqnarray}
\vp(\zeta)=\frac{1-n\e^2}{1-\e\zeta}\sqrt{1-\e^2}.
\end{eqnarray}
and hence
\begin{eqnarray}
\ve(\zeta)=\sqrt{1-\vp^2(\zeta)}=\sqrt{\frac{\e(n\e-\zeta)(2-\e\zeta-n\e^2)}{(1-\e\zeta)^2}(1-\e^2)+\e^2}.
\end{eqnarray}
For $\zeta\in[0,\e]$ and $n\e$ sufficiently small, by Taylor's
expansion, we have
\begin{eqnarray}
1-\e\zeta&=&1+o(\e),\\
2-\e\zeta-n\e^2&=&2+o(\e).
\end{eqnarray}
Hence, we have
\begin{eqnarray}
\ve(\zeta)=\sqrt{\e(\e+2n\e-2\zeta)}+o(\e^2).
\end{eqnarray}
Since $\sqrt{\e(\e+2n\e-2\zeta)}=O(\e)$, we can further estimate
\begin{eqnarray}
\frac{1}{\ve(\zeta)}&=&\frac{1}{\sqrt{\e(\e+2n\e-2\zeta)}}+o(1)\\
-\int_{\k}^{n\e}\frac{\nu(1)}{\ve(\zeta)}\ud{\zeta}&=&\nu(1)\sqrt{\frac{\e+2n\e-2\zeta}{\e}}\bigg|_{\k}^{n\e}+o(\e)
=\nu(1)\bigg(1-\sqrt{\frac{\e+2n\e-2\k}{\e}}\bigg)+o(\e).
\end{eqnarray}
Then we can easily derive the integral estimate
\begin{eqnarray}
\int_0^{n\e}\ue^{-\int_{\k}^{n\e}\frac{\nu(1)}{\ve(\zeta)}\ud{\zeta}}\frac{1}{\ve(\k)}\ud{\k}&=&
\ue^{\nu(1)}\int_0^{n\e}\ue^{-\nu(1)\sqrt{\frac{\e+2n\e-2\k}{\e}}}\frac{1}{\sqrt{\e(\e+2n\e-2\k)}}\ud{\k}+o(\e)\\
&=&\half \ue^{\nu(1)}\int_{\e}^{(1+2n)\e}\ue^{-\nu(1)\sqrt{\frac{\sigma}{\e}}}\frac{1}{\sqrt{\e\sigma}}\ud{\sigma}+o(\e)\nonumber\\
&=&\half \ue^{\nu(1)}\int_{1}^{1+2n}\ue^{-\nu(1)\sqrt{\rho}}\frac{1}{\sqrt{\rho}}\ud{\rho}+o(\e)\nonumber\\
&=&\ue^{\nu(1)}\int_{1}^{\sqrt{{1+2n}}}\ue^{-\nu(1)t}\ud{t}+o(\e)\nonumber\\
&=&\frac{1}{\nu(1)}(1-\ue^{{\nu(1)}(1-\sqrt{1+2n})})+o(\e)\nonumber.
\end{eqnarray}
Then for $0<\e\leq\et_0$, we have $\abs{K[\w](0,0,1)-K[\w](\k),\ve(\k),\vp(\e)}\leq\delta$, which implies
\begin{eqnarray}
&&\int_0^{n\e}\ue^{-\int_{\k}^{n\e}\frac{{\nu(1)}}{\ve(\zeta)}\ud{\zeta}}\frac{1}{\ve(\k)}K[\w](\k,\ve(\k),\vp(\e))\ud{\k}\\
&=&
\int_0^{n\e}\ue^{-\int_{\k}^{n\e}\frac{{\nu(1)}}{\ve(\zeta)}\ud{\zeta}}\frac{1}{\ve(\k)}K[\w](0,0,1)\ud{\k}+O(\delta)+O(\e)\no\\
&=&\frac{1}{\nu(1)}(1-\ue^{{\nu(1)}(1-\sqrt{1+2n})})K[\w](0,0,1)+O(\e)+O(\delta)\nonumber.
\end{eqnarray}
For the boundary data term, since $h(\ve,\vp)$ is $C^1$, a similar
argument shows
\begin{eqnarray}
h(\e_0,\sqrt{1-\e_0^2})\ue^{-\int_0^{n\e}\frac{{\nu(1)}}{\ve(\zeta)}\ud{\zeta}}&=&\ue^{{\nu(1)}(1-\sqrt{1+2n})}h(\sqrt{1+2n}\e,\sqrt{1-(1+2n)\e^2})+O(\e).
\end{eqnarray}
Therefore, we have
\begin{eqnarray}
\w(n\e,\e)=\frac{1}{\nu(1)}(1-\ue^{{\nu(1)}(1-\sqrt{1+2n})})K[\w](0,0,1)+\ue^{{\nu(1)}(1-\sqrt{1+2n})}h(0,1)+O(\e)+O(\delta).
\end{eqnarray}
\ \\
Step 5: Estimate of Difference.\\
Collecting all above, we can estimate the value at point $\et=n\e$, $\ve=\e$ and $\vp=\sqrt{1-\e^2}$ as
\begin{eqnarray}
\\
\w(n\e,\e,\sqrt{1-\e^2})&=&\frac{1}{\nu(1)}(1-\ue^{{\nu(1)}(1-\sqrt{1+2n})})K[\w](0,0,1)+\ue^{{\nu(1)}(1-\sqrt{1+2n})}h(0,1)+O(\e)+O(\delta),\no\\
\\
\wc(n\e,\e,\sqrt{1-\e^2})&=&\frac{1}{\nu(1)}(1-\ue^{-{\nu(1)}n})K[\wc](0,0,1)+\ue^{-{\nu(1)}n}h(0,1)+O(\e)+O(\delta).\no
\end{eqnarray}
By our assumptions on $h$ and a similar argument as in the proof of Theorem \ref{counter theorem 1}, we know $\abs{K[\w](0,0,1)}<<1$ and $\abs{K[\wc](0,0,1)}<<1$.
However, $h(0,1)=1$.
Since $n$ is arbitrary and $\ue^{{\nu(1)}(1-\sqrt{1+2n})}\neq\ue^{-{\nu(1)}n}$, we always have
\begin{eqnarray}
\abs{\w(\e,n\e,\sqrt{1-n^2\e^2})-\wc(\e,n\e,\sqrt{1-n^2\e^2})}\geq C>0,
\end{eqnarray}
which further implies
\begin{eqnarray}
\lnnm{\w-\wc}{0,0}\geq C>0.
\end{eqnarray}
\end{proof}

\bibliographystyle{siam}
\bibliography{Reference}

\begin{thebibliography}{10}

\bibitem{Arkeryd.Esposito.Marra.Nouri2011}
{\sc L.~Arkeryd, R.~Esposito, R.~Marra, and A.~Nouri}, {\em Ghost effect by
  curvature in planar {Couette} flow}, Kinet. Relat. Models, 4 (2011),
  pp.~109--138.

\bibitem{Bardos.Golse.Levermore1991}
{\sc C.~Bardos, F.~Golse, and L.~D.}, {\em Fluid dynamical limits of kinetic
  equations i: formal derivations}, J. Statist. Phys., 63 (1991), pp.~323--344.

\bibitem{Bardos.Golse.Levermore1993}
\leavevmode\vrule height 2pt depth -1.6pt width 23pt, {\em Fluid dynamical
  limits of kinetic equations ii: convergence proofs for the {Boltzmann}
  equation}, Comm. Pure Appl. Math., 46 (1993), pp.~667--753.

\bibitem{Bardos.Golse.Levermore1998}
\leavevmode\vrule height 2pt depth -1.6pt width 23pt, {\em Acoustic and
  {Stokes} limits for the {Boltzmann} equation}, C. R Acad. Sci. Paris, Serie 1
  Math, 327 (1998), pp.~323--328.

\bibitem{Bardos.Golse.Levermore2000}
\leavevmode\vrule height 2pt depth -1.6pt width 23pt, {\em The acoustic limit
  for the {Boltzmann} equation}, Arch. Rational Mech. Anal., 153 (2000),
  p.~177每204.

\bibitem{Bensoussan.Lions.Papanicolaou1979}
{\sc A.~Bensoussan, J.-L. Lions, and G.~C. Papanicolaou}, {\em Boundary layers
  and homogenization of transport processes}, Publ. Res. Inst. Math. Sci., 15
  (1979), pp.~53--157.

\bibitem{Cercignani.Marra.Esposito1998}
{\sc C.~Cercignani, R.~Marra, and R.~Esposito}, {\em The {Milne} problem with a
  force term}, Transport Theory Statist. Phys., 27 (1998), pp.~1--33.

\bibitem{Esposito.Guo.Kim.Marra2013}
{\sc R.~Esposito, Y.~Guo, C.~Kim, and R.~Marra}, {\em Non-isothermal boundary
  in the {Boltzmann} theory and {Fourier} law}, Comm. Math. Phys., 323 (2013),
  pp.~177--239.

\bibitem{Esposito.Guo.Kim.Marra2015}
\leavevmode\vrule height 2pt depth -1.6pt width 23pt, {\em Stationary solutions
  to the {Boltzmann} equation in the hydrodynamic limit}, Arxiv,  (2015),
  p.~1502.05324.

\bibitem{Glassey1996}
{\sc R.~T. Glassey}, {\em The {Cauchy} problem in kinetic theory.}, Society for
  Industrial and Applied Mathematics (SIAM), Philadelphia, PA, 1996.

\bibitem{Golse2005}
{\sc F.~Golse}, {\em Hydrodynamic limits}, Eur. Math. Soc. Zurich,  (2005),
  p.~699每717.

\bibitem{Golse.Poupaud1989}
{\sc F.~Golse and F.~Poupaud}, {\em Stationary solutions of the linearized
  {Boltzmann} equation in a half-space}, Math Meth. Appl. Sciences, 11 (1989),
  p.~503每524.

\bibitem{Golse.Saint-Raymond2004}
{\sc F.~Golse and L.~Saint-Raymond}, {\em The {Navier-Stokes} limit of the
  {Boltzmann} equation for bounded collision kernels}, Invent. Math, 155
  (2004), p.~81每161.

\bibitem{Guo2010}
{\sc Y.~Guo}, {\em Decay and continuity of the {Boltzmann} equation in bounded
  domains.}, Arch. Ration. Mech. Anal., 197 (2010), pp.~713--809.

\bibitem{Masi.Esposito.Lebowitz1989}
{\sc A.~D. Masi, R.~Esposito, and J.~L. Lebowitz}, {\em Incompressible
  {Navier-Stokes} and {Euler} limits of the {Boltzmann} equation}, Comm. Pure
  and Appl. Math., 42 (1989), p.~1189每1214.

\bibitem{Sone2002}
{\sc Y.~Sone}, {\em Kinetic theory and fluid dynamics.}, Birkhauser Boston,
  Inc., Boston, MA, 2002.

\bibitem{Sone2007}
\leavevmode\vrule height 2pt depth -1.6pt width 23pt, {\em Molecular gas
  dynamics. Theory, techniques, and applications.}, Birkhauser Boston, Inc.,
  Boston, MA, 2007.

\bibitem{AA003}
{\sc L.~Wu and Y.~Guo}, {\em Geometric correction for diffusive expansion of
  steady neutron transport equation}, Comm. Math. Phys., 336 (2015),
  pp.~1473--1553.

\bibitem{Yang2012}
{\sc X.~Yang}, {\em Asymptotic behavior on the milne problem with a force
  term}, J. Differential Equations, 252 (2012), pp.~4656--4678.

\end{thebibliography}

\end{document}